\documentclass[a4paper,11pt,reqno]{amsart}
\usepackage{latexsym}
\usepackage{amsmath, amsfonts, amsthm, amssymb, color}
\usepackage{amsmath,amsthm,amsfonts,amssymb,mathrsfs}
\usepackage{graphicx,xcolor}
\usepackage[mathscr]{euscript}
\usepackage{hyperref}
\hypersetup{pdfstartview={XYZ null null 1.00}}
\allowdisplaybreaks

 \theoremstyle{definition}

 \numberwithin{equation}{section}

\def\Dv{\mathrm{div}}
\def\n{{\bf n}}
\def\proj{\mathbb{P}}

\newcommand{\LA}{\left\langle}
\newcommand{\RA}{\right\rangle}

\def\Xint#1{\mathchoice
{\XXint\displaystyle\textstyle{#1}}%
{\XXint\textstyle\scriptstyle{#1}}%
{\XXint\scriptstyle\scriptscriptstyle{#1}}%
{\XXint\scriptscriptstyle\scriptscriptstyle{#1}}%
\!\int}
\def\XXint#1#2#3{{\setbox0=\hbox{$#1{#2#3}{\int}$ }
\vcenter{\hbox{$#2#3$ }}\kern-.6\wd0}}

\def\dashint{\Xint-}
\providecommand{\norm}[1]{\left\Vert#1\right\Vert}
\newcommand{\ess}{{\rm ess\sup}}
\newtheorem{theorem}{Theorem}[section]
\newtheorem{lemma}[theorem]{Lemma}

\newtheorem{proposition}[theorem]{Proposition}

\theoremstyle{definition}
\newtheorem{definition}[theorem]{Definition}

\theoremstyle{remark}
\newtheorem{remark}[theorem]{Remark}

\author[H. Wang]{Huaqiao Wang}
\address{College of Mathematics and Statistics, Chongqing University, Chongqing
400044, China.} \email{wanghuaqiao@cqu.edu.cn}

\title[stochastic MHD equations]{Martingale solutions for the compressible MHD systems with stochastic external forces}

\keywords{ Stochastic compressible MHD systems, martingale
solutions, stochastic compactness method.}

\subjclass[2010]{35Q35, 76D05, 76W05, 60H15, 76M35}
\date{\today}
\begin{document}

\begin{abstract}
In this paper we consider the three-dimensional
compressible MHD system with stochastic external forces in a
bounded domain. We obtain the existence of martingale solution which is a weak solution for
the fluid variables, the Brownian motion on a probability space. The
construction of the solution is based on the Galerkin approximation
method, stopping time, the compactness method and Jakubowski
Skorokhod theorem, etc.
\end{abstract}

\maketitle

\section{Introduction}\label{section1}

In this paper, we consider the following three-dimensional
stochastic compressible Magneto-hydrodynamic (MHD) system:
\begin{equation}\label{MHD}
\begin{cases}
d\rho+\Dv(\rho u)dt=0,\\[-3mm]\\
d(\rho u)\! +[\Dv(\rho u\otimes u )\!+\!\nabla p-\mu\Delta u-\!(\lambda+\mu)\nabla\Dv u-\!(\nabla\times B)\times B]dt\!=\!\sum\! f_k(\rho,\rho u,x)d\beta^1_k(t),\\[-3mm]\\
dB-[\nabla\times(u\times B)+\nu\Delta B]dt=\sum
g_k(B,x)d\beta_k^2(t),\quad \Dv B=0,
\end{cases}
\end{equation}
where $\rho$ is the density, $u$ is the velocity field of the fluid,
$p$ is the scalar pressure and $p=a\rho^\gamma$ with a positive
constant $a$ and adiabatic index $\gamma\ge 1$, $B$ is the magnetic
field induced by the charged fluid, $\mu$ and $\lambda$ are two
viscosity constants satisfying $2\mu+3\lambda\ge 0$, $\nu > 0$ is
the magnetic diffusivity acting as a magnetic diffusion coefficient
of the magnetic field, and all these kinetic coefficients and the
magnetic diffusivity are independent of the magnitude and direction
of the magnetic field, the symbol $\otimes$ denotes the tensor
product, $f_k(\rho,\rho u,x)d\beta^1_k$ and $g_k(B,x)d\beta_k^2$ are
stochastic external forces, where $\beta_k^1,\beta_k^2,
k=1,2,\cdots$ are two sequences of independent one-dimensional
$\mathbb{R}$-valued Brownian motions. Denote
$\beta_k=(\beta_k^1,\beta_k^2)$.\if false $H$-valued Wiener process
with positive symmetric trace class covariance operator $Q$.\fi

System \eqref{MHD} models the interaction between a conductor fluid and the magnetic field in the presence of random perturbations. The inclusion of the stochastic terms in the governing equations is widely used to account for random fluctuations and past history of the system.

Notice that we have followed the usual convention of including the divergence-free condition on the magnetic field $B$. This is really a condition on the initial data,  as this property is preserved by the dynamics of \eqref{MHD}. To be more precise, we impose the assumptions on the initial data
\begin{equation}\label{initialdata}
\rho|_{t=0}=\rho_0,\; \rho u|_{t=0}= m_0,\; B|_{t=0}=B_0,\; \Dv
B_0=0,
\end{equation}
and the boundary conditions
\begin{equation}\label{boundary}
u|_{\partial D}=0,\quad B|_{\partial D}=0,
\end{equation}
where $D \subset \mathbb R^3$ is a bounded domain with smooth boundary $\partial D$, and $\n$ is the unit outer normal at the boundary.

The deterministic version of system \eqref{MHD} was extensively studied in the
literature. For instance, the existence of local
strong solutions with large initial data has proved by Fan-Yu
\cite{FY}; the existence of global weak and variational solution
have obtained by Ducomet-Feireisl\cite{DF}, Fan-Yu\cite{FY1},
Hu-Wang \cite{HW,HW2}; the time decay of the compressible MHD
systems with small initial data has established in
\cite{CT,LY,PG,ZW,ZXW}; the low Mach number limit problem was
studied in \cite{HW3,JJL}.

When $\rho=1$, system \eqref{MHD} reduces to the stochastic
incompressible MHD system, which has also received a lot of attention
\cite{BD-2006,SM,SP}. On the other hand, when $B=0$ and $\rho$ is not a constant, system \eqref{MHD} becomes
the stochastic compressible Navier-Stokes equations, to which, to the best of our knowledge, much less is known. Feireisl-Maslowski-Novotn\'y in \cite{FMN} considered the
three-dimensional compressible Navier-Stokes equations driven by the
stochastic external forces in Sobolev spaces and obtained the
existence of solutions by using the abstract measurability theorem
(\cite{BT-1973}) to prove that the weak solution generates a random
variable. For the existence of martingale solutions to the
three-dimensional compressible Navier-Stokes equations driven by the
multiplicative noise, see \cite{BH-2014,SSA,WW}. Recently,
Breit-Feireisl-Hofmanov\'{a} considered the low Mach number limit
problem in \cite{BFH}.

\subsection{Definition of solutions and assumptions}\label{s1.1}

We will seek a particular type of solution to system \eqref{MHD}, namely the {\it martingale solution}. The definition of such solutions is given as follows.
\begin{definition}\label{martingale-solution}
A martingale solution of \eqref{MHD}-\eqref{boundary} is a system
$((\Omega,\mathscr{F},$ $\mathscr{F}_t, \textrm{P}), \beta_k, \rho,
u, B)$
 with the following properties:

$(1)$ $(\Omega,\mathscr{F},\mathscr{F}_t,\textrm{P})$ is a
stochastic basis, where $\mathscr{F}_t$ is a filtration on the
probability space $(\Omega,\mathscr{F},\textrm{P})$,  i.e., a
nondecreasing family $\{\mathscr{F}_t: t\ge0\}$ of sub
$\sigma$-fields of
$\mathscr{F}:\mathscr{F}_s\subset\mathscr{F}_t\subset\mathscr{F}$
for $0\le s<t<\infty$;
\\
\indent $(2)$ $\{\beta_k^1\}_{k\ge1},\{\beta_k^2\}_{k\ge1}$\! are two sequences of independent
$\mathbb{R}$-valued Brownian motions; \if false an
$\mathscr{F}_t$-cylindrical Wiener process;\fi
\\
\indent $(3)$ for any $\varphi,\psi\in C_0^\infty(D)$, $\langle\rho,\varphi\rangle,
\langle\rho u,\psi\rangle$ and
$\langle B,\psi\rangle$ are progressively measurable;
\\
\indent $(4)$ for all $1\le p<\infty$, $\rho\ge 0,\; \rho\in L^p(\Omega,
L^\infty(0,T;L^\gamma(D)))$, $u\in L^p(\Omega,L^2(0,T;H_0^1(D)))$,
$B\in L^p(\Omega,L^\infty(0,T; L^2(D)))\cap
L^p(\Omega,L^2(0,T;H_0^1(D)))$ such that for all $t\in[0,T]$,
$\varphi\in C_0^\infty([0,T]\times D)$ and $\psi\in
C_0^\infty([0,T]\times D)$, it holds that $\textrm{P}$-a.s.
\begin{align*}&
\qquad\qquad\qquad\int_{D}\rho(t)\varphi dx-\int_{D} \rho_0\varphi
dx=\int_0^t\!\!\int_{D}\rho u\cdot\nabla\varphi dxdt,\\
& \int_{D}(\rho u)(t)\cdot\psi dx-\!\int_0^t\! \! \int_{D}\left[\rho
u u \cdot\nabla\psi-\mu\nabla u\cdot\nabla\psi-(\lambda+\mu)\Dv
u\cdot\Dv\psi+a\rho^\gamma\cdot\Dv\psi\right]dxdt \notag\\
&=\int_{D}\rho_0u_0\cdot\psi dx+\int_0^t\!\!\int_D(\nabla\times
B)\times B\cdot \psi dxdt
+\sum_{k=1}^\infty\int_0^t\!\!\int_{D} f_k(\rho,\rho u,x)\cdot\psi dxd\beta_k^1,\\
&\int_ D\! B(t)\cdot\psi
dx\!-\!\!\int_0^t\!\!\int_D[\nabla\times(u\times B)+\nu\Delta
B]\cdot\psi dxdt\!=\!\!\int_{D}\!B_0\cdot\psi
dx\!+\!\sum_{k=1}^\infty\! \int_0^t\!\!\int_{D}g_k(B,x)\cdot\psi
d\beta_k^2;
\end{align*}

$(5)$ the equation $\eqref{MHD}_1$ is satisfied in sense of
renormalized solutions, that is,
\begin{align}\label{r-equation}
b(\rho)_t+\Dv(b(\rho)u)+(b^\prime(\rho)\rho-b(\rho))\Dv u=0
\end{align}
 holds
in $\mathcal{D}^\prime((0,T)\times \mathbb{R}^3)$ $\textrm{P}$-a.s.
for any $b\in C^1(\mathbb{R})$ with $b^\prime(z)\equiv0$ for all
$z\in \mathbb{R}$ large enough.
\end{definition}

\noindent {\it Notations.} Throughout this paper, we drop the parameter
 $\omega\in\Omega$ for simplicity of notations.
$\{\mathscr{F}_t\}$ is a right continuous filtration over the
probability space $(\Omega,\mathscr{F},\textrm{P})$ such that
$\mathscr{F}_0$ contains all $\textrm{P}$-negligible subsets of
$\Omega$. We use ${\rm E}X=\int_{\Omega}Xd\textrm{P}$ to denote the
expectation of the random variable $X(\omega,t)$ for fixed $t$. All stochastic
integrals are defined in the sense of It\^{o}; see
\cite{PZ-1992,KS-1991}. Moreover, $C$ denotes a generic constant
which may vary in different estimates. For simplicity, we will write
$A\lesssim B$ if $A\leq C B$.

We use $C([0,T];V_w)$ to denote the subspace of
$L^\infty([0,T];V)$ consisting of those functions which are a.e.
equal to weakly continuous functions with values in $V$. We denote $L_2(U,Y)$
the set of all Hilbert-Schmidt operators from $U$ to $Y$.


\subsection*{Assumptions on the stochastic forces.} We assume that the stochastic
forces $f_k,g_k$ satisfy the following conditions:

(A)\quad $f_k(\rho,\rho u,x), g_k(B,x)$ are progressively measurable and $C^1$ continuous in
$\rho,\rho u,x$; $B,x$, respectively, and
\begin{equation}\label{assumption}
\left\{\quad\begin{aligned}
&f_k(\rho,\rho u,x)=f_{k,1}(x,\rho)+f_{k,2}(x)\rho u,\\
&\sum_{k=1}^\infty|f_{k,1}|^2\le C|\rho|^{\gamma+1},\quad
\sum_{k=1}^\infty|\partial_{\rho}f_{k,1}|^2\le
C|\rho|^{\gamma-1},\quad \sum_{k=1}^\infty|f_{k,2}|^2\le
C,\\
&\sum_{k=1}^\infty |g_k(B,x)|^2\le C|B|^2,\quad
\sum_{k=1}^\infty|\partial_B g_k|\le C.
\end{aligned}\right.\end{equation}

\medskip

%
The data $\rho_0, m_0, B_0$ are assumed to satisfy the following
conditions: 
\begin{equation}
\label{G-initialdata}
\begin{split}
&\!\!\! \rho_0\in L^{\gamma
p}(\Omega,L^\gamma(D)), \; \rho_0\ge0; \quad
m_0=0 \ \ \mbox{if} \ \ \rho_0=0, \\
{|m_0|^2\over\rho_0}&\in L^p(\Omega,L^1(D));\; B_0\in
L^{2p}(\Omega, L^2(D)),\;\Dv B_0=0.
\end{split}
\end{equation}

\if false Let $U$ be a separable Hilbert spaces, and $Q\in L(U)$ be
nonnegative definite and symmetric. Note that the $Q$-Wiener process
with finite trace has the following representation:
$$W(t)=\sum_{k\in\mathbb{N}}\beta_k e_k,\; t\in [0,T],$$
where $e_k, \beta_k$ are an orthonormal basis of $Q^{1\over
2}(U)=U_0$ and a family of independent real-valued Brownian motions,
respectively. The series converges in $L^2(\Omega; C([0,T]; U))$,
because $U_0$ defines a Hilbert-Schmidt embedding from
$(U_0,\langle\cdot,\cdot\rangle_0)$ to
$(U,\langle\cdot,\cdot\rangle)$. In the case that $Q$ is no longer
of finite trace one looses this convergence. But, it is possible to
define the Wiener process. To this end, we need a further Hilbert
space $(U_1,\langle\cdot,\cdot\rangle_1)$ and a Hilbert-Schmidt
embedding:
$J:(U_0,\langle\cdot,\cdot\rangle_0)\to(U_1,\langle\cdot,\cdot\rangle_1)
$.

Note that $(U_1,\langle\cdot,\cdot\rangle_1)$ and $J$ always exist,
for example, choose $U_1:=U$ and $\alpha_k\in [0,\infty],
k\in\mathbb{N}$ such that $\sum_{k=1}^\infty \alpha_k^2<\infty$.
Define $J: U_0\to U$ by $J(u):=\sum_{k=1}^\infty\alpha_k\langle
u,e_k\rangle_0 e_k,\; u\in U_0$. Then $J$ is one-to-one and
Hilbert-Schmidt.

Denote by $\mathscr{M}^2_T(U)$ the space of all $U$-valued
continuous, square integrable martingales $M(t), t\in [0,T]$. Let
$e_k$ is an orthonormal basis of $Q^{1\over 2}(U)=U_0$ and $\beta_k$
is a family of independent real-valued Brownian motions. Define
$Q_1=JJ^*$, where $J^*$ is the adjoint of $J$. Then $Q_1\in L(U_1)$,
$Q_1$ is nonnegative definite and symmetric with finite trace and
the series $$W(t)=\sum_{k=1}^\infty\beta_k Je_k,\; t\in [0,T],$$
converges in $\mathscr{M}^2_T(U_1)$ and defines a $Q_1$-Wiener
process on $U_1$.  Moreover, we can prove that
$Q_1^{1\over2}(U_1)=J(U_0)$ and for all $u_0\in U_0$,
$\norm{u_0}_0=\norm{Q_1^{1\over2}Ju_0}_1=\norm{Ju_0}_{Q_1^{1\over2}(U_1)}$,
i.e. $J: U_0\to Q_1^{1\over2}(U_1)$ is an isometry. We fix $Q\in
L(U)$ nonnegative, symmetric but not necessarily of finite trace.
After the above preparations, we can define the stochastic integral
with respect to a cylindrical $Q$-Wiener process $W(t), t\in[0,T]$
(Basically, with respect to the above standard $U_1$-valued
$Q_1$-Wiener process). If $\Phi\in L_2(Q_1^{1\over2}(U_1),H)$ is
predictable and
$\mathrm{P}\left(\int_0^T\norm{\Phi(s)}^2_{L_2(Q_1^{1\over2}(U_1),H)}ds<\infty\right)=1$,
then we first get that a process $\Phi(t)$ is integrable with
respect to $W(t)$. Note that $Q_1^{1\over2}(U_1)=J(U_0)$, then
$\langle Ju_0, Jv_0\rangle_{Q_1^{1\over2}(U_1)}=\langle
Q_1^{-{1\over2}}Ju_0,Q_1^{-{1\over2}}Ju_0\rangle_1=\langle
u_0,v_0\rangle_0 $ for all $u_0,v_0\in U_0$. In particular, it
follows from $Je_k, k\in\mathbb{N}$ is an orthonomal basis of
$Q_1^{1\over2}(U_1)$ that
$\norm{\Phi}^2_{L_2(Q^{1\over2}(U),H)}=\sum_{k=1}\langle \Phi e_k,
\Phi e_k\rangle=\sum_{k=1}\langle\Phi\circ J^{-1}(J e_k),\Phi\circ
J^{-1}(J e_k)\rangle=\norm{\Phi\circ
J^{-1}}^2_{L_2(Q^{1\over2}(U),H)}$ . Then $\Phi\in
L_2(Q^{1\over2}(U),H)\Leftrightarrow\Phi\circ J^{-1}\in
L_2(Q_1^{1\over2}(U_1),H)$. Now we define
$$\int_0^t \Phi(s) d W(s):=\int_0^t \Phi(s)\circ J^{-1} d W(s),\; t\in [0,T],$$
where $W(t), t\in[0,T]$ is a standard $Q$-Wiener process in $U$.
Then the class of all integrable processes is given by
$$\left\{\Phi:\Omega\times[0,T]\to L_2(Q^{1\over2}(U),H)|~ \Phi \ \text{predictable and }\ \!
\mathrm{P}\left(\int_0^T\norm{\Phi(s)}^2_{L_2(Q^{1\over2}(U),H)}ds<\infty\right)=1
\right\}.$$

Next, we will define the It\^{o} integral with respect to a
cylindrical Wiener process.
 Let $\mathscr{E}(L(U,H))$ denote the class of $L(U,H)$-valued elementary processes
adapted to the filtration $\{\mathscr{F}_t\}_{t\le T}$ with the
following form:
$$\Phi(t)=\phi 1_{\{0\}}(t)+\sum_{j=1}^{n-1}\phi_j 1_{(t_j,~ t_{j+1}]}(t),$$
where $0 =t_0\le t_1 \le \cdots\le t_n = T $, and $\phi, \phi_j$ are
respectively $\mathscr{F}_0$-measurable and
$\mathscr{F}_{t_j}$-measurable $L_2(Q^{1\over 2}(U),H)$-valued
random variables such that $\phi, \phi_j\in L(U,H)$ (recall that
$L(U,H)\subset L_2(Q^{1\over2}(U),H)$). Note that if $Q= I_U$, then
the random variables $\phi_j$ are $L_2(U,H)$-valued.

Now, we proceed with the definition of the stochastic integral with
respect to a cylindrical Wiener process. Define the It\^{o} integral
of an elementary process $\Phi\in \mathscr{E} (L(U,H))$ with respect
to a cylindrical Wiener process $\tilde{W}$ by $$\left(\int_0^T
\Phi(s)d \tilde{W}(s)\right)(h)=\sum_{j=0}^n
\left(\tilde{W}_{t_{j+1}\wedge t}( \phi^*_j(h))-\tilde{W}_{t_j\wedge
t}( \phi^*_j(h))\right),$$ for $t\in[0,T]$ and $h\in H$.

We will restrict ourselves to the case where the integrand $\Phi(s)$
is a process taking values in $L_2(U,H)$. If $\Phi(s)$ is an
elementary process, $\Phi(s)\in\mathscr{E}(L(U,H))$, then
$\Phi(s)\in L_2(U,H)$, since $Q=I_U$. Assume that $\Phi$ is bounded
in the norm of $L_2(U,H)$. It follows from $E\left(\left(\int_0^t
\Phi(s)d \tilde{W}(s)\right)(h)\right)^2=\int_0^t
E\norm{\Phi^*(s)h}^2_U ds<\infty$ that
\begin{align*}
&E\sum_{i=1}^\infty \left(\left(\int_0^t \Phi(s)d \tilde{W}(s)
\right)(e_i)\right)^2= \sum_{i=1}^\infty\int_0^t
E\norm{\Phi^*(s)e_i}^2_U
ds\\
&=E\int_0^t \sum_{i=1}^\infty\norm{\Phi^*(s)e_i}^2_U
ds=E\int_0^t \norm{\Phi^*(s)}^2_{L_2(H,U)}\\
&=E\int_0^t \norm{\Phi(s)}^2_{L_2(U,H)},
\end{align*}
where $\{e_i\}_{i=1}^\infty$ is an orthonormal basis of $H$. Then we
define the stochastic integral $\int_0^t \Psi(s)d\tilde{W}(s)$ of a
bounded elementary process $\Phi(s)$ as follows:
$$\int_0^t\Psi(s)d\tilde{W}(s):=\sum_{i=1}^\infty\left(\left(\int_0^t\Psi(s)d\tilde{W}_s\right)e_i\right)e_i.$$
By the above calculations, $\int_0^t \Psi(s)d\tilde{W}(s)\in
L^2(\Omega,H)$ and is adapted to the filtration $\mathscr{F}_t$. The
equality
$$\norm{\int_0^t\Psi(s)d\tilde{W}(s)}_{L^2(\Omega,H)}=E\int_0^t \norm{\Phi(s)}^2_{L_2(U,H)}ds$$
establishes the isometry property of the stochastic integral
transformation.
\fi
 Denote
 \[
 F(\rho,\rho u)dW_1:=\sum_{k=1}^\infty
f_k(\rho,\rho u,x)d\beta_k^1, \quad  G(B)dW_2:=\sum_{k=1}^\infty
g_k(B,x)d\beta_k^2.
\]
By \eqref{assumption} and H\"{o}lder's inequality, we have for
$\ell> \frac{3}{2}$,
\begin{align*}
\norm{F(\rho,\rho u)}^2_{L_2(U,
W^{-\ell,2}(D))} & =\sum_{k=1}^\infty\norm{f_k(\rho,\rho
u,x)}^2_{W^{-\ell,2}(D)}\lesssim \sum_{k=1}^\infty\norm{f_k(\rho,\rho
u,x)}^2_{L^1(D)}\\
&\lesssim \left(\int_D \sqrt{\rho}\sum_{k=1}^\infty
{f_{k,1}(\rho,x) \over \sqrt{\rho}}+\sqrt{\rho}\sqrt{\rho}
u\sum_{k=1}^\infty
f_{k,2}(x)dx\right)^2\\
&\lesssim \norm{\rho}_{L^1}(\norm{\rho}^{\gamma}_{L^\gamma(D)}+\norm{\sqrt{\rho}u}^2_{L^2(D)})\lesssim 1,\\
\norm{G(B)}^2_{L_2(U,
W^{-1,2}(D))} & =\sum_{k=1}^\infty\norm{g_k(B,x)}^2_{W^{-1,2}(D)}\le\sum_{k=1}^\infty\norm{g_k(B,x)}^2_{L^2(D)}\\
&\lesssim \norm{B}^2_{L^2(D)}\lesssim 1.
\end{align*}
Then the It\^{o}'s integrals $\int_0^T F(\rho,\rho u)dW_1$ and
$\int_0^T G(B)dW_2$ are well-defined martingales in $W^{-\ell,2}(D)$
and $W^{-1,2}(D)$ respectively. Since $W_1$ and $W_2$ don't converge
on $U$, we define $U_0\supset U$ by
$U_0=\{v=\sum_{k\ge1}\alpha_ke_k:\ \sum_{k\ge 1}{{\alpha_k^2}\over
k^2}<\infty \}$ with the norm 
$\|v\|^2_{U_0}=\sum_{k\ge 1}{\alpha_k^2}/k^2$, $v=\sum_{k\ge
1}\alpha_ke_k$. Note that the embedding $U\hookrightarrow U_0$ is
Hilbert-Schmidt and $\{e_k\}_{k\ge 1}$ is a complete
 orthonormal basis of $U$. Moreover,  we have $W_1,W_2\in C([0,T],U_0)\ \ \textrm{P}-\text{a.s.}$.


\subsection{Main results}\label{s1.2}

Our main result is stated in the following theorem:

\begin{theorem}\label{theorem1.1}
Suppose that the initial data $(\rho_0, m_0,B_0)$ satisfies
\eqref{G-initialdata} and \eqref{assumption}
 holds. Let $\gamma>3/2$. If
$D\subset\mathbb{R}^3$ is a bounded domain of class $C^{2+\alpha}$ with $
\alpha>0$, then there exists a martingale solution
$\left((\Omega,\mathscr{F},\textrm{P}),\mathscr{F}_t,\beta_k,\rho,u,B\right)$
in the sense of Definition \ref{martingale-solution} with
\begin{align*}
{\rm E }\Big [ \sup\limits_{0\le t\le
T}\mathscr{E}(t)+\int_0^T \int_D \left(\mu|\nabla
u|^2+(\lambda+\mu)|\Dv u|^2+\nu |\nabla B|^2\right) dxdt \Big ]^p
\lesssim 1
\end{align*}
to the problem \eqref{MHD}-\eqref{boundary} for any given $T>0$ and
for all $1\le p<\infty$.
\end{theorem}

The basic idea to prove Theorem \ref{theorem1.1} is to first construct the
approximate schemes to the problem \eqref{MHD}-\eqref{boundary}, and then establish tightness estimates of the approximate solutions and
pass to the limits in the spirit of Feireisl \cite{f2,FNP}. Of
course, the apparent challenge stems from the stochastic term and the It\^o
integral, which is the primary focus of the work. 

More precisely, as a first step we use Faedo-Galerkin method to construct the the approximate system involving the artificial viscosity
and artificial pressure, and apply the standard fixed
point argument to obtain an $n$-dimensional truncated local solution
$(\hat{\beta}_k,\rho_n,u_n,B_n)$ in the time interval $[0,\tau_{n,N})$. The existence of solutions in the whole time $[0,T]$ will be achieved by the energy estimates. Additional challenges arise due to the probabilistic nature of the system. Our method relies on a delicate stopping time argument with the use of the Burkholder-Davis-Gundy inequality. In particular, this approach leads to the convergence of the approximate solutions on $[0,T]$.

However the convergence is too weak to guarantee that the limit is a
solution on $[0,T]$. In the two-dimensional case, it can be shown by
using certain monotonicity principle that the nonlinear terms
converge to the right limit and hence a global strong solution can
be obtained \cite{MS-2002}. But when the space dimension is three
the monotonicity does not hold and to the best of our knowledge
there is no result on the global strong solutions. This is why we
pursue instead the Martingale solutions. As is explained, the main
issue is the convergence of the nonlinear terms. To this end, we
need to derive the tightness and show the convergence as
$n\to\infty$ of the approximate finite-dimensional solutions in the
Faedo-Galerkin method in three steps: (1) We apply the deterministic
compactness criterion and Arzela-Ascoli's Theorem to get the
tightness property of the approximate solution
$(\hat{\beta}_k,\rho_n,u_n,B_n)$. Moreover in order to analyze the
nonlinear term, we prove the tightness of $\rho_nu_n$ as well. Then
from the Jakubowski-Skorokhod Theorem,  there exist a probability
space $(\tilde{\Omega}, \tilde{\mathscr{F}},\tilde{\textrm{P}})$ and
random variables
$(\tilde{\beta}_{k,n},\tilde{\rho}_{n},\tilde{u}_{n},
\tilde{\rho}_{n}\tilde{u}_{n},\tilde{B}_{n})\to
(\beta_k,\rho,u,h,B)\ \ \tilde{\textrm{P}}-\text{a.s.}$,( Denote
$m_n=\rho_n u_n$, since they have same distribution, that is, for
$A\in\mathscr{F}$, $\tilde{\textrm{P}}(\tilde{m}_{n}\in
A)=\tilde{\textrm{P}}(\rho_n u_n\in
A)=\tilde{\textrm{P}}(\tilde{\rho}_{n}\tilde{u}_{n}\in A)$, we have
$\tilde{m}_{n}=\tilde{\rho}_{n}\tilde{u}_{n}$ almost surely) with
the the property that the probability distribution of
$(\tilde{\beta}_{k,n},\tilde{\rho}_{n},
\tilde{u}_{n},\tilde{\rho}_{n}\tilde{u}_{n},\tilde{B}_{n})$ is the
same as that of $(\hat{\beta}_k,\rho_n,u_n,\rho_n u_n,B_n)$. By
using a cut-off function we can also show that the random variable
$(\tilde{\beta}_{k,n},\tilde{\rho}_{n},\tilde{u}_{n},
\tilde{\rho}_{n}\tilde{u}_{n},\tilde{B}_{n})$ satisfy the
approximate equations in $(\tilde{\Omega},
\tilde{\mathscr{F}},\tilde{\textrm{P}})$; (2) In view of the
uniformly integrable criterion and Vitali's convergence Theorem,
together with the almost sure convergence on $(\tilde{\Omega},
\tilde{\mathscr{F}},\tilde{\textrm{P}})$, we can obtain the limit
process of the stochastic parts, it's quadratic variation and cross
variation are martingale; (3) Finally, by the new method in
\cite{BO, OM} on the martingale problem of Stroock and Varadhan, we
can prove that the system $((\tilde{\Omega},
\tilde{\mathscr{F}},\tilde{\textrm{P}}),\beta_k,\rho,u,B)$ is a
martingale solution of the approximation scheme. After obtaining the
foregoing solutions to the approximate scheme by the Faedo-Galerkin
method, we shall follow the idea of \cite{FNP, LLP} to deal with the
artificial viscosity and artificial pressure using the similar
convergence framework   for $n\to\infty$ in the Faedo-Galerkin
solutions. Due to insufficient integrability of the density, the
effective viscous flux and the cut-off function technique will be
used to obtain the strong convergence of the density. Then we pass
to the limits of vanishing artificial viscosity and vanishing
artificial pressure, and the limit of the solutions to the
approximate scheme is a martingale solution with finite energy to
the initial-boundary value problem of \eqref{MHD}-\eqref{boundary}.
Note that here we shall use the property of martingale.

 The rest of the paper is organized as follows.
  We shall recall  some basic theory of stochastic analysis in
Section \ref{section2}.  In Section \ref{section3}, we shall construct the
 solutions to an approximate scheme by the Faedo-Galerkin method.
In Section \ref{section4}, we shall pass to the limit as
 the artificial viscosity goes to zero,
and in Section \ref{section5}, we shall pass to the
 limit as the artificial pressure goes to zero.

\section{Preliminaries}\label{section2}

\setcounter{equation}{0} In this section, we first introduce some
function spaces. Let $D\subset \mathbb{R}^3$ be an open subset with
smooth boundary $\partial D$. Let $L^p(D) \, (1\le p<\infty)$ denote
the Banach space of Lebesgue measurable $\mathbb{R}^3$ valued
integrable functions on the set $D$ with the standard norm
$\|u\|_{L^p(D)}$, and  $L^{\infty}(D)$  the Banach space of Lebesgue
measurable essentially bounded $\mathbb{R}^3$ valued functions
defined on $D$ with the standard norm $\|u\|_{L^{\infty}(D)}$.
 If $p=2$, then $L^2(D)$ is a Hilbert space with the scalar product
denoted by
\begin{align*}
\langle u,v\rangle_{L^2(D)}=\int_{D}u(x)v(x)dx, \ \  \mbox{for} \ \
u, v\in L^2(D).
\end{align*}
Let $H^1(D)$ stand for the Sobolev space of all $u\in L^2(D)$ for
which there exist weak derivatives ${\partial u\over \partial
x_i}\in L^2(D), i=1,2,3$. It is a Hilbert space with the scalar
product denoted by
\begin{align*}
\langle u,v\rangle_{H^1(D)}=\langle
u,v\rangle_{L^2(D)}+\langle\nabla u,\nabla v\rangle_{L^2(D)}, \ \
\mbox{for} \ \ u,v\in H^1(D).
\end{align*}

For a probability space $(\Omega,\mathscr{F},\textrm{P})$ and a
Banach space $X$, denote by  $L^p(\Omega,L^q(0,T;X))\, (1\le
p,q<\infty)$ the space  of random functions defined on $\Omega$ with
value in $L^q(0,T; X)$,  endowed with the norm:
\begin{align*}
\|u\|_{L^p(\Omega,L^q(0,T;X))}= \big ({\rm E}\|u\|^p_{L^q(0,T;X)} \big
)^{1\over p}.
\end{align*}
If $q=\infty$, we write
\begin{align*}
\|u\|_{L^p(\Omega,L^\infty(0,T;L^q(X)))}= \big({\rm E}~\ess_{0\le t\le
T}\|u\|^p_{L^q(X)} \big)^{1\over p}.
\end{align*}

We now recall some preliminaries of stochastic analysis and useful
tools for the sake of convenience and completeness. For details, we
refer the reader to \cite{PZ-1992,KS-1991} and the references
therein.
\if false
\begin{definition}\label{Def2.1}
A stochastic process $W(t,\omega)$ is called a {\em Brownian motion}
if it satisfies the following conditions:

$(1)$ $P\{\omega: W(0,\omega)=0\}=1.$

$(2)$ For any $0\le s\le t$, the random variable $W(t)-W(s)$ is
normally distributed with mean 0 and variance $t-s$, i.e., for any
$a<b$,
\begin{equation*}
P\left\{a\le W(t)-W(s)\le b\right\}={1\over\sqrt{2\pi(t-s)}}\int_a^b
e^{-x^2\over2(t-s)}dx.
\end{equation*}

$(3)$ $W(t,\omega)$ has independent increments, i.e., for any $0\le
t_1<t_2<\ldots<t_n$, the random variables $W(t_1),
W(t_2)-W(t_1),\ldots, W(t_n)-W(t_{n-1})$ are independent.

$(4)$ Almost all simple paths of $W(t,\omega)$ are continuous
functions, i.e.,
\begin{equation*}
P\left\{\omega: W(\cdot,\omega) \mbox{ is continuous}\right\}=1.
\end{equation*}
\end{definition}
\fi
\if false
\begin{definition}\label{Cy-wienerprocess}
Let $H$ be a Hilbert space. A stochastic process $\{W(t)\}_{0\le
t\le T}$ is said to be an $H$-valued $\mathscr{F}_t$-adapted Wiener
process with covariance operator $Q$ if

(1) For each non-zero $h\in H$, $|Q^{1/2}h|^{-1}\langle
W(t),h\rangle$ is a standard one-dimensional Wiener process,

(2) For any $h\in H$, $\langle W(t),h\rangle$ is a martingale
adapted to $\mathscr{F}_t$.
\end{definition}

If $W$ is an $H$-valued Wiener processes with covariance operator
$Q$ with $TrQ<\infty$, then $W$ is a Gaussian process on $H$ and
$E(W(t))=0, Cov(W(t))=tQ, t\ge0$. Let $H_0=Q^{1/2}H$, then $H_0$ is
a Hilbert space equipped with the linear product
$\langle\cdot,\cdot\rangle_0$, $\langle u,v\rangle_0=\langle
Q^{-1/2}u,Q^{-1/2}v\rangle,\forall u,v\in H_0 $, where $Q^{-1/2}$ is
the pseudo-inverse of $Q^{1/2}$. Since $Q$ is a trace class
operator, the imbedding of $H_0$ in $H$ is Hilbert-Schmidt. Let
$L_Q$ denote the space of linear operators $S$ such that $SQ^{1/2}$
is a Hilbert-Schmidt operator from $H$ to $H$.
\begin{definition}\label{h-wienerprocess}
Let $H$ be a Hilbert space with inner product $\langle
\cdot,\cdot\rangle$. An \textcolor{red}{$H$-cylindrical Wiener
process} on $[0,T]$ is a family $(W_H(t))_{[0,T]}$ of mappings from
$H$ to $L^2(\Omega)$ with the following properties:

(1) $(W_H(t)h)_{[0,T]}$ is a Brownian motion for all $h\in H$;

(2) for all $0\le t_1, t_2\le T$ and $h_1, h_2\in H$,
$E(W_H(t_1)h_1\cdot W_H(t_2)h_2)=\min\{t_1,t_2\}\langle
h_1,h_2\rangle$.
\end{definition}
\fi

Let us first recall the It\^{o} formula for a local martingale.
\begin{lemma}\label{Ito-formula}
Suppose that $M_t=(M_t^1,M_t^2,\ldots,M_t^n)$ is a vector valued
continuous local martingale, that is $(M_t^i,\mathscr{F}_t)$ is a
local martingale for each $i=1,2,\ldots,n$ and $t\in\mathbb{R_+}$.
Let $A_t=(A_t^1,A_t^2,\ldots,A_t^n)$ be a vector valued continuous
processes adapted to the same filtration $\{\mathscr{F}_t\}$ such
that the total variation of $A_t^i$ on each finite interval is
bounded almost surely, and $A_0^i=0$ almost surely. Let
$X_t=(X_t^1,X_t^2,\ldots,X_t^n)$ be a vector valued adapted
processes such that $X_t=X_0+M_t+A_t$, and let $\Phi\in
C^{1,2}(\mathbb{R}_+\times\mathbb{R}^n)$. Then, for any $t\ge0$, the
equality
\begin{align*}
\Phi(t,X_t)&=\Phi(0,X_0)+\sum_{i=1}^n\int_0^t {\partial\over\partial
x_i}\Phi(s,X_s)dM_s^i+\sum_{i=1}^n\int_0^t {\partial\over\partial
x_i}\Phi(s,X_s)dA_s^i\\
&\quad+{1\over2}\sum_{i,j=1}^n\int_0^t {\partial^2\over\partial
x_ix_j}\Phi(s,X_s)d\langle
M^i,M^j\rangle_s+\int_0^t{\partial\over\partial s}\Phi(s,X_s)ds
 \end{align*}
holds almost surely. Here $\langle \cdot,\cdot\rangle_t$ is the
cross variation process defined by $\langle X,Y\rangle_t={1\over
4}\{\langle X+Y\rangle_t-\langle X-Y\rangle_t\}$, and $\langle
X\rangle_t$ denotes the quadratic variation of $X$ on $[0,t]$.
\end{lemma}

\section{The Faedo-Galerkin approximation and a priori estimates}\label{section3}

\subsection{Approximation scheme and a priori estimates}\label{s3.1}
For fixed $\varepsilon, \delta > 0$, we consider the following
approximation problem with $\beta>\max\{4,\gamma\}$:
\begin{align}\label{A-MHD}
\begin{cases}
d\rho+\Dv(\rho u)dt=\varepsilon\Delta\rho dt\\
d(\rho u)+[\Dv(\rho u\otimes
u)+a\nabla\rho^\gamma+\delta\nabla\rho^\beta+\varepsilon\nabla
u\cdot \nabla\rho-(\lambda+\mu)
\nabla\Dv u-\mu\Delta u]dt\\
\qquad\qquad\qquad\qquad=(\nabla\times B)\times Bdt+\sum_{k\ge1} f_k(\rho,\rho u,x)d\beta^1_k(t)\\
dB-[\nabla\times(u\times B)+\nu\Delta B]dt=\sum_{k\ge1}
g_k(B,x)d\beta_k^2(t),
\end{cases}
\end{align}
with the boundary conditions:
\begin{align}\label{A-boundary}
\nabla\rho\cdot \n|_{\partial D}=0, \quad u|_{\partial D}=0, \quad
B|_{\partial D}=0
\end{align}
and the initial data:
\begin{align}\label{A-initialdata}
& \rho|_{t=0}=\rho_{0,\delta}\in C^{2+\alpha}(\overline{D}),\quad
 \nabla \rho_{0,\delta}\cdot \n|_{\partial D}=0,\\
&\notag (\rho u)|_{t=0}=m_{0,\delta}\in C^2(\bar{D}),\quad
B|_{t=0}=B_{0,\delta}\in C^2(\bar{D}),\;\Dv B_{0,\delta}=0.
\end{align}
Here we take the initial data that satisfies the following conditions:
\begin{equation*}\label{C-initialdata}
\begin{split}
& \rho_{0,\delta}\to\rho_0   \mbox{ in }   L^\gamma(D),\quad\quad\
B_{0,\delta}\to B_0 \mbox{ in } L^2(D)   \mbox{ as }
  \delta\to 0; \\
& 0<\delta\le\rho_{0,\delta}(x)\le \delta^{-{1\over\beta}},\quad
m_{0,\delta}=h_{\delta}\sqrt{\rho_{0,\delta}},
\end{split}
\end{equation*}
almost surely and $h_{\delta}$ is defined as follows. First, take
\begin{align*}
\tilde{m}_{0,\delta}(x)=
\begin{cases}
m_0(x)\sqrt{{\rho_{0,\delta}(x)\over \rho_{0}(x)}}, \ \ &\mbox{if} \
\
\rho_0(x)>0,\\
0, \ \ &\mbox{if} \ \ \rho_0(x)=0.
\end{cases}
\end{align*}
By using \eqref{G-initialdata}, we have
\begin{align*}
{|\tilde{m}_{0,\delta}|^2\over{\rho_{0,\delta}}} \ \ \mbox{is
bounded in} \ \ L^p(\Omega, L^1(D)) \ \ \mbox{independently of} \ \
\delta>0, \ \forall \ p\in[1,\infty).
\end{align*}
Since $C^2(\overline{D})$ is dense in $L^2(D)$, then we can find
$h_\delta\in C^2(\overline{D})$ such that
\begin{align*}
 \Big \|{\tilde{m}_{0,\delta}\over\sqrt{\rho_{0,\delta}}}-h_\delta \Big \|_{L^p(\Omega,L^2(D))}<\delta.
\end{align*}
On the probability
space $(\Omega,\mathscr{F},\textrm{P})$ with two given Brownian
motions $\hat{\beta}_k^1$ and $\hat{\beta}_k^2, \ k\ge1$, denote
$\hat{\beta}_{k}=(\hat{\beta}_{k}^1,\hat{\beta}_{k}^2)_{k\ge1}$.

In order to solve \eqref{A-MHD}-\eqref{A-initialdata}, we first consider
a suitable orthogonal system formed by a family of smooth functions
$\varphi_n$ vanishing on $\partial D$. One can take the
eigenfunctions of the Dirichlet problem for the Laplacian operator:
\begin{equation*}
-\Delta \varphi_n=\lambda_n\varphi_n \ \ \mbox{on} \ \ D,\quad
\varphi_n|_{\partial D}=0.
\end{equation*}
Now, we consider a sequence of finite dimensional spaces
\begin{equation*}
X_n=\mbox{span}\{\varphi_j\}^n_{j=1},\; n=1,2,\ldots.
\end{equation*}
First, we introduce a family of operators
\begin{align*}
\mathcal{M}[\rho]: X_n \to X_n; \quad  \langle \mathcal{M}[\rho] v,w
\rangle=\int_D \rho v\cdot wdx, \; \forall\; v, w\in X_n.
\end{align*}
If $\inf_{x\in D}\rho>0$, then these operators are invertible and we
have
\begin{align*}
\|\mathcal{M}^{-1}[\rho]\|_{\mathcal {L}(X_n,X_n)}\le
  \Big (\inf_{x\in D}\rho(x)  \Big )^{-1},
\end{align*}
where $\mathcal{M}^{-1}[\rho]$ is the inverse of $\mathcal{M}[\rho]$. Moreover, by using the following identity
\begin{align*}
\mathcal{M}^{-1}[\rho_1]-\mathcal{M}^{-1}[\rho_2]=\mathcal{M}^{-1}[\rho_2]
(\mathcal{M}[\rho_2]-\mathcal{M}[\rho_1])\mathcal{M}^{-1}[\rho_1],
\end{align*}
we can deduce that the map
\begin{align*}
\rho\mapsto\mathcal{M}^{-1}[\rho] \ \ \mbox{mapping}\ \ L^1(D) \ \
\mbox{into} \ \ \mathcal {L}(X_n,X_n)
\end{align*}
is well-defined and satisfies
\begin{align}\label{NS13}
\|\mathcal{M}^{-1}[\rho_1]-\mathcal{M}^{-1}[\rho_2]\|_{\mathcal
{L}(X_n,X_n)}\le C(n,\eta)\|\rho_1-\rho_2\|_{L^1(D)}
\end{align}
for any $\rho_1,\rho_2$ belonging to the set
\begin{align*}
N_\eta=  \Big \{\rho\in L^1:  \inf_{x\in D}\rho\ge\eta>0  \Big \}.
\end{align*}

Now we will try to solve \eqref{A-MHD} on $X_n$. For this,
let $\mathbb{P}$ be the projection from $L^2(D)$ to $X_n$.
$\mathbb{P}$ is also a linear projection from $ H^{\ell}(D)$ to
$X_n$. In fact, $$\mathbb{P} h=\sum_{i=1}^{n}  \langle h , e_i\rangle_{{H^\ell-H^{-\ell}}} e_i, \quad h \in H^{\ell
}(D),e_i\in X_n.$$
We shall look for the sequence of pairs
$(u_n,B_n)\in C([0,T]; X_n)$ satisfying the integral equations:
\begin{equation*}
\begin{split}
&\int_D  \rho(t)u_n(t)\cdot\varphi
dx\!+\int_0^t\!\!\int_D\left[\Dv(\rho u_n\otimes u_n)-\mu\Delta
u_n-(\nabla\times B_n)\times B_n
\right]\cdot\varphi dxds\\
&=\!\int_D m_0\cdot\varphi dx+\!\!\int_0^t\!\!\!\int_D\!\!
\left[\nabla\!\left((\lambda+\mu)\Dv
u_n\!\!-\!a\rho^\gamma\!-\!\delta\rho^\beta\right)\!-\!\varepsilon\nabla
u_n\cdot\nabla\rho\right] \cdot\varphi
dxds\!\\
&\quad+\!\!\int_0^t  \langle \sum_{k\ge1}f^n_k(\rho,\rho u_n,x), \ \varphi \rangle d\hat{\beta}_k^1,
\end{split}
\end{equation*}
\begin{equation*}
\begin{split}
\int_D  B_n(t)\cdot\varphi dx-\int_D B_0\cdot\varphi
dx&=\int_0^t\int_D \left[\nabla\times(u_n\times B_n)+\nu\Delta B_n\right]\cdot\varphi dxds\\
&\quad+\int_0^t\int_D\sum_{k\ge1} g_k(B_n,x) \cdot\varphi
dxd\hat{\beta}_k^2,
\end{split}
\end{equation*}
for all $t\in[0,T]$ and any function $\varphi\in X_n$, where
$$ \sum_{k\ge1}f^n_k(\rho,\rho u_n,x) := \mathcal{M}^{\frac{1}{2}}[\rho]\mathbb{P}(\sum_{k\ge1}f_k(\rho,\rho
u_n,x)/\sqrt{\rho}), \quad \text{with }\  \langle \mathcal{M}^{{1\over2}}[\rho] v,w
\rangle=\int_D \sqrt\rho v\cdot wdx.$$
Then, we can rewrite the above integral equations as 
\begin{align}\label{r-integral}
\notag(u_n(t),B_n(t))= & \left(\mathcal{M}^{-1}[\rho(t)] \left[m_0^*+\!\int_0^t\mathscr{N}_1[\rho(s),u_n(s),B_n(s)]ds
+\!\int_0^t \sum_{k\ge1}f^n_k(\rho,\rho u_n,x)d\hat{\beta}_k^1 \right], \right.\\
 &\quad  \left. B_0^*+\int_0^t\mathscr{N}_2[u_n(s),B_n(s)]ds+\int_0^t\sum_{k\ge1}g_k(B_n,x)d\hat{\beta}_k^2 \right),
\end{align}
where
\begin{align*}
&\langle m_0^*,\varphi\rangle=\int_D m_0\cdot\varphi dx,\qquad \langle B_0^*,\varphi\rangle=\int_D B_0\cdot\varphi dx,\\
&\langle \mathscr{N}_1[\rho,u_n,B_n],\varphi\rangle=\int_D \big[\mu\Delta u_n-\Dv(\rho u_n\otimes u_n)\\
&\quad+\nabla\left((\lambda+\mu)\Dv
u_n-a\rho^\gamma-\delta\rho^\beta\right)-\varepsilon\nabla
u_n\cdot\nabla\rho+(\nabla\times B_n)\times B_n\big]
\cdot\varphi dx,\\
&\langle \mathscr{N}_2[u_n,B_n],\varphi\rangle=\int_D
\left[\nabla\times(u_n\times B_n)+\nu\Delta B_n\right]\cdot\varphi
dx.
\end{align*}

We will take the following steps to solve \eqref{r-integral}.

\medskip

\noindent{\bf Step 1:} Solve $\rho$ in terms of $u_n$ and derive a fixed-point problem.
Note that $\rho$ is determined as the solution of
the following Neumann initial-boundary value problem(for e.g., see Lemma
2.2 \cite{f2}):
\begin{equation}\label{Neumann problem}
\begin{cases}
\rho_t+\Dv(\rho u)=\varepsilon \Delta \rho,\\
\nabla  \rho\cdot\n|_{\partial D}=0,\\
\rho|_{t=0}=\rho_{0,\delta}(x).
\end{cases}
\end{equation}
From \cite[Lemma 2.1, Lemma 2.2]{FNP} we know
that there exists a mapping $\mathcal{S}: C([0,T];
C^2(\overline{D})) \to C([0,T]; C^{2+\alpha}(\overline{D}))$ such
that
\begin{enumerate}
\item[(1)] $\rho=\mathcal{S}[u]$ is the unique classical solution of
\eqref{Neumann problem};
\item[(2)] $\underline{\rho}\exp\left(-\int_0^t \|\Dv
u(s)\|_{L^\infty(D)}ds\right)\le \mathcal{S}[u](t,x)\le
\overline{\rho}\exp\left(-\int_0^t \|\Dv
u(s)\|_{L^\infty(D)}ds\right)$ for all $t\in[0,T]$, where $0 < \underline{\rho} \le \rho_{0, \delta} \le \overline{\rho}$;
\item[(3)] $\|\mathcal{S}[u_1]-\mathcal{S}[u_2]\|_{C([0,T]; W^{1,2}(D))}\le
T C(K, T)\|u_1-u_2\|_{C([0,T]; W^{1,2}(D))}$ for any $u_1, u_2$
in  $M_K :=\left\{u\in C([0,T]; W^{1,2}(D)):
\|u(t)\|_{L^\infty(D)}\!+\!\|\nabla u(t)\|_{L^\infty(D)}\le K\ \
\!\! \text{for all $t$} \right\}$.
\end{enumerate}

Thus on $X_n$ we can write $\rho=\mathcal{S}[u_n]$, and construct the
approximate solutions for \eqref{A-MHD}-\eqref{A-initialdata} by
means of \eqref{r-integral}. That is, we can get the approximate
solutions by solving the following integral equations:
\begin{align*}
&(u_n(t),B_n(t))\\
&=\left(\mathcal{M}^{-1}[\mathcal{S}[u_n](t)]\left[ m_0^*+\!\int_0^t\mathscr{N}_1[\mathcal{S}[u_n](s),u_n(s),B_n(s)]ds+\!\int_0^t\sum_{k\ge1}
f^n_k(\mathcal{S}[u_n],\mathcal{S}[u_n] u_n,x)d\hat{\beta}_k^1(s) \right], \right.\\
&\qquad \left. B_0^*+\int_0^t\mathscr{N}_2[u_n(s),B_n(s)]ds+\int_0^t
\sum_{k\ge1}g_k( B_n,x)d\hat{\beta}_k^2(s) \right).
\end{align*}


\noindent{\bf Step 2.} Solve a cut-off problem. For $N>0$, choose a $C^\infty$ smooth cut-off function
~$\theta_N: [0,\infty)\to[0,1]$ such that
\begin{eqnarray*}
\theta_N(x) :=
\begin{cases}
1, \ \ \mbox{for}\ \ |x|\le N,\\
0, \ \ \mbox{for}\ \  |x|\ge N+1.
\end{cases}
\end{eqnarray*}

Note that we consider the following cut-off problem of
\eqref{r-integral} for a fixed $n $:
\begin{align}\label{integral3a}
\notag(u^N_n(t),B^N_n(t))&=\!\!\left(\mathcal{M}^{-1}[\mathcal{S}[u^N_n](t)]\left[ m_0^*\!+\!\!\int_0^t\!\!\!\theta^{u^N_n, B^N_n}_N(s)\mathscr{N}_1[\mathcal{S}[u^N_n](s),u^N_n(s),B^N_n(s)]ds \right.\right.\\
&\quad+ \left. \int_0^t\sum_{k\ge1}\theta^{u^N_n, B^N_n}_N(s)
f^n_k(\mathcal{S}[u^N_n],\mathcal{S}[u^N_n]
u^N_n,x)d\hat{\beta}_k^1(s) \right],\\
\notag&B_0^*+\int_0^t\theta^{u^N_n,
B^N_n}_N(s)\mathscr{N}_2[u^N_n(s),B^N_n(s)]ds+\int_0^t
\sum_{k\ge1}\theta^{u^N_n, B^N_n}_N(s) g_k(
B^N_n,x)d\hat{\beta}_k^2(s) \bigg),
\end{align}
where $\theta^{u^N_n, B^N_n}_N(s) = \theta_N\left(\max\{
\|u^N_n(s)\|_{W^{1,\infty}}, \|B^N_n(s)\|_{W^{1,\infty}} \}\right)$.
Let $N$ and $n$ be fixed. By means of the standard fixed point
argument on the Banach space $C([0,T];X_n)$, we can solve the
integral equations \eqref{integral3a}, at least on a short time
interval $[0,T_{n,N}]$, $T_{n,N}\le T$. For simplicity, we denote
$(u^N,B^N):=(u^N_n, B^N_n)$.
\begin{proposition}\label{prop_extN}
Given a $T > 0$, for each fixed $n$ and $N$, there exists a $T_{n,N}
\in (0, T]$ such that the equation \eqref{integral3a} admits a
unique solution $(u^N, B^N) \in L^2(\Omega, C([0, T_{n,N}];
X_n))^2$.
\end{proposition}
\begin{proof}
Define
\begin{align*}
B_{N,T_{n,N}}=\left\{U^N:=(u^N,h^N) \in L^2(\Omega, C(I_{n,N};
X_n))^2: ~\|(u^N,h^N)\|_{C(I_{n,N};W^{1,\infty}(D))}\le N\right\},
\end{align*}
with the norm
$\|U^N\|^2_{B_{N,T_{n,N}}}\!\!={\rm E}\sup_{I_{n,N}}\|U^N\|^2_{X_n}$,
where $I_{n,N}=[0,T_{n,N}]$ and $\|U^N\|^2_{X_n} = \|u^N\|^2_{X_n} +
\|h^N\|^2_{X_n}$.
Introduce a map
\begin{align*}
\mathscr{T}: L^2(\Omega, C(I_{n,N}; X_n))^2 \to L^2(\Omega, C(I_{n,N}; X_n))^2,
\end{align*}
defined by
\begin{align*}
\mathscr{T}(U^N) & := q_0^*+\int_0^t\theta^{U^N}_N(s) \mathscr{N}\
ds +\int_0^t\sum_{k\ge1}\theta^{U^N}_N(s)F_k\ d\hat{\beta}_k(s),
\end{align*}
where
\begin{align*}
q_0^* & =(\mathcal{M}^{-1}[\mathcal{S}(u^N)(t)] m_0^*, \ B_0^*),\\
\mathscr{N} & =(\mathcal{M}^{-1}[\mathcal{S}(u^N)(t)]\mathscr{N}_1[\mathcal{S}(u^N),u^N,h^N], \ \mathscr{N}_2[u^N,h^N]),\\
F_k & = (\mathcal{M}^{-1}[\mathcal{S}(u^N)(t)]
f^n_k(\mathcal{S}(u^N), \ \mathcal{S}(u^N) u^N,x), g_k(h^N,x)).
\end{align*}
Similar to \cite{NS}, by using the properties of
$\mathcal{M}[\rho]$, $\mathcal{S}[u]$, H\"{o}lder's inequality and
the Burkholder-Davis-Gundy inequality, we can infer that
$\mathscr{T}$ maps $B_{N,T_{n,N}}$ into itself for some suitable
$T_{n,N}$.

We will prove that $\mathscr{T}$ is a contraction on $L^2(\Omega, C(I_{n,N}; X_n))^2$ for some (small) $T_{n,N} > 0$.
Since the deterministic parts are similar to \cite[Section 7]{NS},
we only need to consider the stochastic parts
\[
\mathscr{T}^s(U^N) := \int_0^t\sum_{k\ge1}\theta^{U^N}_N(s)F_k\
d\hat{\beta}_k(s).
\]

For $U= (u, h)$, $V= (v, \ell)$, one can write
\begin{align*}
\mathscr{T}^s(U) & -\mathscr{T}^s(V)=
(\mathcal{M}^{-1}[\mathcal{S}(u)(t)]-\mathcal{M}^{-1}[\mathcal{S}(v)(t)]
) \!
\int_0^t\sum_{k\ge1}\theta^U_N(s)f^n_k(\mathcal{S}(u),\mathcal{S}(u)u,x)d\hat{\beta}_k^1\\
&+\mathcal{M}^{-1}[\mathcal{S}(v)(t)]\int_0^t\sum_{k\ge1}
\left(\theta^U_N(s)f^n_k(\mathcal{S}(u),\mathcal{S}(u)u,x)
-\theta^V_N(s)f^n_k(\mathcal{S}(v),\mathcal{S}(v)v,x)\right)d\hat{\beta}_k^1\\
&+\int_0^t \sum_{k\ge1}\left(\theta^U_N(s)g_k(h,x)
-\theta^V_N(s)g_k(\ell,x)\right)d\hat{\beta}_k^2.
\end{align*}
Then
\begin{align*}
{\rm E}&\left\|\mathscr{T}^s(U)- \mathscr{T}^s(V)\right\|^2_{X_n}\\
&\le {\rm E}\sup_{0\le t\le T_{n,N}}\left\|\int_0^t(\mathcal{M}^{-1}
[\mathcal{S}(u)(t)]-\mathcal{M}^{-1}[\mathcal{S}(v)(t)])\sum_{k\ge1}
f^n_k(\mathcal{S}(u),\mathcal{S}(u)u,x)d\hat{\beta}_k^1\right\|^2_{X_n}\\
&\quad +{\rm E}\sup_{0\le t\le
T_{n,N}}\left\|\int_0^t\mathcal{M}^{-1}[\mathcal{S}(v)(t)]
\sum_{k\ge1}\left[f^n_k(\mathcal{S}(u),\mathcal{S}(u)u,x)
-f^n_k(\mathcal{S}(v),\mathcal{S}(v)v,x)\right]d\hat{\beta}_k^1\right\|^2_{X_n}\\
&\quad+{\rm E}\sup_{0\le t\le T_{n,N}}\left\|\int_0^t
\sum_{k\ge1}\left(g_k(h,x)
-g_k(\ell,x)\right)d\hat{\beta}_k^2\right\|^2_{X_n}\\
&=: I_1+I_2+I_3.
\end{align*}
For the term $I_1$, using the Burkholder-Davis-Gundy inequality,
\eqref{assumption}, \eqref{NS13} and the fact that $\mathcal{S}(u)$
is a contraction map,  we have
\begin{align*}
I_1&\le {\rm E} \Big (\int_0^{T_{n,N}}\Big\|(\mathcal{M}^{-1}
[\mathcal{S}(u)(t)]-\mathcal{M}^{-1}
[\mathcal{S}(v)(t)])\sum_{k\ge1}f^n_k(\mathcal{S}(u),\mathcal{S}(u)u,x)\Big\|^2_{X_n}ds \Big )\\
&\le {\rm E} \Big
(\|(\mathcal{M}^{-1}[\mathcal{S}(u)(t)]-\mathcal{M}^{-1}[\mathcal{S}(v)(t)])\|^2_{\mathcal
{L}(X_n,X_n)}\int_0^{T_{n,N}}\Big\|\sum_{k\ge1}f^n_k(\mathcal{S}(u),\mathcal{S}(u)u,x)\Big\|^2_{X_n}ds \Big )\\
&\le {\rm E} \Big
(\|\mathcal{S}(u)-\mathcal{S}(v)\|^2_{L^1(D)}\int_0^{T_{n,N}}
\Big\|\sum_{k\ge1}f^n_k(\mathcal{S}(u),\mathcal{S}(u)u,x)\Big\|^2_{X_n}ds \Big )\\
&\le CT_{n,N} {\rm E}\left(\|u-v\|^2_{X_n}\right)\le
CT_{n,N} {\rm E}\left(\|U-V\|^2_{X_n}\right).
\end{align*}

Similarly, for the terms $I_2$ and $I_3$, by \eqref{assumption}, we
obtain
\begin{align*}
I_2&\le {\rm E} \Big
(\!\!\int_0^{T_{n,N}}\!\!\!\|\mathcal{M}^{-1}[\mathcal{S}(u)(t)]\|^2_{\mathcal
{L}(X_n,X_n)} \Big\|
\sum_{k\ge1}(f^n_k(\mathcal{S}(u),\mathcal{S}(u)u,x)\!
-\!f^n_k(\mathcal{S}(v),\mathcal{S}(v)v,x)) \Big\|^2_{X_n}ds \Big )\\
&\le {\rm E} \Big (\left[\inf\mathcal{S}(u)(t)\right]^{-1}\int_0^{T_{n,N}}\Big\|
\sum_{k\ge1}( f^n_k(\mathcal{S}(u),\mathcal{S}(u)u,x)\!
-\!f^n_k(\mathcal{S}(v),\mathcal{S}(v)v,x)) \Big\|^2_{X_n}ds \Big )\\
&\le \left[\inf\mathcal{S}(u)(t)\right]^{-1}{\rm E}
 \Big (\int_0^{T_{n,N}}
\Big\|\sum_{k\ge1}(f^n_k(\mathcal{S}(u),\mathcal{S}(u)u,x)\!
-\!f^n_k(\mathcal{S}(v),\mathcal{S}(v)v,x))\Big\|^2_{X_n}ds \Big )\\
&\le CT_{n,N} {\rm E}\left(\|\mathcal{S}(u)-\mathcal{S}(v)\|^2_{X_n}+\|u-v\|^2_{X_n}\right)\\
&\le CT_{n,N} {\rm E}\left(\|u-v\|^2_{X_n}+\|h-\ell\|^2_{X_n}\right) \le
CT_{n,N} {\rm E}\left(\|U-V\|^2_{X_n}\right),
\end{align*}
and
\begin{align*}
I_3\le {\rm E}\int_0^{T_{n,N}}\Big\|\sum_{k\ge1}(g_k(h,x)
-g_k(\ell,x))\Big\|^2_{X_n}dt\le CT_{n,N} {\rm E}\left(\|U-V\|^2_{X_n}\right).
\end{align*}
Here we have used the fact that
$\|\mathcal{M}^{\frac{1}{2}}[\mathcal{S}(u)]-\mathcal{M}^{\frac{1}{2}}[\mathcal{S}(v)]\|_{\mathcal{L}(X_n,X_n)}\le
C\|\mathcal{S}(u)-\mathcal{S}(v)\|_{L^2}$. This can be proved by
using the definition of $\mathcal{M}$ and the mean value theorem.
Thus, for  sufficiently small $T_{n,N}$ such that $\kappa=CT_{n,N} <1$,
\begin{align*}
{\rm E}\left\|\mathscr{T}(U)-\mathscr{T}(V)\right\|^2_{X_n}\le\kappa
{\rm E}\left(\|U-V\|^2_{X_n}\right).
\end{align*}
Then $\mathscr{T}$ is a contraction. Similar arguments indicate that with a further refined choice of $T_{n,N} $, $\mathscr{T}$ maps $B_{N,T_{n,N}}$ into itself.
Therefore $\mathscr{T}$ has a unique fixed point $(u^N,B^N)$, proving the proposition.
\end{proof}

\begin{remark}
The cut-off function $\theta_N(x)$ is introduced here to guarantee the global Lipschitz continuity of the map.
\end{remark}

\medskip

\noindent{\bf Step 3.} Derive uniform a priori estimates to extend $\tau_{n, N}$ to $T$. Set $U^N_n=(u^N_n,B^N_n)$ be the solution to equation \eqref{r-integral} on $[0,\tau_{n,N})$ obtained from Proposition \ref{prop_extN}. Let us introduce the following stopping times:
\begin{equation*}
\tau_{n,N}=\begin{cases} \inf\left\{t\ge0:
\|U_n^N(t)\|_{L^2(D)}\ge N\right\}\wedge\inf\left\{t\ge0:
\left\|\int_0^tF^N_ndW\right\|_{L^2(D)}ds\ge N\right\},\\
T, \; \text{otherwise}.
\end{cases}
\end{equation*}
We denote $(u_N,B_N):=(u^N_n,B^N_n)$. Let
$\rho_N:=\mathcal{S}(u^N)$.

Our next goal is to prove
\begin{lemma}\label{p-timelemma}
For any fixed $n$,
\begin{align}\label{p-time}
\lim_{N\to\infty}\textrm{P}(\tau_{n,N}=T)=1.
\end{align}
\end{lemma}
To prove the above lemma, we need the following energy estimates.

\begin{proposition}\label{prop_energy}
For any $(u_N, B_N)$ solving equation \eqref{r-integral}, for $1\le
p<\infty$, we have
\begin{equation}\label{1.11_a}
\begin{split}
\!\!\!{\rm E}\bigg[&\sup_{0\le t\le T}\mathscr{E}_\delta+{\rm
E}\!\!\int_0^T\!\! \left( \mu\|\nabla
u_N(t)\|^2_{L^2(D)}+(\lambda+\mu)\|\Dv u_N(t)\|^2_{L^2(D)}+
\nu\|\nabla
B_N(t)\|^2_{L^2(D)} \right)dt\\
&+\varepsilon {\rm
E}\int_0^T\int_{D}\left(a\gamma\rho_N^{\gamma-2}+\delta\beta\rho_N^{\beta-2}\right)|\nabla
\rho_N(t)|^2dxdt\bigg]^p\le C{\rm
E}\left(\mathscr{E}_{0,\delta}\right)^p.
\end{split}
\end{equation}
\end{proposition}
\begin{proof}
 Define the function $\Phi(\rho, m):=\int_D {|m|^2\over\rho} dx$. Note that
$$
\nabla_m \Phi(\rho,m)=\int_D {2m\over\rho}dx,\quad \nabla^2_m
\Phi(\rho,m)=\int_D{2\over\rho}\mathbb{I}dx,\quad \partial_{\rho}
\Phi(\rho,m)=-\int_D \frac{|m|^2}{\rho^2}dx.
$$
 Here
$\mathbb{I}$ is the identity matrix.

Apply  It\^{o}'s formula  in Lemma \ref{Ito-formula} to the above
function $\Phi$ with $(\rho,m)=(\rho_n,\rho_Nu_N)$, and then apply
It\^{o}'s formula  again with $\Phi(B):=\int_D|B|^2dx$. From
equation \eqref{A-MHD}, one deduces that
\begin{align}\label{1.1}
\notag &d\int_D\!\left( {1\over2}|\sqrt{\rho_N}u_N|^2\!+\!{a\over
\gamma-1}\rho_N^\gamma\!+\!{\delta\over\beta-1}\rho_N^\beta\!+\!\frac{1}{2}|B_N|^2 \right)
dx\!+\!\int_D \!\!\left( \mu |\nabla u_N|^2+(\lambda+\mu)|\Dv u_N|^2 \right) dxds
\\&\qquad+\int_D\nu|\nabla B_N|^2 dxds
+\varepsilon\int_D\left(a\gamma\rho_N^{\gamma-2}+\delta\beta\rho_N^{\beta-2}\right)|\nabla
\rho_N|^2dxds\\
\notag &=\int_D
\sum_{k=1}f^n_k(\rho_N,\rho_N u_N,x)\cdot u_Ndxd\hat{\beta}_k^1+{1\over2}\int_D\sum_{k\ge1}|f_k(\rho_N,\rho_Nu_N,x)/\sqrt{\rho_N}|^2dxds\\
&\notag\quad +\int_D \sum_{k\ge1}g_k(B_N,x)\cdot
B_Ndxd\hat{\beta}_k^2+{1\over2}\int_D\sum_{k\ge1}|g_k(B_N,x)|^2dxds,
\end{align}
where $s\in[0,t\wedge\tau_{n,N}]$, $t\wedge\tau_{n,N}=\min\{t,\tau_{n,N}\}$ and
$t\in[0,T]$.
Let $$
\mathscr{E}_\delta=\int_D \left( {1\over
2}\rho_N|u_N|^2+{a\over
\gamma-1}\rho_N^\gamma+{\delta\over\beta-1}\rho_N^\beta+\frac{1}{2}|B_N|^2 \right)
dx,$$
$$\mathscr{E}_{0,\delta}=\int_D \left( {1\over
2}\rho_{0,\delta}|u_{0,\delta}|^2+{a\over
\gamma-1}\rho_{0,\delta}^\gamma+{\delta\over\beta-1}\rho_{0,\delta}^\beta+\frac{1}{2}|B_{0,\delta}|^2
\right) dx,$$  where $m_N=\rho_Nu_N$. Integrating \eqref{1.1} on
$[0,s]$ for all $s\in[0,t\wedge\tau_{n,N}]$ yields
\begin{align*}
&\mathscr{E}_\delta+ \int_0^s \left[ \mu\|\nabla
u_N(r)\|^2_{L^2(D)}+(\lambda+\mu)\|\Dv
u_N(r)\|^2_{L^2(D)}+\nu\|\nabla
B_N(r)\|^2_{L^2(D)} \right]dr\\
&\quad+\varepsilon\int_0^s
\int_{D}\left(a\gamma\rho_N^{\gamma-2}+\delta\beta\rho_N^{\beta-2}\right)|\nabla
\rho_N|^2dxdr\\\notag &\le\mathscr{E}_{0,\delta}+{1\over2}\int_0^s
\sum_{k\ge1} \left\| {f^n_k(\rho_N,\rho_Nu_N,x) \over \sqrt{\rho_N(r)}} \right\|^2_{L^2(D)}dr+\left|\int_0^s
\langle u_N(r),\sum_{k\ge1}f^n_k(\rho_N,\rho_Nu_N,x)\rangle d\hat{\beta}_k^1\right|\\
&\quad\notag+{1\over2}\int_0^s
\sum_{k\ge1}\|g_k(B_N,x)\|^2_{L^2(D)}dr+\left|\int_0^s \langle
B_N(r),\sum_{k\ge1}g_k(B_N,x)\rangle d\hat{\beta}_k^2\right|,
\end{align*}
where $\langle\cdot,\cdot\rangle$ is the inner product in $L^2(D)$.
Taking the sup  over $[0,t\wedge\tau_{n,N}]$ and the mathematical
expectation in the above inequality, one easily deduces that
\begin{align}\label{1.1a}
\notag&{\rm E}\sup_{0\le s\le t\wedge
\tau_{n,N}}\!\!\mathscr{E}_\delta+ {\rm E}\int_0^{t\wedge
\tau_{n,N}}\!\!\! \left( \mu\|\nabla
u_N(s)\|^2_{L^2(D)}+(\lambda+\mu)\|\Dv
u_N(s)\|^2_{L^2(D)}+\nu\|\nabla
B_N(s)\|^2_{L^2(D)} \right)ds\\
\notag&\quad+\varepsilon {\rm E}\int_0^{t\wedge
\tau_{n,N}}\!\!\int_{D}\left(a\gamma\rho_N^{\gamma-2}+\delta\beta\rho_N^{\beta-2}\right)|\nabla
\rho_N(s)|^2dxds\\
\notag&\le\! {\rm E}\mathscr{E}_{0,\delta}\!+\!{1\over2}{\rm
E}\!\int_0^{t\wedge \tau_{n,N}}\!\! \sum_{k\ge1} \left\|
{f_k^n(\rho_N,\rho_Nu_N,x) \over \sqrt{\rho_N(s)}}
\right\|^2_{L^2(D)} ds\!+\!{\rm E}\sup_{s}\left|\int_0^s
\!\!\langle u_N,\sum_{k\ge1}f_k^n(\rho_N,\rho_Nu_N,x)\rangle d\hat{\beta}_k^1\right|\\
\notag&\quad+{1\over2}{\rm E}\int_0^{t\wedge \tau_{n,N}}\!\!
\sum_{k\ge1}\|g_k(B_N,x)\|^2_{L^2(D)}+{\rm E}\sup_{s}\left|\int_0^s \langle B_N,\sum_{k\ge1}g_k(B_N,x)\rangle d\hat{\beta}_k^2\right|\\
&=:{\rm E}\mathscr{E}_{0,\delta}+I_1+I_2+I_3+I_4.
\end{align}
We now estimate the terms $I_i, i=1,2,3,4$.  For the first term
$I_1$, the assumptions \eqref{assumption} on $f$ and the
Cauchy-Schwarz inequality imply that
\begin{align*}
\begin{split}
I_1 &\lesssim {\rm E}\int_0^{t\wedge
\tau_{n,N}}\int_D \left(|\sqrt{\rho_N}u_N|^2+|\rho_N|^\gamma \right) dxds\\
&\lesssim  {\rm E}\int_0^{t\wedge
\tau_{n,N}}\left(\|\sqrt{\rho_N}u_N\|^2_{L^2(D)}+\|\rho_N\|^\gamma_{L^\gamma(D)}\right)
ds,
\end{split}
\end{align*}
For $I_2$, by the Burkholder-Davis-Gundy, H\"{o}lder  and Young
inequalities, and \eqref{assumption},  for small $\eta>0$, we have
\begin{align*}
\notag I_2&\lesssim {\rm E}\left[\int_0^{t\wedge\tau_{n,N}}
\sum_{k\ge1}\left\langle
f^n_k(\rho_N,\rho_Nu_N,x),u_N\right\rangle^2ds\right]^{1\over2}\\
&\lesssim
{\rm E}\left[\int_0^{t\wedge\tau_{n,N}}\|\sqrt{\rho_N}u_N\|_{L^2(D)}^2
\left(\|\sqrt{\rho_N}u_N\|^2_{L^2(D)}+\|\rho_N\|^\gamma_{L^\gamma(D)} \right)ds\right]^{1\over2}\\
&\notag\le \eta {\rm E}\sup_{0\le s\le t\wedge
\tau_{n,N}}\|\sqrt{\rho_N}u_N(s)\|^2_{L^2(D)}+C_{\eta}{\rm E}\int_0^{t\wedge\tau_{n,N}}\!\!\int_{D} \left( |\sqrt{\rho_N}u_N|^2+|\rho_N|^\gamma \right)
dxds.
\end{align*}
Similarly, from \eqref{assumption}, one has
\begin{equation*}
I_3\le C {\rm E}\int_0^{t\wedge \tau_{n,N}}\|B_N\|^2_{L^2(D)} ds,
\end{equation*}
and
\begin{equation*}
I_4\le \eta {\rm E}\sup_{0\le s\le t\wedge
\tau_{n,N}}\|B_N(s)\|^2_{L^2(D)}+C_\eta {\rm E}\int_0^{t\wedge
\tau_{n,N}}\|B_N\|^2_{L^2(D)}ds.
\end{equation*}

When $p>2$, we obtain as $I_2$ that
\begin{align*}
|I_2|^p&\le C{\rm E}\left[\int_0^{t\wedge\tau_{n,N}}
\sum_{k\ge1}\left\langle f_k(\rho_N,\rho_Nu_N,x),u_N\right\rangle^2ds\right]^{p\over2}\\
&\notag\le \eta {\rm E}\left(\sup_{0\le s\le t\wedge
\tau_{n,N}}\|\sqrt{\rho_N}u_N(s)\|^{2p}_{L^2(D)}\right)+C_{\eta}{\rm E}\int_0^{t\wedge\tau_{n,N}}
\left(\int_{D}|\sqrt{\rho_N}u_N|^2+|\rho_N|^\gamma dx\right)^pds,
\end{align*}
and
\begin{align*}
|I_4|^p&\le C{\rm E}\left[\int_0^{t\wedge\tau_{n,N}}
\sum_{k\ge1}\left\langle g_k(B_N,x),B_N\right\rangle^2ds\right]^{p\over2}\\
&\notag\le \eta {\rm E}\left(\sup_{0\le s\le t\wedge
\tau_{n,N}}\|B_N(s)\|^{2p}_{L^2(D)}\right)+C_{\eta}{\rm E}\int_0^{t\wedge\tau_{n,N}}
\left(\int_{D}|B_N|^2dx\right)^pds.
\end{align*}
Plugging the estimates on $I_1-I_4$ into \eqref{1.1a}, one has, for
small enough $\eta>0$,
\begin{align*}
& {\rm E}\sup_{0\le s\le t\wedge \tau_{n,N}}\!\mathscr{E}_\delta+
{\rm E}\int_0^{t\wedge \tau_{n,N}}\!\!\! \mu\|\nabla
u_N(s)\|^2_{L^2(D)}+(\lambda+\mu)\|\Dv u_N(s)\|^2_{L^2(D)}+
\nu\|\nabla
B_N(s)\|^2_{L^2(D)}ds\\
&\quad +\varepsilon {\rm E}\int_0^{t\wedge
\tau_{n,N}}\int_{D}\left(a\gamma\rho_N^{\gamma-2}+\delta\beta\rho_N^{\beta-2}\right)|\nabla
\rho_N(s)|^2dxds\\
&\le {\rm E}\mathscr{E}_{0,\delta}+C_{\eta}{\rm
E}\int_0^{t\wedge\tau_{n,N}}\int_{D}(|\sqrt{\rho_N}u_N|^2+|B_N|^2)dxds.
\end{align*}
Then by Gronwall's inequality, we have
\begin{align*}
&\notag {\rm E}\sup_{0\le s\le t\wedge \tau_{n,N}}\!\mathscr{E}_\delta\!+\!
{\rm E}\!\int_0^{t\wedge \tau_{n,N}}\!\!\!\left(\mu\|\nabla
u_N(s)\|^2_{L^2(D)}\!+\!(\lambda+\mu)\|\Dv u_N(s)\|^2_{L^2(D)}\!+\!
\nu\|\nabla
B_N(s)\|^2_{L^2(D)} \right)ds\\
&\label{1.9}\quad +\varepsilon {\rm E}\int_0^{t\wedge
\tau_{n,N}}\int_{D}\left(a\gamma\rho_N^{\gamma-2}+\delta\beta\rho_N^{\beta-2}\right)|\nabla
\rho_N(s)|^2dxds\le C{\rm E}\mathscr{E}_{0,\delta},
\end{align*}
from which we get the following $L^2$ energy estimates for any
$t\in[0,T]$:
\begin{equation}\label{1.10}
\begin{split}
\!\!\!\!{\rm E}&\sup_{0\le s\le t}\mathscr{E}_\delta+{\rm
E}\!\!\int_0^t\!\! \left(\mu\|\nabla
u_N(s)\|^2_{L^2(D)}+(\lambda+\mu)\|\Dv u_N(s)\|^2_{L^2(D)}+
\nu\|\nabla
B_N(s)\|^2_{L^2(D)} \right)ds\\
&+\varepsilon
{\rm E}\int_0^t\int_{D}\left(a\gamma\rho_N^{\gamma-2}+\delta\beta\rho_N^{\beta-2}\right)|\nabla
\rho_N(s)|^2dxds\le C{\rm E}\mathscr{E}_{0,\delta}.
\end{split}
\end{equation}

In the same way,  for $p>2$, by virtue of the estimates on $|I_2|^p$
and $|I_4|^p$, we obtain the $L^p$ energy estimates for $t\in[0,T]$
\begin{equation*}\label{1.11}
\begin{split}
{\rm E}\bigg[&\sup_{0\le s\le t}\mathscr{E}_\delta+{\rm E}\int_0^t \left( \mu\|\nabla
u_N(s)\|^2_{L^2(D)}+(\lambda+\mu)\|\Dv u_N(s)\|^2_{L^2(D)}+
\nu\|\nabla
B_N(s)\|^2_{L^2(D)} \right)ds\\
&+\varepsilon {\rm
E}\int_0^t\int_{D}\left(a\gamma\rho_N^{\gamma-2}+\delta\beta\rho_N^{\beta-2}\right)|\nabla
\rho_N(s)|^2dxds\bigg]^p\le C{\rm
E}\left(\mathscr{E}_{0,\delta}\right)^p,
\end{split}
\end{equation*}
which completes the proof of the proposition.
\end{proof}

Now, we are ready to prove Lemma \ref{p-timelemma}.
\begin{proof}[Proof of Lemma \ref{p-timelemma}] It follows from \eqref{1.10} that
\begin{align*}
{\rm E}\int_0^T\|\nabla u_N\|^2_{L^2(D)}dt\le C {\rm
E}\mathscr{E}_{0,\delta},\quad {\rm
E}\sup_{t\in[0,T]}\|\sqrt{\rho_N}u_N(t)\|^2_{L^2(D)}\le C {\rm
E}\mathscr{E}_{0,\delta}.
\end{align*}
Since dim$X_n$ is finite,  the $L^\infty, C^2$ and $L^2$ norms are
equivalent on $X_n$. Moreover,  $\rho=\mathcal{S}[u]$ is bounded. It
follows that
\begin{align}\label{1.12}
\underline{\rho}\exp
  \Big (-\!\int_0^T\!\! \|\nabla
u_N(s)\|^2_{L^2(D)}ds\Big )\lesssim \rho_N(t,x)\lesssim
\overline{\rho} \exp\Big(\int_0^T \|\nabla u_N(s)\|^2_{L^2(D)}ds\Big
),
\end{align}
which yields that
\begin{align}\label{1.13a}
{\rm E}\left[\exp\left(-\int_0^T \|\nabla
u_N(s)\|^2_{L^2(D)}ds\right)\sup_{t\in[0,T]}\|u_N\|_{L^2(D)}\right]\le
C,
\end{align}
where $C$ is independent of  $N$. It follows from \eqref{1.11_a},
\eqref{1.12} and \eqref{1.13a} that
$\lim\limits_{N\to\infty}\textrm{P}(\tau_{n,N}=T)=1$.
\end{proof}

Moreover, it follows from \eqref{1.11_a} that
\begin{align*}
\sqrt{\varepsilon\delta}\rho_N^{\beta\over2} \ \ \mbox{is bounded
in} \ \ L^{p}(\Omega, L^2(0,T; H^1(D))).
\end{align*}
The Sobolev embedding $H^1(D)\hookrightarrow L^6(D)$ yields that
\begin{align*}
\rho_N^{\beta\over2}\in L^{p}(\Omega,L^2(0,T; L^6(D))),
\end{align*}
that is,
\begin{align*}
{\rm E}\left(\|\rho_N^\beta\|_{L^1(0,T;L^3(D))}\right)^{2p}\le C, \ \
\mbox{where $C$ is independent of $N$}.
\end{align*}
In view of \eqref{1.11_a}, we deduce that
\begin{align*}
{\rm E}\left(\sup_{t\in[0,T]}\|\rho_N^\beta\|_{L^1(D)}\right)^p\le
C,
\end{align*}
this together with the interpolation inequality
$\|\rho_N^\beta\|_{L^2(D)}\lesssim
\|\rho_N^\beta\|^{1\over4}_{L^1(D)}\|\rho_N^\beta\|^{3\over4}_{L^3(D)}$
yields that
\begin{align*}
\notag{\rm E}\left(\int_0^T\!\!
\|\rho^\beta_N\|^{4\over3}_{L^2(D)}dt\right)^p&\lesssim
{\rm E}\left(\!\!\int_0^T\|\rho_N^\beta\|^{1\over3}_{L^1(D)}\|\rho_N^\beta\|_{L^3(D)}dt\right)^p\\
&\lesssim {\rm
E}\left(\sup_{t\in[0,T]}\|\rho_N^\beta(t)\|_{L^1(D)}\right)^{2p\over3}\!\!+{\rm
E}\left(\int_0^T\!\!\|\rho_N^\beta\|
_{L^3(D)}dt\right)^{2p} \\
&\notag\le C.
\end{align*}
By the H\"{o}lder inequality, we have
\begin{align*}
{\rm E}\left(\int_0^T\int_{D}\rho^{4\beta\over3}_Ndxdt\right)^p\le
{\rm E}\left(\int_0^T\left(\int_D
\rho_N^{2\beta}dx\right)^{2\over3}\left(\int_D
dx\right)^{1\over3}dt\right)\le C.
\end{align*}
Thus
\begin{align*}
{\rm E}\left(\|\rho_N\|_{L^{\beta+1}((0,T)\times D)}\right)^p\le C \ \
\mbox{if} \ \ \beta\ge3.
\end{align*}
If $\beta\ge4$, multiply $\eqref{A-MHD}_1$ by $\rho_N$ and integrate
by parts to get
\begin{align*}
{\rm E}\left(\varepsilon\int_0^T\|\nabla
\rho_N\|^2_{L^2(D)}\right)^p&\lesssim {\rm E}\left(\|
\rho_N(0)\|^2_{L^2(D)}+\int_0^T\int_D \rho_N^2{\rm div}
u_Ndxdt\right)^p\\
&\lesssim {\rm E}\left(\|
\rho_N(0)\|^2_{L^2(D)}+\sup_{t\in[0,T]}\|\rho_N\|^4_{L^4(D)}+\|\nabla
u_N\|^2_{L^2([0,T]\times D)}\right)^p\\
&\lesssim C.
\end{align*}

For each $t\in[0,T]$, we define
$(\rho_n,u_n,B_n):=\lim_{N\to\infty}(\rho_n^N,u^N_n,B^N_n)$. By the
same argument as above, we can infer that $(\rho_n,u_n,B_n)$
satisfies the corresponding a priori estimate uniformly in $n$. More
precisely, for any $1\le p<\infty$, we have the following lemma:

\begin{lemma}\label{lemma3.3}
Let $(\rho_n,u_n,B_n)$ be the solution of
\eqref{A-MHD}--\eqref{A-boundary} on $\Omega\times(0,T)\times D$
constructed above. Then,  for $\beta\ge 4$,  we have,
\begin{align}\label{energy}
\begin{split}
& {\rm E} \Big (\sup_{t\in[0,T]}\|\rho_n(t)\|^\gamma_{L^\gamma(D)}
\Big )^p\le C,\quad\quad {\rm E} \Big
(\delta\sup_{t\in[0,T]}\|\rho_n(t)\|^\beta_{L^\beta(D)} \Big )^p\le
C,
\\
& {\rm E} \Big
(\sup_{t\in[0,T]}\|\sqrt{\rho_n(t)}u_n(t)\|^2_{L^2(D)} \Big )^p\le
C,\quad {\rm E} \Big (\|u_n(t)\|^2_{L^2([0,T];H^1(D))}\Big )^p\le C,
\\
& {\rm E} \Big (\sup_{t\in[0,T]}\|B_n(t)\|^2_{L^2(D)}\Big )^p\le
C,\quad \ \  {\rm E} \Big (\|B_n(t)\|^2_{L^2([0,T];H^1(D))}\Big
)^p\le C,
\\
& {\rm E} \Big (\varepsilon \|\nabla \rho_n(t)\|^2_{L^2((0,T)\times
D)}\Big )^p\le C,\quad {\rm E} \Big
(\|\rho_n\|_{L^{\beta+1}((0,T)\times D)} \Big )^p\le C,
\end{split}
\end{align}
where the constant $C$ is independent of $n$.
\end{lemma}

\subsection{Tightness Property}\label{s3.2}

 In this subsection, we shall show the tightness
property for the approximation solution in the following lemma.
\begin{lemma}\label{Lemma-Tight} Define
\begin{align*}
S : = & \ C(0,T;\mathbb{R})\times\left( C([0,T];L_w^{\beta}(D))\cap
L^2(0,T;L^2(D))\cap L^2(0,T; H_w^1(D))\right)\\
&\times L^2(0,T;H_w^1(D))\times C([0,T];
L_w^{2\beta\over\beta+1}(D))\times \left(L^2(0,T;H_w^1(D))\cap
L^2(0,T; L^2(D))\right)
\end{align*}
equipped with its Borel $\sigma$-algebra.  Let $\Pi_n$ be the
probability on $S$ which is the image of $\textrm{P}$ on $\Omega$ by
the map: $\omega\mapsto
(\hat{\beta}_k(\omega,\cdot),\rho_n(\omega,\cdot),u_n(\omega,\cdot),\rho_nu_n(\omega,\cdot),B_n(\omega,\cdot))$,
that is, for any $A\subseteq S$,
\begin{equation*}
\Pi_n(A)=\textrm{P}\left\{\omega\in\Omega:
(\hat{\beta}_k(\omega,\cdot),\rho_n(\omega,\cdot),u_n(\omega,\cdot),\rho_nu_n(\omega,\cdot),B_n(\omega,\cdot))\in
A \right\}.
\end{equation*}
Then the family $\Pi_n$ is tight.
\end{lemma}
\begin{proof}
We want to check the tightness of the family of $\Pi_n$ in the
following five steps:

$1^\circ$ (the tightness of $\hat{\beta}_k$) First, we will check
the tightness of $\hat{\beta}_k$, that is, for $\varepsilon>0$, we
now need to find the compact subset $\Sigma_\varepsilon\subset
C(0,T;\mathbb{R})$ such that $\textrm{P}(\hat{\beta}_k\notin
\Sigma_\varepsilon)\le {\varepsilon\over5}$. For
$\Sigma_\varepsilon$ we rely on classical results concerning the
Brownian motion. For a constant $L_\varepsilon$ to be chosen later,
we consider the set
\begin{align*}
\Sigma_\varepsilon=\left\{\hat{\beta}_k(\cdot)\in C(0,T;
\mathbb{R}): \sup_{t_1,t_2\in[0,T],\atop |t_1-t_2|<{1\over
m^6}}m|\hat{\beta}_k(t_2)-\hat{\beta}_k(t_1)|\le L_\varepsilon,\;
\forall m\in\mathbb{N}\right\}.
\end{align*}
$\Sigma_\varepsilon$ is relatively compact in $C(0,T; \mathbb{R})$
by Arzela-Ascoli's Theorem. Furthermore $\Sigma_\varepsilon$ is
closed in $C(0,T; \mathbb{R})$. Therefore $\Sigma_\varepsilon$ is a
compact subset of $C(0,T; \mathbb{R})$. We can show that
$\textrm{P}(\hat{\beta}_k\notin \Sigma_\varepsilon)\le {C\over
L^4_\varepsilon}$. In fact, by Chebyshev's inequality
$\textrm{P}\{\omega: \xi(\omega)\ge r\}\le {1\over r^q}{\rm
E}[|\xi(\omega)|^q]$, one has
\begin{align*}
&\textrm{P}\{\omega: \hat{\beta}_k(\omega,\cdot)\notin \Sigma_\varepsilon\}\\
&\le \textrm{P}\left[\cup_{m=1}^\infty\left\{\omega:
\sup_{t_1,t_2\in[0,T],|t_1-t_2|<m^{-6}}|\hat{\beta}_k(t_1)-\hat{\beta}_k(t_2)|>{L_\varepsilon\over
m}\right\}\right]\\
&\le\sum_{m=1}^\infty\sum^{m^6-1}_{i=0}\left({m\over
L_\varepsilon}\right)^4 {\rm E}\left[\sup_{iTm^{-6}\le t\le
(i+1)Tm^{-6}}|\hat{\beta}_k(t)-\hat{\beta}_k(iTm^{-6})|^4\right]\\
&\le C\sum_{m=1}^\infty\left({m\over
L_\varepsilon}\right)^4(Tm^{-6})^2m^6={C\over
L^4_\varepsilon}\sum_{m=1}^\infty {1\over m^2}.
\end{align*}
We choose $L^4_\varepsilon={1\over
5C\varepsilon}\left(\sum_{m=1}^\infty {1\over m^2}\right)^{-1}$ to
obtain $\textrm{P}(\hat{\beta}_k\notin \Sigma_\varepsilon)\le
{\varepsilon\over5}$.

$2^\circ$ (the tightness of $\rho_n$) In this step, we want to find
$X_\varepsilon \subset C([0,T];L_w^{\beta}(D))\cap
L^2(0,T;L^2(D))\cap L^2(0,T; H_w^1(D))$ such that
$\textrm{P}(\rho_n\notin X_\varepsilon)\le {\varepsilon\over5}$. For
this, we define a function space $\mathcal{X}$ with the norm
\begin{align*}\|f\|_{\mathcal{X}}=&\sup_{0\le
t\le T}\|f(t)\|_{L^\beta(D)}+
\|f(t)\|_{L^2([0,T];H^1(D))}\\
&+\sup_{0\le t\le
T}\|\partial_t{f}\|_{W^{-1,{2\beta\over{\beta+1}}}(D)}+
\|\partial_tf\|_{L^2([0,T];H^{-1}(D))}.
\end{align*}

Choose $X_\varepsilon$ to be a closed ball of radius $r_\varepsilon$
centered at $0$ in $\mathcal{X}$. By Aubin-Lions Lemma, we know that
$X_\varepsilon$ is compact in $C(0,T;L_w^\beta(D))\cap
L^2(0,T;H_w^1(D))\cap L^2(0,T;L^2(D))$. It follows from
\eqref{energy} in Lemma \ref{lemma3.3} that
\begin{align*}
\textrm{P}(\rho_n\notin
X_\varepsilon)&=\textrm{P}(\|\rho_n\|_\mathcal{X}>r_\varepsilon)\le
{1\over r_\varepsilon}{\rm E}\left(\|\rho_n\|_\mathcal{X}\right)\le
{C\over r_\varepsilon}.
\end{align*}
Choosing $r_\varepsilon=5C\varepsilon^{-1}$, we have
$\textrm{P}(\rho_n\notin X_\varepsilon)\le {\varepsilon\over5}$.
Then $\textrm{P}\{\omega: \rho_n(\omega,\cdot)\in X_\varepsilon\}\ge
1-{\varepsilon\over 5}$.

$3^\circ$(the tightness of $u_n$) In this step, we find
$Y_\varepsilon\subset L^2(0,T; H_w^1(D))$ such that
$\textrm{P}(u_n\notin Y_\varepsilon)\le {\varepsilon\over5}$. To
this end, we choose $Y_\varepsilon$ as a closed ball of radius
$\tilde{r}_\varepsilon$ centered at $0$ in $L^2(0,T;H^1(D))$. Then
$Y_\varepsilon$ is compact in $L^2(0,T; H_w^1(D))$. \eqref{energy}
in Lemma \ref{lemma3.3} implies that
\begin{align*}
\textrm{P}(u_n\notin
Y_\varepsilon)=\textrm{P}(\|u_n\|_{L^2(0,T;H^1(D))}>\tilde{r}_\varepsilon)\le
{1\over
\tilde{r}_\varepsilon}{\rm E}\left(\|u_n\|_{L^2(0,T;H^1(D))}\right)\le
{C\over \tilde{r}_\varepsilon}.
\end{align*}
Choosing $\tilde{r}_\varepsilon=5C\varepsilon^{-1}$, we have
$\textrm{P}(u_n\notin Y_\varepsilon)\le {\varepsilon\over5}$. Then $
\textrm{P}\{\omega: u_n(\omega,\cdot)\in Y_\varepsilon\}\ge
1-{\varepsilon\over 5}.$

$4^\circ$(the tightness of $\rho_nu_n$) In this step, we find
$Z_\varepsilon \subset C([0,T]; L_w^{2\beta\over\beta+1}(D))$ such
that $\textrm{P}(\rho_nu_n\notin Z_\varepsilon)\le
{\varepsilon\over5}$. For this, define a function space $\mathcal
Z\!=\!L^\infty([0,T]; L^{2\beta\over\beta+1}(D))\cap
C^{0,\alpha}\left([0,T]; W^{-\ell,2}(D)\right),\ell>3,~
0<\alpha<{1\over2}$ with the norm $$\|f\|_{\mathcal Z}=\sup_{0\le
t\le
T}\|f(t)\|_{L^{2\beta\over\beta+1}(D)}\\+\|f(t)\|_{C^{0,\alpha}\left([0,T];W^{-\ell,2}(D)\right)}.$$

Because the Brownian motion is not differentiable in time, we can
not estimate $\partial_(\rho_nu_n)$ directly. To overcome the
difficulty, we consider the momentum equation directly. From
\eqref{A-MHD}, we have
\begin{align}\label{Tight0}
\begin{split}
\!\!\!\mathbb{P}[\rho_nu_n(t)]&=\mathbb{P}[\rho_0u_0]\!-\!\!\int_0^t
\!\mathbb{P}\left[\Dv(\rho_n u_n\otimes u_n)\!-\!\mu\Delta
u_n\!-\!(\lambda+\mu) \nabla \Dv u_n\!-\!a\nabla
\rho_n^\gamma\!-\!\delta\nabla\rho_n^\beta\right]ds\\
&\quad+\int_0^t\mathbb{P}\left[ (\nabla\times B_n)\times B_n-
\varepsilon\nabla
u_n\cdot\nabla\rho_nds\right]-\int_0^t\sum_{k\ge1}f^n_k(\rho_n,\rho_nu_n,x)d\hat{\beta}_k^1,
\end{split}
\end{align}
where $\mathbb{P}: L^2(D)\rightarrow X_n$ is the projection onto
$X_n$.
 First, from \eqref{energy}, we have
\begin{align}\label{Tight1}
\begin{split}
{\rm E}\int^{T}_{0} \left\|\mathbb{P}[\Dv(\rho_n u_n\otimes
u_n)]\right\|^2_{W^{-1,{6\beta\over 4\beta+3}}(D)}ds\le {\rm E}
\int^{T}_{0}\|\rho_n u_n\otimes u_n\|^2_{L^{6\beta\over
4\beta+3}(D)}ds\le C.
\end{split}
\end{align}
Similarly, from Lemma \ref{lemma3.3}, we can handle the other
deterministic terms as:
\begin{align}\label{Tight}
\begin{split}
&{\rm E}\int^{T}_{0}\left\|\mathbb{P}[\mu\Delta
u_n+(\lambda+\mu)\nabla \Dv u_n]\right\|^2_{W^{-1,2}(D)}ds\le C,\\
&{\rm E}\int^{T}_{0} \left\|\mathbb{P}\left(\varepsilon\nabla
u_n\cdot\nabla\rho_n\right)\right\|^{\frac{2q}{q+2}}_{W^{-\ell,2}(D)}ds \le C,\\
& {\rm
E}\int^{T}_{0}\left\|\mathbb{P}\left(a\nabla\rho_n^\gamma+\delta\nabla\rho_n^\beta
\right)\right\|^{{\beta+1}\over\beta}_{W^{-1,{\beta+1\over\beta}}(D)}ds\le
C,\\
&{\rm E}\int^{T}_{0} \left\|\mathbb{P}\left[(\nabla\times B_n)\times
B_n\right]\right\|_{W^{-\ell,2}(D)}^2ds \le C.
\end{split}
\end{align}
Here we have used $\nabla \rho_n\in L^q([0,T]; L^2(D)), q>2$ (see
\cite[Lemma 2.4]{FMN}). For the stochastic terms, by Lemma
\ref{lemma3.3}, one has for $\sigma>1$
\begin{align}\label{Tight8}
\notag {\rm E}&\left(\int^t_{s}
\sum_{k\ge1}\left\|f^n_k(\rho_n,\rho_nu_n,x)\right\|_{W^{-\ell,2}(D)}d\hat{\beta}_k^1(s)\right)^{2\sigma}\\
&\le
{\rm E}\left[\int^t_s\sum_{k=1}\|f^n_k(\rho_n,\rho_nu_n,x)\|^2_{W^{-\ell,2}(D)}ds\right]^\sigma\\
&\notag\le {\rm E}\left[
\int^{t}_{s}\sum_{k=1}\|\mathbb{P}(f_k(\rho_n,\rho_nu_n,x)/\sqrt{\rho_n})\|^2_{L^2(D)}\sup_{\|\varphi\|_{W^{\ell,2}(D)}=1}\|\mathcal{M}^{1\over2}[\rho_n]\varphi\|^2_{L^2(D)}ds\right]^\sigma\\
&\notag\le C(t-s)^\sigma.
\end{align}
Thanks to  the Kolmogorov continuity Theorem, we can infer that the
stochastic term is almost surely $\alpha$-H\"{o}lder continuous for
every $0<\alpha<\frac{1}{2}-\frac{1}{2\sigma}$. \if false Here we
also have used the fractional Sobolev space $ W^{\alpha,p}(0,T;H)$:
Assume that $H$ is a separable Hilbert space. Given $p>1,
\alpha\in(0,1)$. Let $ W^{\alpha,p}(0,T;H)$ be the Sobolev space of
all $u(t)\in L^P(0,T;H)$ such that $\int_0^T\int_0^T
\frac{\norm{u(t)-u(s)}^p_H}{|t-s|^{1+\alpha p}}dtds<\infty$, endowed
with the norm:
\begin{align*}
\norm{u}^p_{W^{\alpha,p}(0,T;H)}=\int_0^T
\norm{u(t)}_H^pdt+\int_0^T\int_0^T
\frac{\norm{u(t)-u(s)}^p_H}{|t-s|^{1+\alpha p}}dtds.
\end{align*}\fi
 Note that when $\ell>3$, we have
$W^{-1,\frac{6\beta}{4\beta+3}}(D)\hookrightarrow W^{-\ell,2}(D)$
and $W^{-1,\frac{\beta+1}{\beta}}(D)\hookrightarrow W^{-\ell,2}(D)$.
\if false Then $W^{\alpha,p}(0,T;W^{-\ell,2}(D))\hookrightarrow
C^{0,\iota}\left([0,T]; W^{-\ell,2}(D)\right), 0<\iota<1/2$.\fi

Choose $Z_\varepsilon$ to be a closed ball of radius
$r^\prime_\varepsilon$ centered at 0 in $\mathcal Z$. By
\cite[Corollary B.2]{OM} and \cite{SJ}, we know that $Z_\epsilon$ is
compact in $C([0,T]; L_w^{2\beta\over\beta+1}(D))$. Moreover, from
\eqref{Tight0}-\eqref{Tight8}, by the Chebyshev inequality, we have
\begin{align}\label{Tight9}
\textrm{P}\left(\rho_nu_n\notin Z_\epsilon\right)=
\textrm{P}\left(\|\rho_nu_n\|_{\mathcal
Z}>r^\prime_\varepsilon\right)\le {1\over
r^\prime_\varepsilon}{\rm E}(\|\rho_n u_n\|_{\mathcal Z})\le {C\over
r^\prime_\varepsilon}.
\end{align}
Let $r^\prime_\varepsilon=5C\varepsilon^{-1}$ in \eqref{Tight9}, one deduces that
\begin{align*}
\textrm{P}\left(\rho_nu_n\notin Z_\epsilon\right)\le
{\varepsilon\over5}.
\end{align*}
Then $\textrm{P}\{\omega: \rho_nu_n(\omega,\cdot)\in
Z_\varepsilon\}\ge 1-{\varepsilon\over5}$.
\begin{remark}
Here we can also use the following tightness criterion: Let
$\{f_n\}^\infty_{n=1}$ be a collection of $C([0,T],L^p_w(D))$-valued
random variables defined on a probability space $(\Omega_n,
\mathscr{F}_n,\textrm{P}_n)$ such that $\sup_n {\rm
E}\big(\|f_n\|_{L^\infty([0,T],L^p(D))}\big)<\infty$ and for any
$\varphi\in C_0^\infty(D)$, there exist an integer $k,\sigma>0$,
$q>{1\over\sigma}$ such that $\sup_n {\rm E}[|\langle
f_n(t)-f_n(s),\varphi \rangle|^q]\le
\|\varphi\|^q_{C^k}|t-s|^{\sigma q}$ for all $0\le s,t\le T$. Then
the sequence of the induced measures $\textrm{P} \circ f_n^{-1}$ are
tight on $C([0,T],L^p_w(D))$.

For example, by H\"{o}lder's inequality and Lemma \ref{lemma3.3}, we
have
\begin{align*}
{\rm E}&\left|\int_s^t\langle \Dv(\rho_nu_n\otimes u_n),\varphi\rangle dr\right|\\
&\le
{\rm E}\int_s^t \|\nabla\varphi\|_{L^\infty}\|\sqrt{\rho_n}u_n\|_{L^2(D)}\|\nabla u_n\|_{L^2(D)}\|\rho_n\|_{L^\gamma(D)}^{1\over 2}dr\\
&\le C(t-s)^{1\over2}{\rm E} \left(\|\sqrt{\rho_n}u_n\|_{L^\infty(0,T;L^2(D))}\|\nabla u_n\|_{L^2(0,T;L^2(D))}\|\rho_n\|_{L^\infty(0,T;L^\gamma(D))}^{1\over 2}\right)\\
&\le C(t-s)^{1\over2},\\
{\rm E}&\left|\int_s^t\langle (\nabla\times B_n)\times B_n,\varphi\rangle dr\right|\\
&\le C(t-s)^{1\over2}\|\varphi\|_{L^\infty}{\rm E}(\|\nabla B_n\|_{L^2([0,T],L^2(D))}\|B_n\|_{L^\infty([0,T],L^2(D))})\le C(t-s)^{1\over2},\\
{\rm E}&\left|\int_s^t\langle \mu \Delta u_n+(\lambda+\mu)\nabla \Dv u_n,\varphi\rangle dr\right|\\
&\le C(t-s)^{1\over2}\|\nabla\varphi\|_{L^\infty}{\rm E}(\|\nabla u_n\|_{L^2([0,T],L^2(D))})\le C(t-s)^{1\over2},\\
{\rm E}&\left|\int_s^t\langle a\nabla\rho_n^\gamma-\delta\nabla\rho_n^\beta,\varphi\rangle dr\right|\\
&\le C{\rm E}\int_s^t \|\nabla\varphi\|_{L^\infty}\|\rho_n\|^\gamma_{L^\gamma(D)}\|\rho_n\|^\beta_{L^\beta(D)}dr\le C(t-s)^{1\over2},\\
{\rm E}&\left|\int_s^t\langle \varepsilon\nabla u_n\cdot\nabla\rho_n,\varphi\rangle dr\right|\\
&\le {\rm E}\int_s^t \varepsilon\|\nabla u_n\|_{L^2(D)}\|\nabla \rho_n\|_{L^2(D)}\|\varphi\|_{L^\infty(D)}dr\\
&\le (t-s)^{{1\over 2}-{1\over q}}{\rm E}(\|\nabla u_n\|_{L^2([0,T],L^2(D))}\|\nabla \rho_n\|_{L^q([0,T],L^2(D))})\le C(t-s)^{{1\over 2}-{1\over q}}.
\end{align*}
Here we have used $\nabla \rho_n\in L^q([0,T]; L^2(D)), q>2$. Then we can get the tightness of $\rho_nu_n$.
\end{remark}

$5^\circ$ (the tightness of $B_n$) First, we will find
$R_\varepsilon\subset L^2(0,T; H_w^1(D))$ such that
$\textrm{P}(B_n\notin R_\varepsilon)\le {\varepsilon\over5}$. First,
we choose $R_\varepsilon$ as a closed ball of radius
$\mathring{r}_\varepsilon$ centered at $0$ in $L^2(0,T;H^1(D))$.
Then $R_\varepsilon$ is compact in $L^2(0,T; H_w^1(D))$. From
\eqref{energy}, we have
\begin{align*}
\textrm{P}(B_n\notin
R_\varepsilon)=\textrm{P}(\|B_n\|_{L^2(0,T;H^1(D))}>\mathring{r}_\varepsilon)\le
{1\over
\mathring{r}_\varepsilon}{\rm E}\left(\|B_n\|_{L^2(0,T;H^1(D))}\right)\le
{C\over \mathring{r}_\varepsilon}.
\end{align*}
Choosing $\mathring{r}_\varepsilon=5C\varepsilon^{-1}$, we have
$\textrm{P}(B_n\notin R_\varepsilon)\le {\varepsilon\over5}$. Then $
\textrm{P}\{\omega: B_n(\omega,\cdot)\in R_\varepsilon\}\ge
1-{\varepsilon\over 5}.$

On the other hand, it follows from \eqref{A-MHD} that
\begin{equation}\label{4.40a}
\begin{split}
{\rm E}\int_0^{T-\theta}\norm{B_n(t+\theta)-B_n(t)}^2_{H^{-1}(D)}dt&={\rm E}\int_0^{T-\theta}\left\|\int_t^{t+\theta}
dB_n(s)\right\|^2_{H^{-1}(D)}dt\\
&\le {\rm E}\int_0^{T-\theta}(J_1+J_2+J_3)dt,
\end{split}
\end{equation}
where
\begin{equation*}
J_1(t)=\left\|\int_t^{t+\theta}\nabla\times(u_n\times
B_n)ds\right\|^2_{H^{-1}(D)},\; J_2(t)=\left\|\int_t^{t+\theta}\nu
\Delta B_n ds\right\|^2_{H^{-1}(D)},
\end{equation*}
\begin{equation*}
J_3(t)=\left\|\int_t^{t+\theta} \sum_{k\ge1}g_k(B_n,x)
d\hat{\beta}_k^2\right\|^2_{H^{-1}(D)}.
\end{equation*}

For the term $J_1(t)$, one has
\begin{equation*}
\begin{split}
J_1^{1/2}&=\sup_{\varphi\in H_0^1(D),~
\|\varphi\|_{H^1(D)}=1}\left\{\int_D
\left(\int_t^{t+\theta}\nabla\times(u_n\times
B_n)ds\right)\varphi(x)dx
\right\}\\
&\le C\int_t^{t+\theta}\norm{u_n}_{L^4(D)}\norm{B_n}_{L^4(D)}ds.
\end{split}
\end{equation*}
By H\"{o}lder's inequality and Lemma \ref{lemma3.3}, then
\begin{equation}\label{4.41}
{\rm E}\int_0^{T-\theta}J_1(t)dt\lesssim \theta^{1\over4} {\rm E}\left(
\norm{u_n}^2_{L^2(0,T;H^1(D))}\norm{B_n}^2_{L^{8\over3}(0,T;L^4(D))}\right)\lesssim\theta^{1\over4}.
\end{equation}
For the term $J_2(t)$, similar to \eqref{4.41}, we have
\begin{equation}\label{4.42}
{\rm E}\int_0^{T-\theta}\!\!\!J_2(t)dt\lesssim {\rm
E}\int_0^{T-\theta}\!\!\left(\int_t^{t+\theta}\!\norm{\nabla
B_n}_{L^2(D)}ds\right)^2dt\lesssim \theta {\rm
E}\int_0^{T}\!\norm{\nabla B_n}^2_{L^2(D)}dt\lesssim\theta.
\end{equation}
For the term $J_3$, using the Burkholder-Davis-Gundy inequality,
H\"{o}lder's inequality and \eqref{assumption}, we obtain
\begin{align}\label{4.44}
\begin{split}
{\rm E}\int_0^{T-\theta}\!\!\!J_3(t)dt&\le \int_0^T {\rm
E}\left(\sup_{\varphi\in
H_0^1(D),~\|\varphi\|_{H^1(D)}=1}\int_t^{t+\theta}\int_D \sum_{k\ge1}g_k(B_n,x)\varphi dxd\hat{\beta}_k^2 \right)^2dt\\
&\le\int_0^T {\rm E}\left(\sup_{\varphi\in
H_0^1(D),~\|\varphi\|_{H^1(D)}=1}\int_t^{t+\theta}\left(\sum_{k\ge1}\int_Dg_k(B_n,x)\varphi dx\right)^2 ds \right)dt\\
&\le
\int_0^T\left({\rm E}\int_t^{t+\theta}\norm{\varphi}^2_{L^2(D)}\sum_{k=1}\norm{g_k(B_n,x)}^2_{L^2(D)}ds\right)dt\\
&\le\int_0^T\left({\rm E}\int_t^{t+\theta}\norm{B_n}^2_{L^2(D)}ds\right)dt\\
&\lesssim\theta.
\end{split}
\end{align}
Combining \eqref{4.40a}-\eqref{4.44}, and then by Aubin-Simon Lemma
in \cite{SJ}, we know that $B_n$ is compact in $L^2(0,T; L^2(D))$.

From the above argument, it follows from \cite[pp.16,
Corollary1.3]{KHJ} that the distribution of the joint processes
$(\hat{\beta}_k, \rho_n, u_n, \rho_n u_n, B_n)$ are tight. Hence,
the proof of the lemma is completed.
\end{proof}

\subsection{Application of Jakubowski-Skorokhod Theorem}\label{s3.3}

From the tightness property and Jakubowski-Skorokhod's Theorem,
there exists a subsequence, still indexed by $n$, such that
$\Pi_{n}\to \Pi$ weakly, where $\Pi$ is a probability on $S$, a
probability space $(\tilde{\Omega},
\tilde{\mathscr{F}},\tilde{\textrm{P}})$; and there exist random
variables $
(\tilde{\beta}_{k,n},\tilde{\rho}_{n},\tilde{u}_{n},\tilde{\rho}_{n}\tilde{u}_{n},\tilde{B}_{n})$
with distribution $\Pi_{n}$, $ (\beta_k, \rho, u, h, B)$ with values
in $S$ such that
\begin{align}\label{tightness}
(\tilde{\beta}_{k,n},\tilde{\rho}_{n},\tilde{u}_{n},\tilde{\rho}_{n}\tilde{u}_{n},\tilde{B}_{n})\to
(\beta_k,\rho,u,h,B) \ \ \mbox{in} \ \ S \ \
\tilde{\textrm{P}}-\text{a.s.} .
\end{align}

We set
\begin{align*}
\tilde{\mathscr{F}_t}=\sigma\{\beta_k(s), \rho(s), u(s),
B(s)\}_{s\in[0,t]},
\end{align*}
which is the union $\sigma$-algebra generated by random variables
$\left(\beta_k(s), \rho(s), u(s), B(s)\right)$ for $s\in[0,t]$. It's
easy to see that $\beta_k(t)$ is a $\tilde{\mathscr{F}}_t$ standard
Brownian motion. 

Now, we need to prove that
$(\tilde{\beta}_{k,n},\tilde{\rho}_{n},\tilde{u}_{n},\tilde{B}_{n})$
satisfies the equation \eqref{A-MHD}, that is,
\begin{align*}
\mathbb{P}[\tilde{\rho}_{n}(t)-\tilde{\rho}_{n}(0)]=\int_0^t\mathbb{P}\left[\varepsilon
\Delta\tilde{\rho}_{n}-\Dv(\tilde{\rho}_{n}\tilde{u}_{n})\right]ds,
\end{align*}
\begin{align*}
\notag&\mathbb{P}[\tilde{\rho}_{n}\tilde{u}_{n}(t)]\!+\!\int_0^t\mathbb{P}\left[\Dv(\tilde{\rho}_{n}\tilde{u}_{n}\otimes
\tilde{u}_{n})\!-\!\mu\Delta
\tilde{u}_{n}\!-\!(\lambda+\mu)\nabla\Dv
\tilde{u}_{n}\!+\!a\nabla\tilde{\rho}^\gamma_{n}\!+\!\delta\nabla\tilde{\rho}^\beta_{n}\right]ds\\
&=\mathbb{P}[\tilde{\rho}_{n}\tilde{u}_{n}(0)]\!+\!\int_0^t\!\mathbb{P}\!\left(\!(\nabla\times
\tilde{B}_{n})\times \tilde{B}_{n} \!\!- \!\varepsilon \nabla
\tilde{u}_{n}\!\cdot\!\nabla\tilde{\rho}_{n}\right)\!ds\!+\!\int_0^t\!\sum_{k\ge1}
f^{n}_k(\tilde{\rho}_{n},\tilde{\rho}_{n}\tilde{u}_{n},x)d\tilde{\beta}_{k,n}^1,
\end{align*}
and
\begin{align*}
\mathbb{P}[\tilde{B}_{n}(t)]\!-\!\int_0^t\!\mathbb{P}\!\left[\nabla\times(\tilde{u}_{n}\times
\tilde{B}_{n})+\nu\Delta
\tilde{B}_{n}\right]\!ds\!=\!\mathbb{P}[\tilde{B}_{n}(0)]
\!+\!\int_0^t\!\sum_{k\ge1}\mathbb{P}\!\left[
g_k(\tilde{B}_{n},x)\right]d\tilde{\beta}_{k,n}^2.
\end{align*}
Here $\mathbb{P}$ is a projection from $L^2$ to $X_n$. To this end,
we define
\begin{align*}
\varpi_n(t)=\mathbb{P}[\rho_n(t)-\rho(0)]-\int_0^t\mathbb{P}\left[\varepsilon\Delta\rho_n-\Dv(\rho_nu_n)\right]ds,
\end{align*}
\begin{align*}
\xi_n(t)&=\mathbb{P}[\rho_nu_n(t)]\!+\!\!\int_0^t\!\mathbb{P}\left[\Dv(\rho_nu_n\otimes
u_n)\!-\!\mu\Delta u_n\!-\!(\lambda+\mu)\nabla\Dv
u_n+a\nabla\rho^\gamma_n+\delta\nabla\rho^\beta_n\right]ds\\
&\quad
-\mathbb{P}[\rho_nu_n(0)]-\!\int_0^t\!\mathbb{P}\!\left((\nabla\times
B_n)\times B_n\!-\!\varepsilon\nabla u_n\cdot\nabla
\rho_n\right)\!ds+\!\int_0^t\!\sum_{k\ge1}f^n_k(\rho_n,\rho_nu_n)d\hat{\beta}_k^1,
\end{align*}
\begin{align*}
\zeta_n(t)=\mathbb{P}[B_n(t)-B_n(0)]\!-\!\!\int_0^t\!\!
\mathbb{P}\left[\nabla\times(u_n\times B_n)+\nu\Delta
B_n\right]dt\!\!-\!\!\int_0^t \!\sum_{k\ge1}\mathbb{P}\left[g_k(
B_n,x)\right]d\hat{\beta}_k^2,
\end{align*}
\begin{align*}
G_n\!=\!\int_0^T \norm{\varpi_n(t)}^2_{H^{-1}(D)}dt,\;
Z_n\!=\!\int_0^T \|\xi_n(t)\|^2_{W^{-\ell,2}(D)}dt,\;\ell\ge3,\;
H_n\!=\!\int_0^T \norm{\zeta_n(t)}^2_{H^{-1}(D)}dt.
\end{align*}
Since $(\rho_n, u_n, B_n)$ is a solution, we have
\begin{align}\label{GZH}
G_n=0,\; Z_n=0, \; H_n=0,\; \textrm{P}-\text{a.s.}.
\end{align}

Let
\begin{align*}
\tilde{\varpi}_{n}(t)=\mathbb{P}[\tilde{\rho}_{n}(t)-\tilde{\rho}_{n}(0)]-\int_0^t\mathbb{P}
\left[\varepsilon\Delta\tilde{\rho}_{n}-\Dv(\tilde{\rho}_{n}\tilde{u}_{n})\right]ds,
\end{align*}
\begin{align*}
&\tilde{\xi}_{n}(t)\!=\!\mathbb{P}[\tilde{\rho}_{n}\tilde{u}_{n}(t)]\!-\!\int_0^t\!\mathbb{P}\!
\left[\Dv(\tilde{\rho}_{n}\tilde{u}_{n}\!\otimes
\tilde{u}_{n})\!-\!\mu\Delta
\tilde{u}_{n}\!\!-\!(\lambda+\mu)\nabla\Dv
\tilde{u}_{n}\!+\!a\nabla\tilde{\rho}^\gamma_{n}\!+\!\delta\nabla\tilde{\rho}^\beta_{n}\right]\!ds\\
&-\mathbb{P}[\tilde{\rho}_{n}\tilde{u}_{n}(0)]\!+\!\int_0^t\!\mathbb{P}\!\left(\varepsilon\nabla
\tilde{u}_{n}\!\!\cdot\nabla \tilde{\rho}_{n}\!+\!(\nabla\times
\tilde{B}_{n})\times\tilde{B}_{n}\right)\!ds\!-\!\int_0^t\!\sum_{k\ge1}
f^{n}_k(\tilde{\rho}_{n},\tilde{\rho}_{n}\tilde{u}_{n},x)d\tilde{\beta}_{k,n}^1,
\end{align*}
\begin{align*}
\tilde{\zeta}_{n}(t)\!=\!\mathbb{P}[\tilde{B}_{n}(t)\!-\!\tilde{B}_{n}(0)]\!-\!\int_0^t\!
\mathbb{P}\!\left[\nabla\times(\tilde{u}_{n}\!\times\!
\tilde{B}_{n})\!+\!\nu\Delta \tilde{B}_{n}\right]\!dt\!
-\!\int_0^t\!\sum_{k\ge1}
\mathbb{P}\!\left[g_k(\tilde{B}_{n},x)\right]\!d\tilde{\beta}_{k,n}^2,
\end{align*}
and
\begin{align*}
\tilde{G}_{n}=\int_0^T\|\tilde{\varpi}_{n}(t)\|^2_{H^{-1}(D)}dt,\;\tilde{Z}_{n}
=\int_0^T\|\tilde{\xi}_{n}(t)\|^2_{W^{-\ell,2}(D)}dt,\;
\tilde{H}_{n}=\int_0^T\|\tilde{\zeta}_{n}(t)\|^2_{H^{-1}(D)}dt.
\end{align*}
We want to verify that
\begin{align*}
\tilde{{\rm E}}\tilde{G}_{n}=0,\; \tilde{{\rm E}}\tilde{Z}_{n}=0,
\;\tilde{{\rm E}}\tilde{H}_{n}=0 .
\end{align*} Here $\tilde{{\rm E}}$ denotes the mathematical expectation with respect
to the probability space $(\tilde{\Omega},
\tilde{\mathscr{F}},\tilde{\textrm{P}})$. To this end, we have the
following proposition:
\begin{proposition}\label{prop3.8}
$\tilde{G}_{n}=0, \tilde{Z}_{n}=0, \tilde{H}_{n}=0  \ \
\tilde{\textrm{P}}-\mbox{a.s.}$, that is, $(\tilde{\beta}_{k,n},
\tilde{\rho}_{n},\tilde{u}_{n},\tilde{B}_{n})$ solves the
equation \eqref{A-MHD}.
\end{proposition}
\begin{proof}The difficulty comes from the fact that $(Z_n, H_n)$ is not expressed as a
deterministic function of  $(\hat{\beta}_k, \rho_n, u_n, B_n)$ due
to the presence of the stochastic integrals. By Theorem 2.4 and
\cite[Corollary 2.5]{BH-2011}, we can infer that
\begin{equation}\label{distrbution1}
\mathscr{L}(\hat{\beta}_k,\rho_n, u_n,\rho_nu_n, B_n,\xi_n, \zeta_n
)=\mathscr{L}(\tilde{\beta}_{k,n},\tilde{\rho}_{n},\tilde{u}_{n},
\tilde{\rho}_{n}\tilde{u}_{n},\tilde{B}_{n},
\tilde{\xi}_{n},\tilde{\zeta}_{n}).
\end{equation}
Here $\mathscr{L}(f)$ is the probability distribution of $f$. Note
that $(\tilde{Z}_{n},\tilde{H}_{n})$ is continuous as a function of
$(\tilde{\xi}_{n},\tilde{\zeta}_{n})$. From \eqref{distrbution1},
one deduces that the distribution of $(\tilde{Z}_{n},\tilde{H}_{n})$
is equal to the distribution of $(Z_{n},H_{n})$ on $\mathbb{R}_+$,
which yields that
\begin{equation}\label{distrbution2}
\tilde{{\rm E}}\phi(\tilde{Z}_{n})={\rm E}\phi(Z_{n}),\; \tilde{{\rm
E}}\phi(\tilde{H}_{n})={\rm E}\phi(H_{n}),\; \text{ for any }
\phi\in C_b(\mathbb{R}_+),
\end{equation}
where $C_b(X)$ is the space of continuous bounded functions defined
on $X$. Now, define $\phi_\varepsilon\in C_b(\mathbb{R}_+)$ by
\begin{equation*}
\phi_\varepsilon=
\begin{cases}
{y\over\varepsilon},\; 0\le y<\varepsilon;\\
1,\; y\ge \varepsilon.
\end{cases}
\end{equation*}
One can check that
\begin{equation*}
\begin{split}
\tilde{\textrm{P}}(\tilde{Z}_{n}\ge
\varepsilon)&=\int_{\tilde{\Omega}}
1_{[\varepsilon,\infty]}\tilde{Z}_{n}d\tilde{\textrm{P}}\le
\int_{\tilde{\Omega}}
1_{[0,\varepsilon]}{\tilde{Z}_{n}\over\varepsilon}d\tilde{\textrm{P}}+\int_{\tilde{\Omega}}
1_{[\varepsilon,\infty]}\tilde{Z}_{n}d\tilde{\textrm{P}},\\
\tilde{\textrm{P}}(\tilde{H}_{n}\ge
\varepsilon)&=\int_{\tilde{\Omega}}
1_{[\varepsilon,\infty]}\tilde{H}_{n}d\tilde{\textrm{P}}\le
\int_{\tilde{\Omega}}
1_{[0,\varepsilon]}{\tilde{H}_{n}\over\varepsilon}d\tilde{\textrm{P}}+\int_{\tilde{\Omega}}
1_{[\varepsilon,\infty]}\tilde{H}_{n}d\tilde{\textrm{P}},
\end{split}
\end{equation*}
Hence by the definition of $\tilde{{\rm
E}}\phi_\varepsilon(\tilde{Z}_{n})$ and $\tilde{{\rm
E}}\phi_\varepsilon(\tilde{H}_{n})$, we can infer that
\begin{equation*}
\tilde{\textrm{P}}(\tilde{Z}_{n}\ge \varepsilon)\le \tilde{{\rm
E}}\phi_\varepsilon(\tilde{Z}_{n}),\;\tilde{\textrm{P}}(\tilde{H}_{n}\ge
\varepsilon)\le \tilde{{\rm E}}\phi_\varepsilon(\tilde{H}_{n}),
\end{equation*}
which together with \eqref{distrbution2} imply that
\begin{equation*}
\tilde{\textrm{P}}(\tilde{Z}_{n}\ge \varepsilon)\le {\rm
E}\phi_\varepsilon(Z_{n}),\; \tilde{\textrm{P}}(\tilde{H}_{n}\ge
\varepsilon)\le {\rm E}\phi_\varepsilon(H_{n}).
\end{equation*}
From \eqref{GZH}, it holds that
\begin{equation}\label{distrbution3}
\tilde{\textrm{P}}(\tilde{Z}_{n}\ge \varepsilon)\le {\rm
E}\phi_\varepsilon(Z_{n})=0,\; \tilde{\textrm{P}}(\tilde{H}_{n}\ge
\varepsilon)\le {\rm E}\phi_\varepsilon(H_{n})=0.
\end{equation}
Since $\varepsilon>0$ is arbitrary, from \eqref{distrbution3}, we
can infer that
\begin{equation*}\label{distrbution4}
\tilde{Z}_{n}=0,\; \tilde{H}_{n}=0  \ \
\tilde{\textrm{P}}-\mbox{a.s.}
\end{equation*}
Similarly, we can prove that $\tilde{G}_{n}=0,\ \
\tilde{\textrm{P}}-\mbox{a.s.}. $
\end{proof}

\subsection{Passage to the limit for $n\to\infty$}\label{s3.4}

By Proposition \ref{prop3.8}, we know that
$(\tilde{\rho}_{n},\tilde{u}_{n},\tilde{\rho}_{n}\tilde{u}_{n},\tilde{B}_{n})$
satisfies the same estimates as $(\rho_n,u_n,\rho_nu_n,B_n)$. In
this subsection, to simplify notations,  we will drop the tildes on
the random variables. In order to obtain a solution of the problem
\eqref{A-MHD}-\eqref{A-initialdata}, we will use the estimates
obtained in Lemma \ref{lemma3.3} to pass to the limit for
$n\to\infty$ in the sequence $(\rho_n, u_n, B_n)$ in the sample
space. For this, we need the following proposition and lemma in
\cite[Chapter 3]{KO}.

\begin{proposition}[Uniformly integrablility]\label{Propo2}
If there exists a nonnegative  measurable function $f$ in
$\mathbb{R}^+$, such that $\lim\limits_{x\to\infty} {f(x)\over
x}=\infty$ and $\sup_{t\in T}{\rm E}[f(|X_t|)]<\infty$. Then $X_t$ is a
set of uniformly integrable.
\end{proposition}

\begin{lemma}[Vitali's convergence Theorem]\label{Lemma2}
Suppose $p\in[1,\infty)$, $X_n\in L^p$ and $X_n$ converges to $X$ in
probability. Then the following are equivalent:

(1) $X_n\stackrel{L^p}{\to}X$;

(2) $|X_n|^p$ is uniformly integrable;

(3) ${\rm E}(|X_n|^p)\to {\rm E}(|X|^p)$.
\end{lemma}

The estimate \eqref{energy} along with the fact that
$\partial_t\rho_n$ satisfies \eqref{A-MHD} makes it possible to use
\eqref{tightness} to infer that there exists a function $\rho$ such
that
\begin{align}\label{convergence1}
\rho_n\to\rho \ \ \mbox{in}\ \ L^4((0,T)\times D) \ \
\tilde{\textrm{P}}-\text{a.s.}.
\end{align}
In fact, \eqref{tightness} implies that $\rho_n\to \rho$ in
$L^2((0,T)\times D)$ $\tilde{\textrm{P}}$--a.s.. Since $\rho_n\in
L^\infty(0,T; L^\beta(D))$, $\beta>4$, using the interpolation of
$L^2((0,T)\times D)$ and $L^\infty(0,T;L^\beta(D))$,   we have
\eqref{convergence1}.

 On the other hand, similar to
\eqref{convergence1}, in view of \eqref{energy}, and $\beta>\gamma$,
one has
\begin{align*}
\rho_n^\gamma\to\rho^\gamma,\; \rho_n^\beta \to \rho^\beta \ \
\mbox{in}\ \ L^1((0,T)\times D)\;  \ \
\tilde{\textrm{P}}-\text{a.s.},
\end{align*}
which together with integration by parts yields that for $\phi\in
C_0^\infty([0,T]\times D)$
\begin{align}\label{convergence4}
\langle \nabla \rho_n^\gamma,\phi\rangle
\to\langle\nabla\rho^\gamma,\phi\rangle,\;\langle \nabla
\rho_n^\beta,\phi\rangle \to\langle\nabla\rho^\beta,\phi\rangle \ \
\tilde{\textrm{P}}-\text{a.s. as  $n\to \infty$}.
\end{align}

Now, we pass to the limit in the products $\rho_nu_n$ and
$\rho_nu_n\otimes u_n$. From \eqref{tightness}, we know that
\begin{align*}
\rho_nu_n\rightharpoonup h \ \ \text{in $C([0,T];
L^{2\beta\over\beta+1}(D))$ $\tilde{\textrm{P}}$}-\text{a.s.}.
\end{align*}
Then, by \eqref{tightness}($\rho_n\to\rho \ \ \text{in $L^2(0,T;
L^2(D))$}$, $u_n\rightharpoonup u \ \ \text{in $L^2(0,T; H^1(D))$
$\tilde{\textrm{P}}$}-\text{a.s.}$), one has
\begin{align*}
\rho_n u_n\rightharpoonup \rho u \ \ \text{in $L^1(0,T; L^1(D))$
$\tilde{\textrm{P}}$}-\text{a.s.}.
\end{align*}
which together with \cite[Lemma 2.4]{FNP} yields that $h=\rho u$.
Then $\rho_nu_n\rightharpoonup \rho u$ in $C([0,T];
L^{2\beta\over\beta+1}(D))~ \tilde{\textrm{P}}$-a.s.. Since
$L^{2\beta\over \beta+1}(D)\hookrightarrow\hookrightarrow
H^{-1}(D)$, by Aubin-Lions Lemma, one has
\begin{align}\label{P-a.s.6}
\rho_nu_n\to \rho u \ \ \text{in $C([0,T]; H^{-1}(D))$
$\tilde{\textrm{P}}$}-\text{a.s.}.
\end{align}
Therefore, from \eqref{tightness}($u_n\rightharpoonup u \ \ \text{in
$L^2(0,T; H^1(D))$ $\tilde{\textrm{P}}$}-\text{a.s.}$) and
\eqref{P-a.s.6}, we can deduce that
\begin{align}\label{P-a.s.6a}
\rho_nu_n\otimes u_n\to \rho u\otimes u \ \ \text{in
$\mathcal{D}^\prime([0,T]\times D)\ \
\tilde{\textrm{P}}$}-\text{a.s.}.
\end{align}

Now, we turn to the magnetic terms. By \eqref{tightness}, we know
that
\begin{align}\label{P-a.s.7}
B_n\to B \ \ \mbox{in $L^2(0,T; L^2(D))$ and } B_n\rightharpoonup B
\ \ \mbox{in $L^2(0,T; H^1(D))$} \ \ \tilde{\textrm{P}}-\text{a.s.}.
\end{align}
Moreover, from $\eqref{A-MHD}_3$ and \eqref{energy}, by Aubin-Lions
Lemma, one has
\begin{align*}
B_n\to B \ \ \mbox{in} \ \ C([0,T]; H^{-1}(D))\ \
\tilde{\textrm{P}}-\text{a.s.}.
\end{align*}
Then we have
\begin{align}\label{P-a.s.8}
(\nabla\times B_n)\times B_n\to(\nabla\times B)\times B \ \ \mbox{in
$\mathcal{D}^\prime([0,T]\times D)$} \ \
\tilde{\textrm{P}}-\text{a.s.}.
\end{align}
Similarly, it follows from \eqref{tightness}($u_n\rightharpoonup u \
\ \text{in
$L^2(0,T; H^1(D))$ $\tilde{\textrm{P}}$}-\text{a.s.}$) 
and \eqref{P-a.s.7} that
\begin{align}\label{P-a.s.9}
\nabla\times(u_n\times B_n)\to\nabla\times (u\times B) \ \ \mbox{in
$\mathcal{D}^\prime([0,T]\times D)$} \ \
\tilde{\textrm{P}}-\text{a.s.}.
\end{align}

Finally, we will treat the force terms. We claim that
\begin{equation}\label{P-a.s.11a}
\begin{split}
\langle f^n_k(\rho_n,\rho_nu_n,x),\proj\phi\rangle\to
\langle f_k(\rho,\rho u,x),\phi\rangle  \ \ \mbox{in} \ \ L^1([0,T])\ \ \tilde{\textrm{P}}-\text{a.s.},\\
\langle \sum_{k\ge1}f^n_k(\rho_n,\rho_nu_n,x),\proj\phi\rangle^2\to
\langle \sum_{k\ge1}f_k(\rho,\rho u,x),\phi\rangle^2 \ \ \mbox{in} \
\ L^1([0,T])\ \ \tilde{\textrm{P}}-\text{a.s.}.
\end{split}
\end{equation}
\begin{proof}[Proof of the claim] By definition and the symmetry of $\mathcal M[\rho_n]$ we
have 
\begin{align*}
\big\langle f_k^n(\rho_n,\rho_n
u_n,x),\phi\big\rangle=\bigg\langle\mathcal
M^{\frac{1}{2}}[\rho_n]\mathbb{P} \Big(\frac{f_k(\rho_n,\rho_n
u_n,x)}{\sqrt{\rho_n}}\Big),\proj\phi\bigg\rangle=\bigg\langle\mathbb{P}
\Big(\frac{f_k(\rho_n,\rho_n u_n,x)}{\sqrt{\rho_n}}\Big),\mathcal
M^{\frac{1}{2}}[\rho_n]\proj\phi\bigg\rangle.
\end{align*}
By using  \eqref{assumption}, \eqref{energy} and the strong
convergence in \eqref{tightness}, \eqref{convergence1},
\eqref{P-a.s.6},  we have
\begin{align}\label{ps13}
\frac{f_k(\rho_n,\rho_n u_n,x)}{\sqrt{\rho_n}}\rightarrow
\frac{f_k(\rho,\rho u,x)}{\sqrt{\rho}}\quad\text{in}\quad L^2(D)\ \
\tilde{\textrm{P}}\otimes\tilde{\textrm{L}}-\text{a.e.},
\end{align}
where $\tilde{\textrm{L}}$ is the Lebesgue measure in time.

Next we extend the operator $\mathcal{M}[\rho]$ to
\begin{equation*}
\tilde{\mathcal{M}}[\rho]: H^1(D) \to L^2(D), \quad \LA \tilde{\mathcal{M}[\rho]}v, w \RA = \LA \mathcal{M}[\rho]\proj v, \proj w \RA
\end{equation*}
for $v \in H^1(D), \ w \in L^2(D)$. It is easy to see that as $n\to \infty$,
\begin{align*}
\left| \LA \tilde{\mathcal{M}}[\rho_n]v, w \RA - \LA \rho v, w \RA \right| & \le \|\rho_n\|_{L^\infty} \| v - \proj v \|_{L^2} \|w - \proj w\|_{L^2} + \|\rho_n\|_{L^\infty}\| v \|_{L^2} \|w - \proj w\|_{L^2} \\
& \quad + \| \rho_n - \rho\|_{L^4} \|v\|_{L^4} \|w\|_{L^2} + \| \rho_n \|_{L^\infty} \| w \|_{L^2} \|v  - \proj v\|_{L^2} \to 0,
\end{align*}
which implies that
\begin{align*}
\LA \tilde{\mathcal{M}}[\rho_n]v, \cdot \RA \rightarrow \langle\rho v,\cdot\rangle_2\
\ \tilde{\textrm{P}}\otimes\tilde{\textrm{L}}-\text{a.e}.,
\end{align*}



Next we look at $\tilde {\mathcal M}^{\frac{1}{2}}[\rho_n]$. Note that $\tilde{\mathcal{M}}{\rho_n}$ and $\rho$ are both symmetric positive definite operators with a lower bound (independent of $n$) strictly away from 0, and hence they are both invertible. Moreover we know that they also commute. This way we have
\[
\tilde{\mathcal M}^{\frac{1}{2}}[\rho_n] - \sqrt{\rho} = \left( \tilde{\mathcal M}^{\frac{1}{2}}[\rho_n] + \sqrt{\rho} \right)^{-1} \left( \tilde{\mathcal M}[\rho_n] - {\rho}  \right),
\]
and hence
\begin{align}\label{ps14}
\LA \tilde {\mathcal M}^{\frac{1}{2}}[\rho_n]v, \cdot \RA \rightarrow
\big\langle\sqrt{\rho}\,v,\cdot\big\rangle_2\ \
\tilde{\textrm{P}}\otimes\tilde{\textrm{L}}-\text{a.e}..
\end{align}
This way $\eqref{P-a.s.11a}_1$ follows from \eqref{ps13} and
\eqref{ps14}.

Next, we prove $\eqref{P-a.s.11a}_2$. Direct computation yields that
\begin{align*}
&\left|\sum_{k\ge1}\langle
f^n_k(\rho_n,\rho_nu_n,x),\proj\phi\rangle^2-\sum_{k\ge1}\langle
f_k(\rho_n,\rho_n
u_n,x),\phi\rangle^2\right|\\
\le & \sum_{k\ge1} \Big| \langle
f^n_k(\rho_n,\rho_nu_n,x),\proj\phi\rangle - \langle
f_k(\rho_n,\rho_n u_n,x),\phi\rangle \Big| \Big| \langle
f^n_k(\rho_n,\rho_nu_n,x),\proj\phi\rangle + \langle
f_k(\rho_n,\rho_n u_n,x),\phi\rangle \Big|.
\end{align*}

By \eqref{assumption} and Lemma \ref{lemma3.3}, we can bound
\begin{align*}
\sum_{k\ge1} \Big| \langle
f_k(\rho_n,\rho_n u_n,x),\phi\rangle \Big| \lesssim \sum_{k\ge1}\|f_k\|_{L^2} \lesssim \|\rho_n\|^{(\gamma+1)/2}_{L^{\gamma+1}} + \|\rho_n u_n\|_{L^2} \le C.
\end{align*}
Similarly we can bound $\sum_{k\ge1} \Big| \langle
f^n_k(\rho_n,\rho_n u_n,x),\proj\phi\rangle \Big|$. Then from
$\eqref{P-a.s.11a}_1$ we get $\eqref{P-a.s.11a}_2$.
\end{proof}

Thanks to \eqref{P-a.s.7} and \eqref{assumption}, we can similarly obtain
\begin{align}\label{P-a.s.12}
\begin{split}
\left\langle g_k(B_n,x),\phi\rangle\to \langle
g_k(B,x),\phi\right\rangle \ \ \mbox{in} \ \ L^1([0,T])\ \
\tilde{\textrm{P}}-\text{a.s.},\\
\langle \sum_{k\ge1}g_k(B_n,x),\phi\rangle^2\to \langle
\sum_{k\ge1}g_k(B,x),\phi\rangle^2 \ \ \mbox{in} \ \ L^1([0,T])\ \
\tilde{\textrm{P}}-\text{a.s.}.
\end{split}
\end{align}

Consequently, by \eqref{tightness},
\eqref{convergence1}-\eqref{P-a.s.9}, Lemma \ref{lemma3.3},
Proposition \ref{Propo2} and Lemma \ref{Lemma2}, we can pass to the
limit in the continuity equation $\eqref{A-MHD}_1$ and have the
following propositions.
\begin{proposition}\label{pro3.4}
$((\tilde{\Omega}, \tilde{\mathscr{F}},\tilde{\textrm{P}}), \beta_k,
\rho, u, B)$ is a weak solution of $\eqref{A-MHD}_1$. Moreover,
$\nabla \rho_n\cdot\nabla u_n\to \nabla\rho\cdot\nabla u$ in
$\mathscr{D}^\prime((0,T)\times D)$, $\tilde{\textrm{P}}$-a.s..
\end{proposition}
\begin{proof} For all $t\in[0,T]$ and any $\varphi\in C_0^\infty(D)$, denote
\begin{align*}
m_n(t)=\langle \rho_n(t),\varphi\rangle- \langle
\rho_0,\varphi\rangle+\int_0^t\langle \rho_nu_n,\nabla\varphi\rangle
ds- \int_0^t\langle \varepsilon\nabla\rho_n,\nabla\varphi\rangle ds.
\end{align*}
By Lemma \ref{lemma3.3}, we know that $\sup_{0\le t\le T}\tilde{{\rm
E}}(|m_n(t)|^2)\le C$. Then $m_n(t)$ is uniformly integrable. By
\eqref{tightness}, \eqref{P-a.s.6}, and Lemma \ref{lemma3.3} we know
that $m_n(t) \to m(t)$ $\tilde{\textrm{P}}$-a.s.. Hence applying
Proposition \ref{Propo2} and Lemma \ref{Lemma2}, we have
$\tilde{{\rm E}}(|m(t)|^2)=\lim\limits_{n\to\infty} \tilde{{\rm
E}}(|m_n(t)|^2)=0$.

Next, we shall show the convergence of the term
$\nabla\rho_n\cdot\nabla u_n$. Multiplying $\eqref{A-MHD}_1$ by
$\rho_n$, integrating by parts, we obtain $\tilde{\textrm{P}}$-a.s.
\begin{equation}\label{mass1}
\|\rho_n(t)\|^2_{L^2(D)}+2\varepsilon\!\! \int_0^t\!\!
\|\nabla\rho_n\|^2_{L^2(D)}dt= -\!\!\int_0^t\!\!\int_D \rho_n^2\Dv
u_ndxdt+\|\rho_0\|^2_{L^2(D)}.
\end{equation}
Sending $n\to\infty$ in equation $\eqref{A-MHD}_1$ and then
multiplying the result by $\rho$ (since from \cite[Lemma 2.4]{FNP}
we know $\rho$ enjoys enough regularity due to parabolic smoothing),
one has $\tilde{\textrm{P}}$-a.s.,
\begin{equation}\label{mass2}
\|\rho(t)\|^2_{L^2(D)}+2\varepsilon \int_0^t
\|\nabla\rho\|^2_{L^2(D)}dt=-\int_0^t\int_D\rho^2\Dv
udxdt+\|\rho_0\|^2_{L^2(D)}.
\end{equation}
From \cite[Lemma 2.4]{FNP},  \eqref{tightness},
\eqref{convergence1}, \eqref{mass1}, and \eqref{mass2}, for any $t$,
we deduce that $\tilde{\textrm{P}}$-a.s.
\begin{align*}
\|\nabla \rho_n\|^2_{L^2((0,T)\times D)}\to
\|\nabla\rho\|^2_{L^2((0,T)\times D)},\quad
\|\rho_n(t)\|^2_{L^2(D)}\to \|\rho(t)\|^2_{L^2(D)},
\end{align*}
which together with \eqref{tightness} imply $\nabla \rho_n\to\nabla
\rho$ in $L^2(0,T; L^2(D))$\;$\tilde{\textrm{P}}$-a.s.. Then, by the
fact $u_n\rightharpoonup u \in{L^2(0,T;
H^1(D))}$\;$\tilde{\textrm{P}}$-a.s.(\eqref{tightness}), one has
\begin{align*}
\nabla \rho_n\cdot\nabla u_n^i\to \nabla\rho\cdot\nabla u^i\ \
\mbox{in}\ \ \mathcal{D}^\prime((0,T)\times D) \ \
\tilde{\textrm{P}}-\mbox{a.s.}, i=1,2,3.
\end{align*}
\end{proof}
\begin{proposition}
The system $((\tilde{\Omega},
\tilde{\mathscr{F}},\tilde{\textrm{P}}),\beta_k,\rho, u, B)$ is a
martingale solution of $\eqref{A-MHD}_2$ and $\eqref{A-MHD}_3$.
\end{proposition}
\begin{proof}It follows from
Proposition \ref{prop3.8} that
\begin{align*}
\langle
M_n(t),\phi\rangle&=\langle\rho_nu_n(t)-\rho_nu_n(0),\phi\rangle+\int_0^t\langle
\Dv(\rho_nu_n\otimes u_n)+
\delta\nabla\rho^\beta_n-\mu\Delta u_n ,\phi\rangle ds\\
&\quad +\int_0^t\langle(\nabla\times B_n)\times
B_n+\varepsilon\nabla \rho_n\cdot \nabla u_n-(\lambda+\mu)\nabla\Dv
u_n+a\nabla\rho^\gamma_n,\phi\rangle ds\\
&:=\langle\rho_nu_n(t)-\rho_nu_n(0),\phi\rangle+\int_0^t \langle \Lambda_n,\phi\rangle ds,\\
\langle \tilde{M}_n(t),\phi\rangle&=\langle B_n(t)-B_n(0),\phi\rangle+\int_0^t\langle \nabla\times(u_n\times B_n)
+\nu\Delta B_n,\phi\rangle ds\\
&:=\langle B_n(t)-B_n(0),\phi\rangle+\int_0^t\langle
\tilde{\Lambda}_n,\phi\rangle ds,
\end{align*}
where $\langle M_n(t),\phi\rangle=\int_0^t\sum_{k\ge1}\langle
f_k(\rho_n,\rho_nu_n,x), \phi\rangle d\beta_{k,n}^1$ and $\langle
\tilde{M}_n(t),\phi\rangle=\int_0^t\sum_{k\ge1}\langle
g_k(B_n,x),\phi\rangle d\beta_{k,n}^2$ are martingales under
$\tilde{\textrm{P}}$. Therefore, with $0\le s< t<\infty$ and a
bounded, continuous $\mathscr{F}_s$-measurable function
$\varphi_n=\varphi(\beta_{k,n},\rho_n,u_n,B_n)$, we have
\begin{align}\label{3.88c}
\begin{split}
&\tilde{{\rm E}}\left[\langle
M_n(t)-M_n(s),\phi\rangle\varphi_n\right]=0,\; \tilde{{\rm
E}}\left[\langle
\tilde{M}_n(t)-\tilde{M}_n(s),\phi\rangle\varphi_n\right]=0,\\
&\tilde{{\rm E}}\left[\left(\langle M_n(t)-M_n(s),\phi\rangle
\beta_{k,n}^1-\int_s^t\langle
f_k^n(\rho_n,\rho_nu_n,x),\phi\rangle d\tau\right)\varphi_n\right]=0,\\
&\tilde{{\rm E}}\left[\left(\langle
\tilde{M}_n(t)-\tilde{M}_n(s),\phi\rangle\beta_{k,n}^2-\int_s^t\langle
g_k(B_n,x),\phi\rangle d\tau\right)\varphi_n\right]=0,\\
&\tilde{{\rm E}}\left[\left(\langle M_n(t),\phi\rangle^2-\langle
M_n(s),\phi\rangle^2-\int_s^t\sum_{k\ge1}\langle
f_k^n(\rho_n,\rho_nu_n,x),\phi\rangle^2d\tau\right)\varphi_n\right]=0,\\
&\tilde{{\rm E}}\left[\left(\langle
\tilde{M}_n(t),\phi\rangle^2-\langle
\tilde{M}_n(s),\phi\rangle^2-\int_s^t\sum_{k\ge1}\langle
g_k(B_n,x),\phi\rangle^2d\tau\right)\varphi_n\right]=0.
\end{split}
\end{align}
By \eqref{tightness}, \eqref{convergence1}-\eqref{P-a.s.9}, one
deduces that
 $$\langle\rho_nu_n(t)-\rho_nu_n(s),\phi\rangle+
\int_s^t\langle\Lambda_n,\phi\rangle d\tau,\; \langle
B_n(t)-B_n(s),\phi\rangle+\int_s^t\langle
\tilde{\Lambda}_n,\phi\rangle d\tau$$ converges to
$$
\langle\rho u(t)-\rho u(s),\phi\rangle-
\int_s^t\langle\Lambda,\phi\rangle d\tau,\; \langle
B(t)-B(s),\phi\rangle+\int_s^t\langle \tilde{\Lambda},\phi\rangle
d\tau \ \ \mbox{almost surely on} \ \ \tilde{\Omega},\; $$ where
\begin{align*}
&\langle\Lambda,\phi\rangle=\langle \Dv(\rho u\otimes u)-\mu\Delta
u-(\lambda+\mu)\Dv\nabla
u+a\nabla\rho^\gamma+\delta\nabla\rho^\beta+\varepsilon\nabla
u\cdot\nabla\rho+(\nabla\times B)\times B,\phi\rangle,\\
&\langle \tilde{\Lambda},\phi\rangle=\langle \nabla\times(u\times
B)+\nu\Delta B,\phi\rangle.
\end{align*}
 Next, we prove that the functions
$\overline{f_n}:=\langle
M_n(t)-M_n(s),\phi\rangle\varphi(\beta_{k,n},\rho_n,u_n,B_n)$ are
uniformly integrable. Indeed, by the Schwarz inequality and the
Burkholder-Davis-Gundy inequality, for each $n$, we have
\begin{align}\label{I1}
\notag\tilde{{\rm E}}\left(|\overline{f_n}|^2\right)&\lesssim
\tilde{{\rm E}}\left(\|M_n(t)\|^2_{W^{-2,2}(D)}+\|M_n(s)\|^2_{W^{-2,2}(D)}\right)\\
&\lesssim
\tilde{{\rm E}}\left(\int_0^T\sum_{k=1}\|f^n_k(\rho_n,\rho_nu_n,x)\|^2_{L^1(D)}dt\right)\\
&\notag\lesssim
\tilde{{\rm E}}\left(\int_0^T\|\rho_n\|_{L^1(D)}(\|\rho_n\|^\gamma_{L^\gamma(D)}
+\|\sqrt{\rho_n}u_n\|^2_{L^2(D)})dt\right)\le
C.
\end{align}
Then the sequence $\overline{f_n}$ is uniformly integrable. For some
$r>1$, we obtain as \eqref{I1}
\begin{align}\label{I2}
\tilde{{\rm
E}}\left(\int_0^T\sum_{k=1}\|f^n_k(\rho_n,\rho_nu_n,x)\|^2_{L^1(D)}dt\right)^r
\lesssim C.
\end{align}
Similarly, we have
\begin{align}\label{I3}
\tilde{{\rm E}}\left(\int_0^T\sum_{k=1}\|g(B_n,x)\|^2_{H^{-1}(D)}dt\right)^r\le
C.
\end{align}
In view of \eqref{P-a.s.11a}, \eqref{P-a.s.12}, \eqref{3.88c},
\eqref{I2}-\eqref{I3},
 together with Proposition \ref{Propo2} and  Lemma
\ref{Lemma2}, one has
\begin{align*}
&\tilde{{\rm E}}\left[\langle
M(t)-M(s),\phi\rangle\varphi(\beta_k,\rho,u,B)\right]
=\lim_{n\to\infty}\tilde{{\rm E}}\left[\langle
M_n(t)-M_n(s),\phi\rangle\varphi_n\right]=0,\\
&\tilde{{\rm E}}\left[\langle
\tilde{M}(t)-\tilde{M}(s),\phi\rangle\varphi(\beta_k,\rho,u,B)\right]
=\lim_{n\to\infty}\tilde{{\rm E}}\left[\langle
\tilde{M}_n(t)-\tilde{M}_n(s),\phi\rangle\varphi_n\right]=0,\\
&\tilde{{\rm E}}\left[\left(\langle M(t)-M(s),\phi\rangle
\beta_k^1-\int_s^t\langle
 f_k(\rho,\rho u,x),\phi\rangle d\tau\right)\varphi(\beta_k,\rho,u,B)\right]\\
&=\lim_{n\to\infty}\tilde{{\rm E}}\left[\left(\langle
M_n(t)-M_n(s),\phi\rangle \beta_{k,n}^1-\int_s^t\langle
 f^n_k(\rho_n,\rho_nu_n,x),\phi\rangle d\tau\right)\varphi_n\right]=0,\\
&\tilde{{\rm E}}\left[\left(\langle \tilde{M}(t)-\tilde{M}(s),\phi\rangle
\beta_k^2-\int_s^t\langle
 g_k(B,x),\phi\rangle d\tau\right)\varphi(\beta_k,\rho,u,B)\right]\\
&=\lim_{n\to\infty}\tilde{{\rm E}}\left[\left(\langle
\tilde{M}_n(t)-\tilde{M}_n(s),\phi\rangle
\beta_{k,n}^2-\int_s^t\langle
 g_k(B_n,x),\phi\rangle d\tau\right)\varphi_n\right]=0,
\end{align*}
and
\begin{align*}
&\tilde{{\rm E}}\left[\left(\langle M(t),\phi\rangle^2-\langle
M(s),\phi\rangle^2-\int_s^t\sum_{k=1}\langle f_k(\rho,\rho u,x),
\phi\rangle^2d\tau\right
)\varphi(\beta_k,\rho,u,B)\right]\\
&=\lim_{n\to\infty}\tilde{{\rm E}}\left[\left(\langle
M_n(t),\phi\rangle^2-\langle
M_n(s),\phi\rangle^2-\int_s^t\sum_{k\ge1}\langle
f^n_k(\rho_n,\rho_nu_n,x),
\phi\rangle^2d\tau\right )\varphi_n\right]=0,\\
&\tilde{{\rm E}}\left[\left(\langle
\tilde{M}(t),\phi\rangle^2-\langle
\tilde{M}(s),\phi\rangle^2-\int_s^t\sum_{k\ge1}\langle g_k(B,x),
\phi\rangle^2d\tau\right
)\varphi(\beta_k,\rho,u,B)\right]\\
&=\lim_{n\to\infty}\tilde{{\rm E}}\left[\left(\langle
\tilde{M}_n(t),\phi\rangle^2-\langle
\tilde{M}_n(s),\phi\rangle^2-\int_s^t\sum_{k=1}\langle g_k(B_n,x),
\phi\rangle^2d\tau\right )\varphi_n\right]=0.
\end{align*}
Then we can deduce that $\langle M(t),\phi\rangle$, $\langle
M(t)-M(s),\phi\rangle \beta_k^1-\int_s^t\langle
 f_k(\rho,\rho u,x),\phi\rangle d\tau$, $\langle
M(t),\phi\rangle^2-\int_0^t\!\sum_{k=1}\!\langle f_k(\rho,\rho
u,x),\phi\rangle^2ds$; $\langle \tilde{M}(t),\phi\rangle$, $\langle
\tilde{M}(t)\!-\tilde{M}(s),\phi\rangle \beta_k^2-\int_s^t\langle
 g_k(B,x),\phi\rangle d\tau$ and $\langle
\tilde{M}(t),\phi\rangle^2-\int_0^t\sum_{k=1}\langle
g_k(B,x),\phi\rangle^2ds$ are martingales. 

Applying the method in \cite{BO,OM}, we can infer that
$$\langle M(t),\phi\rangle=\int_0^t\langle f_k(\rho,\rho u,x),\phi \rangle d\beta_k^1,\;
\langle \tilde{M}(t),\phi\rangle=\int_0^t\langle g_k(B,x),\phi
\rangle d\beta_k^2 .$$ Therefore, $((\tilde{\Omega},
\tilde{\mathscr{F}},\tilde{\textrm{P}}), \beta_k,\rho,u,B)$ is a
martingale solution of \eqref{A-MHD}-\eqref{A-initialdata}.
\end{proof}
\section{The Vanishing Viscosity Limit}\label{section4}

In this section, we shall consider the limit of
\eqref{A-MHD}-\eqref{A-initialdata} as $\varepsilon\to 0$ following
the idea of \cite{f2,FNP,LLP}. Similar to Lemma \ref{lemma3.3}, we
can get the following estimates:
\begin{align}\label{4.1a}
\begin{split}
&\rho_\varepsilon\in L^p(\Omega,L^\infty(0,T;
L^\gamma(D)))\cap L^p(\Omega,L^\infty(0,T; L^\beta(D))),\\
& \sqrt{\rho_\varepsilon}u_\varepsilon\in L^p(\Omega,L^\infty(0,T;
L^2(D))),\;\rho_\varepsilon u_\varepsilon\in
L^p(\Omega,L^\infty(0,T; L^{2\beta\over \beta+1}(D))),\\
& u_\varepsilon\in L^p(\Omega,L^2(0,T; H^1(D))),\\
& \sqrt{\varepsilon}\nabla\rho_\varepsilon\in L^p(\Omega,L^2(0,T;
L^2(D))),\;\rho_\varepsilon u^i_\varepsilon u^j_\varepsilon\in
L^p(\Omega,L^2(0,T; L^{6\beta\over 4\beta+3}(D)))\\
& B_\varepsilon\in L^p(\Omega,L^2(0,T; H^1(D)))\cap
L^p(\Omega,L^\infty(0,T; L^2(D))),
\end{split}
\end{align}
for $1\le p < \infty$, $\beta > \max\{4, \gamma\}$.
 It follows from \eqref{4.1a} that $\rho_\varepsilon\in L^\infty(0,T;
L^\beta(D))$, i.e.,  $\rho_\varepsilon^\beta \in
L^\infty(0,T;L^1(D))$, and hence we can only conclude that
$\rho_\varepsilon^\beta$ converges to a Radon measure. To gain some
better convergence of this term we need to improve the integrability
of the density $\rho$ as follows.

\subsection{On the equation $\Dv v=f$}\label{s4.1}
We introduce a linear operator
$$\mathcal{B}:\left\{f\in L^p(D): \int_D fdx=0\right\}\rightarrow [W_0^{1,p}(D)]^3,$$  which is a
bounded linear operator, that is,
\begin{align}\label{4.1}
\|\mathcal{B}[f]\|_{W^{1,p}(D)}\le C(p)\|f\|_{L^p(D)},\quad
1<p<\infty.
\end{align}
and $\mathcal{B}(f)$ solves
\begin{align*}
\Dv \mathcal{B}(f)=f \ \ \mbox{in}\ \ D,~ \mathcal{B}(f)|_{\partial
D}=0.
\end{align*}
Moreover, if $g\in L^r(D)$ and $g\cdot \n|_{\partial D}=0$, then
\begin{equation*}\label{4.3}
\|\mathcal{B}[\Dv g]\|_{L^r(D)}\le C(p)\|g\|_{L^p(D)}\; \text{for
$1<p<\infty$.}
\end{equation*}
The operator $\mathcal{B}$ was first constructed by Bogovskii
\cite{B-1980}. The existence and the above properties of the
operator $\mathcal{B}$ have been proved in many papers, and we refer
the reader to \cite{G-1994} and the references therein for details.

\subsection{Uniform (on viscosity) estimates of density}\label{s4.2}

Next, we shall use the operator $\mathcal{B}$ to improve the
integrability of the density. First, let $\psi\in\mathcal{D}(0,T),
0\le \psi\le 1$ and $\bar{m}={1\over |D|}\int_D \rho(t)dx$, we note
that $\tilde{m}$ is conserved. Taking
$\psi(t)\mathcal{B}[\rho_\varepsilon-\bar{m}]$ as a test function
for $\eqref{A-MHD}_2$, integrating over $\Omega\times (0,T)\times
D$, we arrive at
\begin{align*}
\notag &{\rm
E}\int_0^T\int_D\left[\psi(t)\mathcal{B}[\rho_\varepsilon-\bar{m}](\rho_\varepsilon
u_\varepsilon)_t+\psi(t)\mathcal{B}[\rho_\varepsilon-\bar{m}]\Dv(\rho_\varepsilon
u_\varepsilon\otimes u_\varepsilon)\right]dxdt\\
\notag &\quad+a{\rm
E}\int_0^T\int_D\psi(t)\mathcal{B}[\rho_\varepsilon-\bar{m}]\nabla\rho_\varepsilon^\gamma
dxdt+\delta {\rm
E}\int_0^T\!\!\int_D\psi(t)\mathcal{B}[\rho_\varepsilon-\bar{m}]\nabla\rho_\varepsilon^\beta
dxdt\\
 &\quad+{\rm E}\int_0^T\!\!\int_D
\psi(t)\mathcal{B}[\rho_\varepsilon-\bar{m}][\varepsilon\nabla
\rho_\varepsilon\cdot\nabla u_\varepsilon-(\nabla\times B_\varepsilon)\times B_\varepsilon] dxdt\\
\notag &={\rm E}\int_0^T\!\!\int_D\left[\mu
\psi(t)\mathcal{B}_i[\rho_\varepsilon-\bar{m}]\Delta u_\varepsilon
+(\lambda+\mu)
\psi(t)\mathcal{B}[\rho_\varepsilon-\bar{m}]\nabla\Dv u_\varepsilon\right] dxdt\\
\notag &\quad+{\rm E}\left(\int_0^T\!\!\int_D
\psi(s)\mathcal{B}[\rho_\varepsilon-\bar{m}]
\sum_{k=1}f_k(\rho_\varepsilon,\rho_\varepsilon
u_\varepsilon,x)dxd\beta_{k,\varepsilon}^1\right).
\end{align*}
For convenience, we write the above equation as
\begin{equation*}
I_1+I_2+I_3+I_4+I_5+I_6=J_1+J_2+J_3.
\end{equation*}
We will just discuss the term $I_6$. Other terms can be handled as
in \cite{SSA,WW}.
For the term $I_6$, using H\"{o}lder's inequality and the fact that
$\beta>4$, from \eqref{4.1a} and \eqref{4.1}, we obtain
\begin{align*}
I_6&=-{\rm E}\int_0^T\psi\int_D \mathcal{B}[\rho_\varepsilon-\bar{m}](\nabla\times B_\varepsilon)\times B_\varepsilon dxdt\\
&\lesssim {\rm E}\left(\|B_\varepsilon\|_{L^2(0,T;L^2(D))}\|\nabla
B_\varepsilon\|_{L^2(0,T;L^2(D))}\|\mathcal{B}[\rho_\varepsilon-\bar{m}]\|_{L^\infty(0,T;L^\infty(D))}\right)\\
&\lesssim {\rm E}\left(\|\nabla
B_\varepsilon\|^2_{L^2(0,T;L^2(D))}\norm{\rho_\varepsilon}_{L^\infty(0,T;L^\beta(D))}\right)\\
&\lesssim \left({\rm E}\|\nabla
B_\varepsilon\|^4_{L^2(0,T;L^2(D))}\right)^{\frac{1}{2}} \left({\rm
E}\norm{\rho_\varepsilon}^2_{L^\infty(0,T;L^\beta(D))}\right)^{\frac{1}{2}}\le
C.
\end{align*}
Summing up the estimates for $I_1, I_2, \ldots,I_6$ and
$J_1,J_2,J_3$, we obtain the following results:
\begin{lemma}\label{lemma4.1}
Let $(\beta_{k,\varepsilon},\rho_\varepsilon, u_\varepsilon,
B_\varepsilon)$ be a weak solution of the problem
\eqref{A-MHD}-\eqref{A-initialdata}. Then
\begin{align}\label{4.6a}
{\rm E}\left(\|\rho_\varepsilon\|_{L^{\gamma+1}((0,T)\times
D)}+\|\rho_\varepsilon\|_{L^{\beta+1}((0,T)\times D)}\right)\le C,
\end{align}
where $C=C(\delta,\rho_{0,\delta},m_{0,\delta},B_{0,\delta})$ is a
constant independent of $\varepsilon$.
\end{lemma}

\subsection{Tightness property}\label{s4.3}
\begin{lemma}\label{Lemma-Tight1} Define
\begin{align*}
S&=C(0,T;\mathbb{R})\times C([0,T];L_w^{\beta}(D))\times
L^2(0,T;H_w^1(D))\\
&\quad\times C([0,T]; L_w^{2\beta\over\beta+1}(D))\times
\left(L^2(0,T;H_w^1(D))\cap L^2(0,T; L^2(D))\right)
\end{align*}
equipped with its Borel $\sigma$-algebra.  Let $\Pi_\varepsilon$ be
the probability on $S$ which is the image of $\textrm{P}$ on
$\Omega$ by the map: $\omega\mapsto
(\beta_{k,\varepsilon}(\omega,\cdot),\rho_\varepsilon(\omega,\cdot),u_\varepsilon(\omega,\cdot),\rho_\varepsilon
u_\varepsilon(\omega,\cdot),B_\varepsilon(\omega,\cdot))$, that is,
for any $A\subseteq S$,
\begin{equation*}
\Pi_\varepsilon(A)=\textrm{P}\left\{\omega\in\Omega:
(\beta_{k,\varepsilon}(\omega,\cdot),\rho_\varepsilon(\omega,\cdot),u_\varepsilon(\omega,\cdot),\rho_\varepsilon
u_\varepsilon(\omega,\cdot),B_\varepsilon(\omega,\cdot))\in A
\right\}.
\end{equation*}
Then the family $\Pi_\varepsilon$ is tight.
\end{lemma}
\begin{proof} Let $\Pi^1_\varepsilon$ be the probability on $$S_1=C(0,T;\mathbb{R})\times
C([0,T];L_w^{\beta}(D))\times L^2(0,T;H_w^1(D))\times
\left(L^2(0,T;H_w^1(D))\cap L^2(0,T; L^2(D))\right)$$ which is the
image of $\textrm{P}$ on $\Omega$ by the map $\omega\mapsto
(\beta_{k,\varepsilon}(\omega,\cdot),
\rho_\varepsilon(\omega,\cdot),u_\varepsilon(\omega,\cdot),
B_\varepsilon(\omega,\cdot) )$. Similar to Section \ref{s3.2}, we
can obtain the tightness of the family of $\Pi^1_\varepsilon$:
\begin{equation}\label{T1}
\Pi^1_\varepsilon(\Sigma_\eta\times X_\eta\times Y_\eta\times
R_\eta)\ge \left(1-{\eta\over5}\right)^4.
\end{equation}
Hence we only need to check the tightness of $\Pi^2_\varepsilon$,
which is defined as: for any $A\subseteq S_2=C([0,T];
L_w^{2\beta\over\beta+1}(D))$,
\begin{equation*}
\Pi^2_\varepsilon(A)=\textrm{P}\left(\omega\in\Omega:
\rho_\varepsilon u_\varepsilon(\omega,\cdot)\in A \right).
\end{equation*}  From \eqref{A-MHD}, we have
\begin{equation*}\label{Tight11}
\begin{split}
\rho_\varepsilon u_\varepsilon (t)=&\rho_0u_0-\int_0^T
\left[\Dv(\rho_\varepsilon u_\varepsilon \otimes u_\varepsilon
)-\mu\Delta u_\varepsilon -(\lambda+\mu)\Dv \nabla u_\varepsilon
+a\nabla
\rho_\varepsilon ^\gamma+\delta\nabla\rho_\varepsilon ^\beta\right]ds\\
&-\int_0^t[\varepsilon\nabla
u_\varepsilon\cdot\nabla\rho_\varepsilon- (\nabla\times
B_\varepsilon)\times B_\varepsilon]ds+\int_0^t
f_k(\rho_\varepsilon,\rho_\varepsilon
u_\varepsilon,x)d\beta^1_{k,\varepsilon}.
\end{split}
\end{equation*}
For $\phi\in C_0^\infty(D)$, denote $\Gamma_\varepsilon:=\langle
\rho_\varepsilon u_\varepsilon,
\phi\rangle+\varepsilon\int_0^t\langle \nabla u_\epsilon\cdot \nabla
\rho_\epsilon, \phi\rangle ds$. By virtue of \eqref{4.1a}, for some
$\alpha>1$, one has $\textrm{P}$-a.s.
\begin{align*}\label{e-continuous}
\notag\partial_t{\Gamma_\varepsilon}&=\langle \partial_t
(\rho_\varepsilon u_\varepsilon)+\varepsilon\nabla u_\epsilon\cdot
\nabla \rho_\epsilon, \phi\rangle\\
&=\langle -\Dv(\rho_\varepsilon u_\varepsilon\otimes
u_\varepsilon)-a\nabla \rho_\varepsilon^\gamma-\delta\nabla
\rho_\varepsilon^\beta
+\mu \Delta u_\varepsilon+(\lambda+\mu)\nabla\Dv u_\varepsilon, \phi\rangle\\
\notag&\quad+\langle(\nabla\times B_\varepsilon)\times
B_\varepsilon, \phi\rangle\in L^\alpha(0,T).
\end{align*}
On the other hand, $\Gamma_\varepsilon$ is uniformly bounded in
$L^\alpha(0,T)$. Hence $\Gamma_\varepsilon$ is equi-continuous in
$C[0,T]$.
 We can infer that
$\int_0^T\varepsilon\nabla u_\varepsilon\cdot\nabla\rho_\varepsilon
dt\to 0$ in $C([0,T];L^1(D))$ almost surely on $\Omega$. Since
$L^1(D)\hookrightarrow H^{-\ell}(D) (\ell\ge3)$ is compact, by
\cite[Lemma 3.9]{KO}, the distribution of $\int_0^T\varepsilon\nabla
u_\varepsilon \cdot\nabla\rho_\varepsilon dt$ is tight on
$H^{-\ell}(D)$. That is, for any $\eta>0$, there exists a compact
set $K\subset C([0,T]; H^{-\ell}(D))$ such that
$\textrm{P}\big(\omega\in\Omega: \int_0^T\varepsilon\nabla
u_\varepsilon\cdot\nabla\rho_\varepsilon dt \notin K\big)\le
{\eta\over10}$. Set
$$\mathcal{H}=L^\infty([0,T];L^{2\beta\over\beta+1}(D))\cap
C^{0,\alpha^\prime}([0,T];H^{-\ell}(D)), \; \ell\ge3, \;
0<\alpha^\prime<{1\over2}-{1\over 2\sigma},$$ with the norm
$\|f\|_{\mathcal{H}}=\sup_{0\le t\le
T}\|f(t)\|_{L^{2\beta\over\beta+1}(D)}+\|f(t)\|_{C^{0,\alpha^\prime}([0,T];H^{-\ell}(D))}$.

We choose $Z_\eta$ as a closed ball of radius $r_\eta$ centered at 0
in $\mathcal{H}\cap K$ . Similar to Section \ref{s3.2}, choosing
$r_\eta=10C\eta^{-1}$, we have
\begin{equation}\label{T2}
\begin{split}
\Pi^2_\varepsilon(Z_\eta)&=1-\Pi^2_\varepsilon\left(Z_\eta^c\right)
=1-\textrm{P}\left(\omega\in\Omega: \rho_\varepsilon
u_\varepsilon(\omega,\cdot)\in Z_\eta^c\right)\\
&=1-\textrm{P}\left(\omega\in\Omega:
\rho_\varepsilon u_\varepsilon(\omega,\cdot)\in \mathcal{H}^c\cup K^c\right)\\
&\ge 1-\textrm{P}\left(\omega\in\Omega: \rho_\varepsilon
u_\varepsilon(\omega,\cdot)\notin
\mathcal{H}\right)-\textrm{P}\left(\omega\in\Omega:
\int_0^T\varepsilon\nabla\rho_\varepsilon\cdot\nabla
u_\varepsilon(\omega,\cdot)dt\notin
K\right)\\
&\ge 1-{1\over r_\eta}{\rm E}\left(\|\rho_\varepsilon u_\varepsilon\|_{\mathcal{H}}\right)-{\eta\over10}\\
&\ge 1-{C\over r_\eta}-{\eta\over10}=1-{\eta\over5},
\end{split}
\end{equation}
where $Z^c$ is the complement of $Z$ in $\mathcal{H}\cap K$.
Therefore, \eqref{T1} and \eqref{T2} imply that
$$\Pi_\varepsilon(\Sigma_\eta\times X_\eta\times Y_\eta\times
Z_\eta\times R_\eta)=\Pi^1_\varepsilon(\Sigma_\eta\times
X_\eta\times Y_\eta\times R_\eta)\times\Pi^2_\varepsilon(Z_\eta)\ge
1-\eta.$$ Then $\Pi_\varepsilon$ is tight.\end{proof}

According to Jakubowski-Skorohod Theorem, there exists a subsequence
such that $\Pi_{\varepsilon}\to \Pi$ weakly, where $\Pi$ is a
probability on $S$. Moreover, there exist a probability space
$(\tilde{\Omega}, \tilde{\mathscr{F}},\tilde{\textrm{P}})$ and two
random variables
$(\tilde{\beta}_{k,\varepsilon},\tilde{\rho}_\varepsilon,\tilde{u}_\varepsilon,\tilde{\rho}_\varepsilon
\tilde{u}_\varepsilon,\tilde{B}_\varepsilon)$ with distribution
$\Pi_\varepsilon$, $\left(\beta_k, \rho, u, h, B\right)$ with values
in $S$ such that as $\varepsilon\to0$
\begin{equation}\label{2-tightness}
(\tilde{\beta}_{k,\varepsilon},\tilde{\rho}_\varepsilon,\tilde{u}_\varepsilon,\tilde{\rho}_\varepsilon
\tilde{u}_\varepsilon,\tilde{B}_\varepsilon)\to (\beta_k,\rho,u,h,B)
\ \ \mbox{in} \ \ S\ \ \tilde{\textrm{P}}-\text{a.s.}.
\end{equation}

\subsection{The vanishing viscosity limit}\label{s4.4}

In this subsection, we shall pass to the limit as $\varepsilon\to 0$
in \eqref{A-MHD}-\eqref{A-initialdata}. Note that the parameter
$\delta$ is kept fixed throughout this section and then we can use
the previously derived estimates. Similarly, from Proposition
\ref{prop3.8}, we can deduce that
$(\tilde{\rho}_\varepsilon,\tilde{u}_\varepsilon,\tilde{\rho}_\varepsilon\tilde{u}_\varepsilon,\tilde{B}_\varepsilon)$
satisfies the same estimates as
$(\rho_\varepsilon,u_\varepsilon,\rho_\varepsilon u_\varepsilon,
B_\varepsilon)$. For simplicity of notations, we will again drop the
tildes on the random variables.
It follows from \eqref{4.1a} that
\begin{equation}\label{4.20}
\varepsilon\nabla u_\varepsilon\cdot\nabla\rho_\varepsilon\to 0\ \
\mbox{in} \ \ L^1(\tilde{\Omega}\times(0,T)\times D).
\end{equation}
Similarly, we have $\varepsilon\Delta\rho_\varepsilon \to 0$ in
$L^2(\tilde{\Omega},L^2(0,T; H^{-1}(D))$. From \eqref{4.6a}, we have
\begin{equation}\label{4.6b}
a\rho_\varepsilon^\gamma+\delta\rho_\varepsilon^\beta \to
\overline{P} \ \ \mbox{weakly in} \ \
L^{\beta+1\over\beta}(\tilde{\Omega}\times(0,T)\times D).
\end{equation}
We can obtain as in Section \ref{s3.4}
\begin{align}\label{4.6c1}
\rho_\varepsilon\to\rho \ \ \mbox{in}\ \ C(0,T;H^{-1}(D)) \ \
\tilde{\textrm{P}}-\text{a.s.}.
\end{align}
Furthermore, \eqref{2-tightness} implies that
\begin{equation}\label{4.6c}
u_\varepsilon\rightharpoonup u \ \ \mbox{in}\ \ L^2(0,T;H^1(D)) \ \
\tilde{\textrm{P}}-\text{a.s.}.
\end{equation}
It follow from \eqref{4.6c} and \eqref{4.6c1} that
\begin{align*}
\rho_\varepsilon u_\varepsilon\to\rho u \ \ \mbox{in}\ \
\mathcal{D}^\prime((0,T)\times D) \ \
\tilde{\textrm{P}}-\text{a.s.}.
\end{align*}
%
%
Since $C_0^\infty(D)$ is density in $L^{2\gamma\over \gamma-1}(D)$,
we know that
\begin{equation*}\label{4.6d}
\rho_\varepsilon u_\varepsilon \to \rho u \ \ \mbox{in} \ \ C([0,T];
L_w^{2\gamma\over \gamma+1}(D))\ \ \tilde{\textrm{P}}-\text{a.s.},
\end{equation*}
which together with the fact that ${2\gamma\over \gamma+1}>{6\over
5}$, $L^{2\gamma\over \gamma+1}(D))\hookrightarrow\hookrightarrow
H^{-1}(D)$ and Aubin-Lions Lemma implies that
\begin{equation}\label{4.6d1}
\rho_\varepsilon u_\varepsilon\to\rho u \ \ \mbox{in} \ \ C([0,T];
H^{-1}(D))\ \ \tilde{\textrm{P}}-\text{a.s.}.
\end{equation}
Therefore, by \eqref{4.6c} and \eqref{4.6d1} one has
\begin{equation}\label{4.6e}
\rho_\varepsilon u^i_\varepsilon u^j_\varepsilon \to \rho u^i u^j \
\ \mbox{in} \ \ \mathcal{D}^\prime((0,T)\times D), i,j=1,2,3 \ \
\tilde{\textrm{P}}-\text{a.s.}.
\end{equation}
By \eqref{2-tightness}, we know that
\begin{align}\label{4.6e1}
B_\varepsilon\to B \ \ \mbox{in $L^2(0,T; L^2(D))$ and }
B_\varepsilon\rightharpoonup B \ \ \mbox{in $L^2(0,T; H^1(D))$} \ \
\tilde{\textrm{P}}-\text{a.s.}.
\end{align}
Moreover, from $\eqref{A-MHD}_3$ and \eqref{4.1}, by Aubin-Lions
Lemma, one has
\begin{align*}
B_\varepsilon\to B \ \ \mbox{in} \ \ C([0,T]; H^{-1}(D))\ \
\tilde{\textrm{P}}-\text{a.s.}.
\end{align*}
which together with  \eqref{4.6e1} yields
\begin{align}\label{4.6e2}
(\nabla\times B_\varepsilon)\times B_\varepsilon\to(\nabla\times
B)\times B \ \ \mbox{in $\mathcal{D}^\prime([0,T]\times D)$} \ \
\tilde{\textrm{P}}-\text{a.s.}.
\end{align}
Similarly, it follows from \eqref{4.6c} and \eqref{4.6e1} that
\begin{align}\label{4.6e3}
\nabla\times(u_\varepsilon\times B_\varepsilon)\to\nabla\times
(u\times B) \ \ \mbox{in $\mathcal{D}^\prime([0,T]\times D)$} \ \
\tilde{\textrm{P}}-\text{a.s.}.
\end{align}
By \eqref{assumption}, H\"{o}lder's inequality, \eqref{2-tightness},
\eqref{4.6d1}, \eqref{4.6e1}, we have
\begin{align}\label{4.26}
\begin{split}
\langle f_k(\rho_\varepsilon,\rho_\varepsilon
u_\varepsilon,x),\phi\rangle\to\langle f_k(\rho,\rho
u,x),\phi\rangle\ \ \mbox{in} \ \ L^1([0,T]) \ \
\tilde{\textrm{P}}-\text{a.s.},\\
 \langle
g_k(B_\varepsilon,x),\phi\rangle\to\langle g_k(B,x),\phi\rangle\ \
\mbox{in} \ \ L^1([0,T]) \ \ \tilde{\textrm{P}}-\text{a.s.}.
\end{split}
\end{align}

For $\phi\in C^\infty_0(D)$, we set $\langle
M_\varepsilon(t),\phi\rangle :=\int_0^t\sum_{k\ge1}\langle
f_k(\rho_\varepsilon,\rho_\varepsilon u_\varepsilon,x), \phi \rangle
d\beta^1_{k,\varepsilon}$ and $\langle
\tilde{M}_\varepsilon(t),\phi\rangle :=\int_0^t\sum_{k\ge1}\langle
g_k( B_\varepsilon,x),\phi\rangle d\beta^2_{k,\varepsilon}$, which
are martingales under $\tilde{\textrm{P}}$. It follows from
Proposition \ref{prop3.8} that
\begin{equation*}\label{3.88b}
\begin{split}
\langle M_\varepsilon(t),\phi\rangle&=\langle\rho_\varepsilon
u_\varepsilon(t)-\rho_\varepsilon
u_\varepsilon(0),\phi\rangle+\int_0^t\langle \Dv(\rho_\varepsilon
u_\varepsilon\otimes u_\varepsilon)+
\delta\nabla\rho^\beta_\varepsilon-\mu\Delta u_\varepsilon ,\phi\rangle ds\\
&\quad +\int_0^t\langle(\nabla\times B_\varepsilon)\times
B_\varepsilon+\varepsilon\nabla \rho_\varepsilon\cdot \nabla
u_\varepsilon-(\lambda+\mu)\nabla\Dv
u_\varepsilon+a\nabla\rho^\gamma_\varepsilon,\phi\rangle ds\\
&:=\langle\rho_\varepsilon u_\varepsilon(t)-\rho_\varepsilon
u_\varepsilon(0),\phi\rangle+\int_0^t \langle
\Lambda_\varepsilon,\phi\rangle ds,\\
\langle \tilde{M}_\varepsilon(t),\phi\rangle&=\langle
B_\varepsilon(t)-B_\varepsilon(0),\phi\rangle+\int_0^t\langle
\nabla\times(u_\varepsilon\times B_\varepsilon)+\nu\Delta B_\varepsilon,\phi\rangle ds\\
&:=\langle
B_\varepsilon(t)-B_\varepsilon(0),\phi\rangle+\int_0^t\langle
\tilde{\Lambda}_\varepsilon,\phi\rangle ds,
\end{split}
\end{equation*}
We define $\LA M(t), \phi \RA$ and $\LA \tilde M(t), \phi \RA$ by
dropping the subscript $\varepsilon$ on the right-hand side of the
above two equalities.

\begin{proposition}
The system $((\tilde{\Omega},
\tilde{\mathscr{F}},\tilde{\textrm{P}}),\beta_k, \rho, u, B)$ is a
martingale solution of the following equations:
\begin{align}\label{4.6}
\begin{cases}
d\rho+\Dv(\rho u)dt=0,\\
d(\rho u)+[\Dv(\rho u\otimes u)+\nabla \overline{P}-\mu \Delta
u+(\lambda+\mu)\nabla \Dv u-(\nabla\times B)\times B]dt=dM,\\
dB-[\nabla\times(u\times B)+\nu\Delta B]dt=d\tilde{M},
\end{cases}
\end{align}
in $\mathcal{D}^{\prime}((0,T)\times D)$ $\tilde{\textrm{P}}$-a.s..
\end{proposition}
\begin{proof}
Similar to Proposition \ref{pro3.4}, we can easily prove that
$((\tilde{\Omega}, \tilde{\mathscr{F}},\tilde{\textrm{P}}),\beta_k,
\rho, u, B)$ satisfies the equation $\eqref{4.6}_1$. Now, we only
need to prove that $((\tilde{\Omega},
\tilde{\mathscr{F}},\tilde{\textrm{P}}),\beta_k,\rho,u,B)$ satisfies
the equations $\eqref{4.6}_2$ and $\eqref{4.6}_3$.

For $0\le s< t<\infty$ and a bounded, continuous
$\mathscr{F}_s$-measurable function
$\varphi_\varepsilon=\varphi(\beta_{k,\varepsilon},\rho_\varepsilon,u_\varepsilon,B_\varepsilon)$,
we have
\begin{align}\label{4a}
\begin{split} &\tilde{{\rm E}}\left[\langle
M_\varepsilon(t)-M_\varepsilon(s),\phi\rangle\varphi_\varepsilon\right]=0,\quad
\tilde{{\rm E}}\left[\langle
\tilde{M}_\varepsilon(t)-\tilde{M}_\varepsilon(s),\phi\rangle\varphi_\varepsilon\right]=0,\\
&\tilde{{\rm E}}\left[\left(\langle
M_\varepsilon(t)-M_\varepsilon(s),\phi\rangle
\beta^1_{k,\varepsilon}-\int_s^t\langle
f_k(\rho_\varepsilon,\rho_\varepsilon u_\varepsilon,x),\phi\rangle d\tau\right)\varphi_\varepsilon\right]=0,\\
& \tilde{{\rm E}}\left[\left(\langle
\tilde{M}_\varepsilon(t)-\tilde{M}_\varepsilon(s),\phi\rangle
\beta^2_{k,\varepsilon}-\int_s^t\langle
 g_k(B_\varepsilon,x),\phi\rangle
 d\tau\right)\varphi_\varepsilon\right]=0.
\end{split}
\end{align}

By \eqref{2-tightness}-\eqref{4.6b}, \eqref{4.6c}-\eqref{4.26}, we
can infer that
 $$\notag
\langle\rho_\varepsilon u_\varepsilon(t)-\rho_\varepsilon
u_\varepsilon(s),\phi\rangle+
\int_s^t\langle\Lambda_\varepsilon,\phi\rangle d\tau,\; \langle
B_\varepsilon(t)-B_\varepsilon(s),\phi\rangle+\int_s^t\langle
\tilde{\Lambda}_\varepsilon,\phi\rangle d\tau$$ converges to
$$\notag \langle\rho u(t)-\rho u(s),\phi\rangle-
\int_s^t\langle\Lambda,\phi\rangle d\tau,\; \langle
B(t)-B(s),\phi\rangle+\int_s^t\langle \tilde{\Lambda},\phi\rangle
d\tau \ \ \mbox{almost surely on} \ \ \tilde{\Omega},$$ where
\begin{align*}
&\langle\Lambda,\phi\rangle=\langle \Dv(\rho u\otimes
u)-\mu\Delta u-(\lambda+\mu)\Dv\nabla u+(\nabla\times B)\times B+\overline{P},\phi\rangle\\
&\langle \tilde{\Lambda},\phi\rangle=\langle \nabla\times(u\times
B)+\nu\Delta B,\phi\rangle.
\end{align*}  We can prove that the functions
$f_\varepsilon:=\langle
M_\varepsilon(t)-M_\varepsilon(s),\phi\rangle\varphi(\beta_{k,\varepsilon},\rho_\varepsilon,u_\varepsilon,B_\varepsilon)$
are uniformly integrable as \eqref{I1}, \eqref{I2} and \eqref{I3}.
By the uniform integrability, \eqref{4a},
and Lemma \ref{Lemma2}, we have
\begin{align*}
&\tilde{{\rm E}}\left[\langle
M(t)-M(s),\phi\rangle\varphi(\beta_k,\rho,u,B)\right]=\lim_{\varepsilon\to0}\tilde{{\rm
E}}\left[\langle
M_\varepsilon(t)-M_\varepsilon(s),\phi\rangle\varphi_\varepsilon\right]=0,\\
&\tilde{{\rm E}}\left[\langle
\tilde{M}(t)-\tilde{M}(s),\phi\rangle\varphi(\beta_k,\rho,u,B)\right]
=\lim_{\varepsilon\to 0}\tilde{{\rm E}}\left[\langle
\tilde{M}_\varepsilon(t)-\tilde{M}_\varepsilon(s),\phi\rangle\varphi_\varepsilon\right]=0,\\
&\tilde{{\rm E}}\left[\left(\langle M(t)-M(s),\phi\rangle
\beta_k^1-\int_s^t\langle
 f_k(\rho,\rho u,x),\phi\rangle d\tau\right)\varphi(\beta_k,\rho,u,B)\right]\\
&=\lim_{\varepsilon\to 0}\tilde{{\rm E}}\left[\left(\langle
M_\varepsilon(t)-M_\varepsilon(s),\phi\rangle
\beta^1_{k,\varepsilon}-\int_s^t\langle
f_k(\rho_\varepsilon,\rho_\varepsilon u_\varepsilon,x),\phi\rangle d\tau\right)\varphi_\varepsilon\right]=0,\\
&\tilde{{\rm E}}\left[\left(\langle
\tilde{M}(t)-\tilde{M}(s),\phi\rangle \beta_k^2-\int_s^t\langle
 g_k( B,x),\phi\rangle d\tau\right)\varphi(\beta_k,\rho,u,B)\right]\\
&=\lim_{\varepsilon\to 0}\tilde{{\rm E}}\left[\left(\langle
\tilde{M}_\varepsilon(t)-\tilde{M}_\varepsilon(s),\phi\rangle
\beta^2_{k,\varepsilon}-\int_s^t\langle
 g_k( B_\varepsilon,x),\phi\rangle
 d\tau\right)\varphi_\varepsilon\right]=0.
\end{align*}
\end{proof}
Our next goal is  to prove
\begin{equation*}\label{4.9}
\overline{P}=\alpha\rho^\gamma+\delta\rho^\beta \quad
\text{($\overline{P}$ appears in the limit in \eqref{4.6b})}
\end{equation*}
which is equivalent to the strong convergence of $\rho_\varepsilon$
in $L^1(\tilde{\Omega}\times(0,T)\times D)$.

\subsection{The effective viscous flux}\label{s4.5}
We introduce the quantity
$a\rho^\gamma+\delta\rho^\beta-(\lambda+2\mu)\Dv u$ that is usually
called the effective viscous flux. This quantity satisfies many
properties for which we refer to \cite{FNP, LLP} for details. As in
\cite{SSA,WW}, we can obtain the following lemma:

\begin{lemma}\label{lemma4.2}
Assume that $\beta>\max\{4,\gamma,{4\gamma\over 2\gamma-3}\}$. Let
$(\rho_\varepsilon,u_\varepsilon,B_\varepsilon)$ be the sequence of
approximate solutions obtained in Lemma \ref{lemma3.3}, and let
$\rho, u$, $\overline{P}$ be the limits appearing in \eqref{4.6c1},
\eqref{4.6c},  and \eqref{4.6b} respectively. Then we have
\begin{equation*}
\begin{split}
\lim_{\varepsilon\to 0+}&\tilde{{\rm E}}\int_0^T\psi\int_D
\phi\left[a\rho_\varepsilon^\gamma+\delta\rho_\varepsilon^\beta-(\lambda+2\mu)\Dv
u_\varepsilon\right]\rho_\varepsilon dxdt\\
&=\tilde{{\rm
E}}\int_0^T\psi\int_D\phi\left[\overline{P}-(\lambda+2\mu)\Dv
u\right]\rho dxdt,
\end{split}
\end{equation*}
for any $\psi\in\mathcal{D}(0,T)$ and $\phi\in\mathcal{D}(D)$.
\end{lemma}

The idea of the proof of the above lemma is based on the Div-Curl
Lemma. First, we introduce the operator $\mathcal{A}$ by
\begin{equation*}
\mathcal{A}_i v=\Delta^{-1}[\partial_{x_i} v], \quad i=1,2,3.
\end{equation*}
Here $\Delta^{-1}$ stands for the inverse of the Laplace operator in
$\mathbb{R}^3$. Furthermore
\begin{equation*}
\mathcal{A}_j(\xi)={-i\xi_j\over|\xi|^2},\quad
\partial_{x_i}\mathcal{A}_i[v]=v.
\end{equation*}
The classical Mikhlin multiplier theorem implies that
\begin{equation*}\label{4.11}
\|\mathcal{A}_i v\|_{W^{1,s}(D)}\le
C(s,D)\|v\|_{L^s(\mathbb{R}^3)},\quad 1<s<\infty,
\end{equation*}
\begin{equation*}\label{4.12}
\|\mathcal{A}_i v\|_{L^q(D)}\le C(q,s,D) \|\mathcal{A}_i
v\|_{W^{1,s}(D)}\le C(q,s,D)\|v\|_{L^s(\mathbb{R}^3)}, \quad {1\over
q}\ge {1\over s}-{1\over 3},
\end{equation*}
and
\begin{equation*}\label{4.13}
\|\mathcal{A}_i v\|_{L^\infty(D)}\le
C(s,D)\|v\|_{L^s(\mathbb{R}^3)}, \quad s>3.
\end{equation*}

Now, extending $\rho_\varepsilon$ to be zero outside $D$. Consider
\begin{equation*}
\varphi(t,x)=\psi(t)\phi(x)\mathcal{A}[\rho_\varepsilon],\quad
\psi\in\mathcal{D}(0,T), \quad \phi\in\mathcal{D}(D).
\end{equation*}
Note that \eqref{4.1a} ensures that $\varphi$ can be used as a test
function for \eqref{A-MHD}. Note that
$\Dv\mathcal{A}[\rho_\varepsilon]=\rho_\varepsilon$ and $\Delta
\mathcal{A}_j=\partial_j$. Similar to Section \ref{s4.2}, we have
\begin{equation}\label{4.14}
\begin{split}
&\tilde{{\rm E}}\int_0^T\psi\int_D
\phi\left[a\rho_\varepsilon^\gamma+\delta\rho_\varepsilon^\beta-(\lambda+2\mu)\Dv
u_\varepsilon\right]\rho_\varepsilon dxdt\\
=&(\lambda+\mu)\tilde{{\rm E}}\int_0^T\psi\int_D \Dv u_\varepsilon
\partial_{x_i}\phi\mathcal{A}_i[\rho_\varepsilon]
dxdt-\tilde{{\rm
E}}\int_0^T\psi\int_D(a\rho_\varepsilon^\gamma+\delta\rho_\varepsilon^\beta)\partial_{x_i}\phi
\mathcal{A}_i[\rho_\varepsilon]dxdt\\
&+\mu \tilde{{\rm E}}\int_0^T\psi\int_D
\partial_{x_j}u_\varepsilon^i\partial_{x_j}\phi\mathcal{A}_i[\rho_\varepsilon]dxdt-\tilde{{\rm E}}\int_0^T\psi\int_D\rho_\varepsilon
u_\varepsilon^iu_\varepsilon^j\partial_{x_j}\phi\mathcal{A}_i[\rho_\varepsilon]dxdt\\
&-\tilde{{\rm E}}\int_0^T\psi_t\int_D\phi\rho_\varepsilon
u_\varepsilon^i \mathcal{A}_i[\rho_\varepsilon]dxdt-\varepsilon
\tilde{{\rm E}}\int_0^T\psi\int_D \phi\rho_\varepsilon
u_\varepsilon\mathcal{A}_i[\Dv(1_D\nabla\rho_\varepsilon)]dxdt\\
&+\varepsilon \tilde{{\rm E}}\int_0^T\psi\int_D
\phi\partial_{x_j}\rho_\varepsilon\partial_{x_j}u^i_\varepsilon\mathcal{A}_i[\rho_\varepsilon]dxdt+
\mu \tilde{{\rm E}}\int_0^T\psi\int_D
u_\varepsilon\partial_{x_i}\phi\rho_\varepsilon
dxdt\\
&-\mu \tilde{{\rm E}}\int_0^T\psi\int_D
u_\varepsilon^i\partial_{x_j}\phi\partial_{x_j}\mathcal{A}_i[\rho_\varepsilon]dxdt
+ \tilde{{\rm E}} \int_0^T\psi\int_D
\phi\mathcal{A}[\rho_\varepsilon](\nabla\times B_\varepsilon)\times
B_\varepsilon dxdt\\
&+\tilde{{\rm E}}\int_0^T\psi\int_D \phi
u^i_\varepsilon\left(\rho_\varepsilon\mathcal{R}_{ij}[\rho_\varepsilon
u^j_\varepsilon]-\rho_\varepsilon
u^j_\varepsilon\mathcal{R}_{ij}[\rho_\varepsilon]\right)dxdt,
\end{split}
\end{equation}
where $\mathcal{R}_{ij}$ is the Rize operators, which is defined by
\begin{equation*}\label{4.15}
\mathcal{R}_{i,j}[v]=\partial_{x_j}\mathcal{A}_i[v]\ \ \mbox{or} \ \
\hat{\mathcal{R}}_{i,j}(\xi)={\xi_i\xi_j\over|\xi|^2}.
\end{equation*}
Here we can extend the function $\rho_\varepsilon$ by zero outside
$D$ and 
\begin{equation*}
\partial_t\rho_\varepsilon=
\begin{cases}
\varepsilon\Delta\rho_\varepsilon-\Dv(\rho_\varepsilon
u_\varepsilon) \ \
&\mbox{in} \ \ D,\\
0 \ \ &\mbox{in}\ \ \mathbb{R}^3\setminus D.
\end{cases}
\end{equation*}
Since $u_\varepsilon|_{\partial D}=0$, then we have
\begin{equation*}
\Dv(\rho_\varepsilon u_\varepsilon)=0 \ \ \mbox{on} \ \
\mathbb{R}^3\setminus D.
\end{equation*}
Similarly, for $\nabla\rho_\varepsilon$, we also have
\begin{equation*}
\Dv(1_{D}\nabla\rho_\varepsilon)=
\begin{cases}
\Delta \rho_\varepsilon \ \ &\mbox{in}\ \ D,\\
~~0 \ \ &\mbox{on} \ \ \mathbb{R}^3\setminus D.
\end{cases}
\end{equation*}
We have the following lemma:
\begin{lemma}\label{lemma4.3}
Assume that $\rho\in L^p(\tilde{\Omega},L^2((0,T)\times D))$ and
$u\in L^p(\tilde{\Omega},L^2(0,T; H_0^1(D)))$ solve \eqref{4.6} in
$\mathcal{D}^{\prime}((0,T)\times D)$ almost surely. Then we can
extend $(\rho, u)$ to be zero on $\mathbb{R}^3\setminus D$.
Moreover, the equation \eqref{4.6} holds in
$\mathcal{D}^{\prime}((0,T)\times\mathbb{R}^3)$ almost surely.
\end{lemma}

Next, we consider the test functions
\begin{equation*}
\varphi_i(t,x)=\psi(t)\phi(x)\mathcal{A}_i[\rho], \quad i=1,2,3,
\end{equation*}
where $\rho$ is zero outside $D$. Similar to \eqref{4.14}, taking
$\varphi$ as test functions for $\eqref{4.6}_2$, integrating over
$\tilde{\Omega}\times(0,T)\times D$, we obtain
\begin{equation}\label{4.16}
\begin{split}
 &\tilde{{\rm E}}\int_0^T\psi\int_D\phi\left[\overline{P}-(\lambda+2\mu)\Dv u\right]\rho
dxdt\\
= &(\lambda+\mu)\tilde{{\rm E}}\int_0^T\psi\int_D \Dv
u\partial_{x_i}\phi\mathcal{A}_i[\rho]dxdt-\tilde{{\rm
E}}\int\psi\int_D
\overline{P}\partial_{x_i}\phi\mathcal{A}_i[\rho]dxdt\\
&+\mu \tilde{{\rm
E}}\int_0^T\psi\int_D\partial_{x_j}u_i\partial_{x_j}\phi\mathcal{A}_i[\rho]dxdt-\tilde{{\rm
E}}\int_0^T\psi\int_D
\rho u_iu_j\partial_{x_j}\phi\mathcal{A}_i[\rho]dxdt\\
&-\tilde{{\rm E}}\int_0^T\psi_t\int_D\phi\rho
u_i\mathcal{A}_i[\rho]dxdt+\mu \tilde{{\rm E}}\int_0^T\psi\int_D
u_i\partial_{x_i}\phi\rho dxdt\\
&-\mu \tilde{{\rm E}}\int_0^T \psi\int_D
u_i\partial_{x_j}\phi\partial_{x_j}\mathcal{A}_i[\rho]dxdt+\tilde{{\rm
E}}\int_0^T \psi\int_D \phi
\mathcal{A}[\rho](\nabla\times B)\times B dxdt\\
&+\tilde{{\rm E}}\int_0^T \psi\int_D \phi
u_i\left(\rho\mathcal{R}_{i,j}[\rho u_j]-\rho
u_j\mathcal{R}_{i,j}[\rho]\right)dxdt.
\end{split}
\end{equation}
\noindent Here we have used the following property: Recall that
$M(t)=\rho u(t)-\rho u(0)+\int_0^t[ \Dv(\rho u\otimes
u)+\nabla\overline{P}-\mu\Delta u-(\lambda+\mu)\nabla \Dv
u-(\nabla\times B)\times B]ds$ and $M$ is a martingale. So we can
prove that $\int_0^t \psi(t)\phi(x)\mathcal{A}_i[\rho]dM=0$, where
$\psi\in C_0^\infty([0,T]), \phi(x)\in C_0^\infty(D)$. In fact,
integrating by parts, we have
$$\int_0^t \psi(t)\phi(x)\mathcal{A}_i[\rho]dM=\psi(t)\phi(x)\mathcal{A}_i[\rho]M-
\int_0^tM d(\psi(s)\phi(x)\mathcal{A}_i[\rho]).$$ We denote the term
$\int_0^tM d(\psi(s)\phi(x)\mathcal{A}_i[\rho])$ by $\int_0^t Md
N(s)$. Let $\Delta_n:=\{0=t_0\le t_1\le t_2\le \ldots\le t_{k-1} \le
t_k\le\ldots\le t_n=t\}$ be a subdivision of the time interval and
denote $\norm{\Delta_n}:=\max_{1\le k\le n}|t_k-t_{k-1}|$. Then we
have
\begin{align*}&\int_0^tM d(\psi(s)\phi(x)\mathcal{A}_i[\rho])=\sum^n_{k=1}M(t_k)[N(t_k)-N(t_{k-1})]\\
&=\sum^n_{k=1} N(t_k)M(t_k)-\sum^{n-1}_{k=1} N(t_k)M(t_k)+\sum^{n-1}_{k=0} N(t_k)M(t_k)-\sum^{n-1}_{k=0} M(t_{k+1})N(t_k)\\
&=N(t)M(t)-\sum^{n-1}_{k=0}N(t_k)[M(t_{k+1})-M(t_k)].
\end{align*}
Let $\norm{\Delta_n}\to 0$, since $N(t_k)$ is independent on
$M(t_{k+1})-M(t_k)$ and $M$ is a martingale, we have $\tilde{{\rm
E}}\int_0^tM d(\psi(s)\phi(x)\mathcal{A}_i[\rho])=\tilde{{\rm
E}}N(t)M(t)$. Then $\tilde{{\rm E}}\int_0^t
\psi(t)\phi(x)\mathcal{A}_i[\rho]dM=0$.

Now following \cite{SSA,WW} we can prove that the right-hand side of
\eqref{4.14} converges to the right-hand side of \eqref{4.16} as
$\varepsilon\to 0$, which proves Lemma \ref{lemma4.2}.
\subsection{Strong convergence of the density}\label{s4.6}

In this subsection, we will follow the idea of \cite{FNP,LLP} to
prove the strong convergence of the sequence $\rho_\varepsilon$. Our
goal is to show $\overline{P}=a\rho^\gamma+\delta\rho^\beta$.

Lemma \ref{lemma4.3} implies that we can extend $(\rho, u)$ to be
zero on $\mathbb{R}^3\setminus D$, and $$\rho\in
L^p(\tilde{\Omega},L^2(0,T; L^2(D))), \quad u\in L^p(\tilde{\Omega},
L^2(0,T; H_0^1(D)))$$ solve the continuity equation $\eqref{4.6}_1$
in $\mathcal{D}^\prime((0,T)\times\mathbb{R}^3)$ almost surely.
Taking a regularizing sequence $\vartheta_m=\vartheta_m(x)$ and then
using $S_m$ to test $\eqref{4.6}_1$, we have
\begin{equation}\label{4.31}
\partial_t S_m[\rho]+\Dv\left[S_m(\rho)u\right]=r_m \ \ \mbox{on} \ \
(0,T)\times\mathbb{R}^3.
\end{equation}
Here $S_m[v]=\vartheta_m*v$ is the standard smoothing operator. We
note that \cite[Lemma 2.3]{LLP1} implies that the remainder term
$r_m\to 0$ in $L^1(\tilde{\Omega}\times(0,T)\times\mathbb{R}^3)$ as
$m\to\infty$.

For any $b$ satisfying the conditions in Definition
\ref{martingale-solution}, multiplying \eqref{4.31} by
$b^\prime(S_m[\rho])$ and passing to the limit as $m\to\infty$, we
can infer that $(\rho, u)$ solves $\eqref{4.6}_1$ in the sense of
renormalized solutions. Furthermore, taking $b(z)=z\ln z$ and then
integrating \eqref{r-equation} over $\tilde{\Omega}\times(0,T)\times
D$, one has
\begin{equation}\label{4.32}
\tilde{{\rm E}}\int_0^T\int_D \rho\Dv udxdt=\tilde{{\rm E}}\int_D
\rho_0\ln \rho_0dx-\tilde{{\rm E}}\int_0^T\rho(T)\ln\rho(T)dx.
\end{equation}
Since $(\rho_\varepsilon, u_\varepsilon)$ satisfies the equation
$\eqref{A-MHD}_1$ a.e. on $\tilde{\Omega}\times(0,T)\times D$,
multiplying $\eqref{A-MHD}_1$ with $b^\prime(\rho_\varepsilon)$,
then for any $b$ convex and globally Lipschitz on $\mathbb{R}^+$, we
have
\begin{equation*}
\partial_t b(\rho_\varepsilon)+\Dv(b(\rho_\varepsilon)u_\varepsilon)+\left[b^\prime(\rho_\varepsilon)\rho_\varepsilon-b(\rho_\varepsilon)\right]\Dv
u_\varepsilon-\varepsilon\Delta b(\rho_\varepsilon)=-\varepsilon
b^{\prime\prime}(\rho_\varepsilon)|\nabla\rho_\varepsilon|^2\le 0.
\end{equation*}
Integrating the above on $\tilde{\Omega}\times(0,T)\times D$, the
boundary condition yields
\begin{equation*}
\tilde{{\rm
E}}\int_0^T\int_D\left[b^\prime(\rho_\varepsilon)\rho_\varepsilon-b(\rho_\varepsilon)\right]\Dv
u_\varepsilon dxdt\le \tilde{{\rm E}}\int_D b(\rho_0)dx-\tilde{{\rm
E}}\int_D b(\rho_\varepsilon(T))dx.
\end{equation*}
Let $b(z)=z\ln z$ in the above inequality, then
\begin{equation}\label{4.33}
\begin{split}
\varlimsup_{\varepsilon\to 0}\tilde{{\rm E}}\int_0^T\int_D
\rho_\varepsilon\Dv u_\varepsilon dxdt&\le \tilde{{\rm E}}\int_D
\rho_0\ln \rho_0dx-\varliminf_{\varepsilon\to 0}\tilde{{\rm
E}}\int_D
\rho_\varepsilon(T)\ln\rho_\varepsilon(T)dx\\
&\le \tilde{{\rm E}}\int_0^T\int_D\rho\Dv u\,dxdt.
\end{split}
\end{equation}

Taking two nondecreasing sequences $0\le\psi_n
\in\mathcal{D}(0,T),\; 0\le\phi_n\in\mathcal{D}(D)$ satisfying
\begin{align*}
&\psi_n\to 1, \quad \phi_n \to 1\ \ \mbox{as} \ \ n\to\infty;\;
\psi_n(t)=1 \ \ \mbox{for} \ \ t\ge{1\over n} \ \ \mbox{or}\ \ t \le
T-{1\over n}; \\
&\qquad\qquad\quad\phi_n(x)=1 \ \ \mbox{for} \ \ x\in D, \;
\mbox{dist}(x,\partial D)\ge{1\over n}.
\end{align*}
Combining Lemma \ref{lemma4.2}, \eqref{4.32} and \eqref{4.33}, for
all $m\le n$, we obtain
\begin{align*}
&\varlimsup_{\varepsilon\to 0}\tilde{{\rm E}}\int_0^T
\psi_m\int_D\phi_m\left(a\rho_\varepsilon^\gamma+\delta\rho_\varepsilon^\beta\right)\rho_\varepsilon
dxdt\le\varlimsup_{\varepsilon\to 0+}\tilde{{\rm
E}}\int_0^T\psi_n\int_D\phi_n\left(a\rho_\varepsilon^\gamma+\delta\rho_\varepsilon^\beta\right)\rho_\varepsilon
dxdt\\
& \le\varlimsup_{\varepsilon\to 0}\tilde{{\rm
E}}\int_0^T\psi_n\int_D\phi_n\left[a\rho_\varepsilon^\gamma+\delta\rho_\varepsilon^\beta-(\lambda+2\mu)\Dv
u_\varepsilon\right]\rho_\varepsilon
dxdt\\
&\quad +(\lambda+2\mu)\varlimsup_{\varepsilon\to 0}\tilde{{\rm
E}}\int_0^T \psi_n\int_D\phi_n\rho_\varepsilon\Dv u_\varepsilon
dxdt\\
&\le \tilde{{\rm
E}}\int_0^T\psi_n\int_D\phi_n[\overline{P}-(\lambda+2\mu)\Dv u]\rho
dxdt+(\lambda+2\mu)\varlimsup_{\varepsilon\to 0}\tilde{{\rm
E}}\int_0^T\int_D\rho_\varepsilon|1-\psi_n\phi_n|\;|\Dv
u_\varepsilon|dxdt\\
&\quad +(\lambda+2\mu)\varlimsup_{\varepsilon\to 0}\tilde{{\rm E}}\int_0^T\int_D\rho_\varepsilon\Dv u_\varepsilon dxdt\\
&\le \tilde{{\rm E}}\int_0^T\int_D \overline{P}\rho
dxdt+(\lambda+2\mu)\tilde{{\rm E}}\int_0^T\int_D \rho\Dv udxdt-(\lambda+2\mu)\tilde{{\rm E}}\int_0^T\int_D \rho\Dv udxdt+\eta(n)\\
&\le \tilde{{\rm E}}\int_0^T\int_D \overline{P}\rho dxdt+\eta(n),
\end{align*}
where $\eta(n)\!=\!\tilde{{\rm E}}\!\int_0^T\!\!\int_D
(\psi_n\phi_n-1)\!\!\left[\overline{P}\!-\!(\lambda+2\mu)\Dv
u\right]\!\rho dxdt+(\lambda+2\mu)\tilde{{\rm
E}}\!\int_0^T\!\!\int_D\! \rho_\varepsilon|1-\psi_n\phi_n|\;|\Dv
u_\varepsilon|dxdt\to 0$ as $n\to\infty$. Passing to the limit in
the above inequality as $n\to\infty$, one deduces that
\begin{equation*}
\varlimsup_{\varepsilon\to 0+}\tilde{{\rm E}}\int_0^T
\psi_m\int_D\phi_m\left(a\rho_\varepsilon^\gamma+\delta\rho_\varepsilon^\beta\right)\rho_\varepsilon
dxdt\le \tilde{{\rm E}}\int_0^T\int_D \overline{P}\rho dxdt \ \
\mbox{for all} \ \ m=1,2,\ldots
\end{equation*}
Since $P(z)=az^\gamma+\delta z^\beta$ is monotone, then we have
\begin{equation*}
\tilde{{\rm E}}\int_0^T \psi_m\int_D
\phi_m\left[P(\rho_\varepsilon)-P(v)\right](\rho_\varepsilon-v)dxdt\ge
0.
\end{equation*}
Furthermore, the facts $\rho_\varepsilon\rightharpoonup\rho$ and
$P(\rho_\varepsilon)\rightharpoonup\overline{P}$ imply that
\begin{equation*}
\tilde{{\rm E}}\int_0^T\psi_m\int_D\phi_m \overline{P}\rho
dxdt+\tilde{{\rm E}}\int_0^T\psi_m\int_D\phi_mP(v)vdxdt-\tilde{{\rm
E}}\int_0^T\psi_m\int_D\phi_m\left[\overline{P}v+P(v)\rho\right]
dxdt\ge 0.
\end{equation*}
Now, letting $m\to\infty$ in the above inequality, one has
\begin{equation*}
\tilde{{\rm
E}}\int_0^T\int_D\left[\overline{P}-P(v)\right](\rho-v)dxdt\ge 0.
\end{equation*}
Choosing $v=\rho+\alpha\varphi$, for an arbitrary $\varphi$,  and
then letting $\alpha\to 0$, we have
\begin{equation*}
\overline{P}=a\rho^\gamma+\delta\rho^\beta.
\end{equation*}
That is,
\begin{equation*}
a\rho_\varepsilon^\gamma+\delta\rho_\varepsilon^\beta\rightharpoonup
a\rho^\gamma+\delta\rho^\beta \ \ \mbox{in}\ \
L^{\beta+1\over\beta}(\tilde{\Omega}\times(0,T)\times D).
\end{equation*}
Then we have
\begin{equation}\label{5a}
\rho_\varepsilon\to\rho \ \ \mbox{in} \ \
L^s(\tilde{\Omega}\times(0,T)\times D),\; \forall \ 1\le s<\beta+1.
\end{equation}
Similar to the proof of \eqref{P-a.s.12}, by \eqref{assumption},
H\"{o}lder's inequality, \eqref{4.6d1}, \eqref{4.6e1}, \eqref{5a},
we can deduce that
\begin{align}\label{P-a.s.11b}
\begin{split}
\left\langle \sum_{k\ge1}f_k(\rho_\varepsilon,\rho_\varepsilon
u_\varepsilon,x),\phi\right\rangle^2\to \left\langle
\sum_{k\ge1}f_k(\rho,\rho u,x),\phi\right\rangle^2 \ \ \mbox{in} \ \ L^1([0,T])\ \ \tilde{\textrm{P}}-\text{a.s.},\\
\left\langle \sum_{k\ge1}g_k(B_\varepsilon,x),\phi\right\rangle^2\to
\left\langle \sum_{k\ge1}g_k(B,x),\phi\right\rangle^2 \ \ \mbox{in}
\ \ L^1([0,T])\ \ \tilde{\textrm{P}}-\text{a.s.}.
\end{split}
\end{align}
Note that
\begin{equation}\label{3.88d}
\begin{split}
&\tilde{{\rm E}}\left[\left(\langle
M_\varepsilon(t),\phi\rangle^2-\langle
M_\varepsilon(s),\phi\rangle^2-\int_s^t\sum_{k\ge1}\langle
f_k(\rho_\varepsilon,\rho_\varepsilon u_\varepsilon,x),
\phi\rangle^2d\tau\right
)\varphi_\varepsilon\right]=0,\\
&\tilde{{\rm E}}\left[\left(\langle
\tilde{M}_\varepsilon(t),\phi\rangle^2-\langle
\tilde{M}_\varepsilon(s),\phi\rangle^2-\int_s^t\sum_{k\ge1}\langle
g_k(
B_\varepsilon,x),\phi\rangle^2d\tau\right)\varphi_\varepsilon\right]=0.
\end{split}
\end{equation}
Then, it follows from \eqref{P-a.s.11b}, \eqref{3.88d}, the uniform
integrability, Proposition \ref{Propo2} and Lemma \ref{Lemma2} that
\begin{equation*}\label{3.88f}
\begin{split}
&\tilde{{\rm E}}\left[\left(\langle M(t),\phi\rangle^2-\langle
M(s),\phi\rangle^2-\int_s^t\sum_{k\ge1}\langle f_k(\rho,\rho
u,x),\phi\rangle^2d\tau\right
)\varphi(\beta_k,\rho,u,B)\right]\\
&=\lim_{\varepsilon\to 0}\tilde{{\rm E}}\left[\left(\langle
M_\varepsilon(t),\phi\rangle^2-\langle
M_\varepsilon(s),\phi\rangle^2-\int_s^t\sum_{k\ge1}\langle
f_k(\rho_\varepsilon,\rho_\varepsilon
u_\varepsilon,x),\phi\rangle^2d\tau\right
)\varphi_\varepsilon\right]=0,\\
&\tilde{{\rm E}}\left[\left(\langle
\tilde{M}(t),\phi\rangle^2-\langle
\tilde{M}(s),\phi\rangle^2-\int_s^t\sum_{k\ge1}\langle g_k(B,x),
\phi\rangle^2d\tau\right
)\varphi(\beta_k,\rho,u,B)\right]\\
&=\lim_{\varepsilon\to 0}\tilde{{\rm E}}\left[\left(\langle
\tilde{M}_\varepsilon(t),\phi\rangle^2-\langle
\tilde{M}_\varepsilon(s),\phi\rangle^2-\int_s^t\sum_{k\ge1}\langle
g_k(B_\varepsilon,x), \phi\rangle^2d\tau\right
)\varphi_\varepsilon\right]=0.
\end{split}
\end{equation*}
Hence we deduce that $\langle
M(t),\phi\rangle^2-\int_0^t\sum_{k\ge1}\langle f_k(\rho,\rho
u,x),\phi\rangle^2ds$; $\langle \tilde{M}(t),\phi\rangle$ and
$\langle \tilde{M}(t),\phi\rangle^2-\int_0^t\sum_{k\ge1}\langle
g_k(B,x),\phi\rangle^2ds$ are continuous martingales.
%
Using again the method in \cite{BO,OM} we can also infer that
$$\langle M(t),\phi\rangle=\int_0^t\langle \sum_{k\ge 1} f_k(\rho,\rho u,x),\phi\rangle d\beta_k^1,\;\langle
\tilde{M}(t),\phi\rangle=\int_0^t\langle \sum_{k\ge 1} g_k(B,x),\phi
\rangle d\beta_k^2.$$

Summing up the above results, we have the following proposition:

\begin{proposition}\label{proposition4.1}
Assume that $D\in C^{2+\alpha}$ is a bounded domain. If
$\beta>\max\{{6\gamma\over 2\gamma-3},\gamma,4\}$, then there exists
a finite energy martingale solution
$((\tilde{\Omega},\tilde{\mathscr{F}},\tilde{\textrm{P}}),\beta_k,\rho,
u,B)$ of the problem:
\begin{align}\label{4.21}
\begin{cases}
d\rho+\Dv(\rho u)dt=0,\\
d(\rho u)+\left[\Dv(\rho u\otimes
u)+a\nabla\rho^\gamma+\delta\nabla\rho^\beta-\mu\Delta
u+(\lambda+\mu)\nabla\Dv u\right]dt\\
\quad=(\nabla\times B)\times B dt+\sum_{k\ge1}f_k(\rho,\rho u,x)d\beta_k^1,\\
dB-[\nabla\times(u\times B)+\nu\Delta
B]dt=\sum_{k\ge1}g_k(B,x)d\beta_k^2,
\end{cases}
\end{align}
$\tilde{\textrm{P}}$-a.s. and $(\rho, u, B)$ satisfies the
initial-boundary valued conditions \eqref{A-boundary},
\eqref{A-initialdata}.
Moreover, if $(\rho, u)$ is prolonged to be zero on
$\mathbb{R}^3\setminus D$, then the equation $\eqref{4.21}_1$ holds
in sense of renormalized solutions on
$\mathcal{D}^\prime((0,T)\times\mathbb{R}^3)$
$\tilde{\textrm{P}}$-a.s.. Finally, $(\rho, u, B)$ satisfies the
following estimates:
\begin{align*}
&\tilde{{\rm E}} \Big
(\sup_{t\in[0,T]}\|\rho(t)\|^\gamma_{L^\gamma(D)} \Big )^p\le C,\;
\tilde{{\rm E}} \Big
(\delta\sup_{t\in[0,T]}\|\rho(t)\|^\beta_{L^\beta(D)} \Big )^p\le
C,\\
&\tilde{{\rm E}}\Big (\sup_{t\in [0,T]}\|\sqrt{\rho} u\|^2_{L^2(D)}
\Big )^p\le C,\; \tilde{{\rm E}}\Big
(\|\rho\|_{L^{\beta+1}((0,T)\times D)} \Big )^p\le C,
\\
&\tilde{{\rm E}} \Big (\|u\|^2_{L^2(0,T; H_0^1(D))} \Big )^p\le C,\;
\tilde{{\rm E}} \Big (\norm{B}_{L^\infty(0,T;
L^2(D))}+\|B\|^2_{L^2(0,T; H_0^1(D))} \Big )^p\le C,
\end{align*}
where the constant $C$ is independent of $\delta>0$.
\end{proposition}

\section{Passing to the Limit in the Artificial Pressure Term}\label{section5}

In this section,  we shall follow the idea of \cite{FNP, LLP} to
pass to the limit in \eqref{4.21} as $\delta\to 0$ for the
artificial pressure term and relax the hypotheses on the initial
data $(\rho_{0,\delta}, m_{0,\delta}, B_{0,\delta})$, that is,
\eqref{A-boundary} and \eqref{A-initialdata}. First, from
\eqref{C-initialdata}, we have
\begin{align*}
\rho_{0,\delta}\to\rho_0 \ \ \text{in $L^\gamma(D)$}, \ \
m_{0,\delta}\to m_0 \ \ \text{in $L^1(D)$}, \ \ B_{0,\delta}\to B_0
\ \ \text{in $L^2(D)\ \  \tilde{\textrm{P}}$}-\text{a.s.}  \ \
\mbox{as} \ \ \delta\to 0.
\end{align*}
Next, we consider the problem \eqref{4.21} with the initial data
$(\rho_{0,\delta}, m_{0,\delta}, B_{0,\delta})$. From Proposition
\ref{proposition4.1}, we obtain the existence of martingale
solution, denoted by $(\beta_{k,\delta},\rho_\delta,u_\delta,
B_\delta)$. From \eqref{C-initialdata}, $\mathscr{E}_{0,\delta}$ is
bounded uniformly in $\delta$. Hence, the estimates in Proposition
\ref{proposition4.1} hold independently of $\delta$.
\subsection{On integrability of the density}\label{s5.1}
First, we shall estimate the density $\rho_\delta$, uniformly in 
$\delta>0$. Recall
\begin{align}\label{5.5a}
\partial_t S_m[b(\rho_\delta)]+\Dv(S_m[b(\rho_\delta)]u_\delta)+S_m[(b^\prime(\rho_\delta)\rho_\delta-b(\rho_\delta))\Dv
u_\delta]=r_m,
\end{align}
where
\begin{align}\label{5.6a}
r_m\to 0\ \ \mbox{in}\ \ L^2(\Omega, L^2(0,T;L^2(\mathbb{R}^3))) \ \
\mbox{ as} \ \ m\to\infty,
\end{align}
and $b$ is uniformly bounded.

Denote $\dashint_D fdx={1\over |D|}\int_D fdx$.
Similar to Section \ref{s4.2}, 
we use the operator $\mathcal{B}$ introduced in Section \ref{s4.1} to
construct the following test functions:
\begin{equation*}
\varphi_i(t,x)=\psi(t)\mathcal{B}_i\left[S_m[b(\rho_\delta)]-\dashint_D
S_m[b(\rho_\delta)]dx\right], \; i=1,2,3, \quad
\psi\in\mathcal{D}(0,T).
\end{equation*}
Note that $\varphi_i|_{\partial D}=0$, $\varphi_i\in
L^\infty(0,T;H_0^1(D))$, $\partial_t\varphi_i\in L^2(0,T;
H^1_0(D))$. Thus, we can use $\varphi_i$ as test functions for the
equation $\eqref{4.21}_2$. Using \eqref{5.5a}, we obtain
\begin{align}\label{5.6b}
\notag {\rm E}&\int_0^T\psi\int_D
\left(a\rho^\gamma_\delta+\delta\rho_\delta^\beta\right)S_m[b(\rho_\delta)]dxdt\\
\notag=&{\rm E}\int_0^T\psi\int_D
\left(a\rho^\gamma_\delta+\delta\rho_\delta^\beta\right)dx\dashint_D
S_m[b(\rho_\delta)]dx dt+(\lambda+\mu){\rm E}\int_0^T\psi\int_D
S_m[b(\rho_\delta)]\Dv u_\delta dxdt\\
\notag&-{\rm E}\int_0^T\psi_t\int_D \rho_\delta
u^i_\delta\mathcal{B}_i\left[S_m[b(\rho_\delta)]-\dashint_D
S_m[b(\rho_\delta)]dx\right] dxdt\\
\notag&+{\rm E}\int_0^T\psi\int_D
\left(\mu\partial_{x_j}u_\delta^i-\rho_\delta
u_\delta^iu_\delta^j\right)\partial_{x_j}\mathcal{B}_i\left[S_m[b(\rho_\delta)]-\dashint_D
S_m[b(\rho_\delta)]dx\right] dxdt\\
\notag&+{\rm E}\int_0^T\!\psi\!\int_D\rho_\delta
u_\delta^i\mathcal{B}_i\left[S_m[(b(\rho_\delta)-b^\prime(\rho_\delta)\rho_\delta)\Dv
u_\delta]\!-\dashint_D
S_m[(b(\rho_\delta)-b^\prime(\rho_\delta)\rho_\delta)\Dv
u_\delta]dx\right]dxdt\\
\notag&-{\rm E}\int_0^T\psi\int_D\rho_\delta
u_\delta^i\mathcal{B}_i\left[r_m-\dashint_D
r_mdx\right]+{\rm E}\int_0^T\psi\int_D \rho_\delta
u_\delta^i\mathcal{B}_i\left[\Dv\left(S_m[b(\rho_\delta)]u_\delta\right)\right]
dxdt\\
&+{\rm E}\int_0^T\psi\int_D\mathcal{B}_i\left[S_m[b(\rho_\delta)]-\dashint_D
S_m[b(\rho_\delta)]dx\right](\nabla\times B_\delta)\times B_\delta
dxdt.
\end{align}

Using \eqref{5.6a}, we can pass to the limit  in \eqref{5.6b} for
$m\to\infty$. Moreover, we can apply the function $z^\theta\approx
b(z)$ to obtain
\begin{equation*}
{\rm E}\int_0^T\psi\int_D
\left(a\rho_\delta^{\gamma+\theta}+\delta\rho_\delta^{\beta+\theta}\right)dxdt:=
\sum_{i=1}^8J_i,
\end{equation*}
for some $\theta>0$ to be determined below. Now, we just estimate
the term $J_8$ on the right hand side of the above identity. Other
terms can be bounded as in \cite{WW}. For the term $J_8$, if
$\theta<{\gamma\over3}$, using Proposition \ref{proposition4.1} and
\eqref{4.1}, together with the embedding $W^{1,p}(D)\subset
L^\infty(D)$ for $p>3$, we have
\begin{align*}
J_8&\le {\rm E}\left|\int_0^T \int_D
\psi\mathcal{B}_i\left[\rho_\delta^\theta-\dashint_D
\rho_\delta^\theta
dx\right](\nabla\times B_\delta)\times B_\delta dxdt\right|\\
&\lesssim {\rm E}\int_0^T |\psi| \|\nabla
B_\delta\|_{L^2(D)}\|B_\delta\|_{L^2(D)}\left\|\mathcal{B}_i\left[\rho_\delta^\theta
-\dashint_D\rho_\delta^\theta dx\right]\right\|_{L^\infty(D)}dt\\
&\lesssim {\rm E}\int_0^T |\psi|\|\nabla
B_\delta\|_{L^2(D)}\|B_\delta\|_{L^2(D)}\left\|\mathcal{B}_i\left[\rho_\delta^\theta-\dashint_D
\rho_\delta^\theta dx\right]\right\|_{W^{1,{\gamma\over\theta}}(D)}dt\\
&\lesssim {\rm E}\int_0^T
|\psi|\|\nabla B_\delta\|_{L^2(D)}\|B_\delta\|_{L^2(D)}\left\|\rho_\delta^\theta\right\|_{L^{\gamma\over\theta}(D)}dt\\
&\lesssim {\rm E}\left(\|\nabla B_\delta\|_{L^2(0,T;L^2(D))}\|B_\delta\|_{L^2(0,T;L^2(D))}\|\rho_\delta\|^\theta_{L^{\gamma}(D)}\right)\\
&\lesssim \left({\rm E}\|\nabla B_\delta\|_{L^2(0,T;L^2(D))}^3
\right)^{1\over3} \left({\rm E}
\|B_\delta\|^3_{L^2(0,T;L^2(D))}\right)^{1\over3}\left({\rm E}
\|\rho_\delta\|^{3\theta}_{L^{\gamma}(D)}\right)^{1\over3}\\
&\le C.
\end{align*}
Here the constant $C$ is independent of $\delta$.
\if false For the term
$J_1$, if $\theta\le\gamma$, it follows from \eqref{4.37} that
\begin{align*}
J_1&=E\int_0^T\psi\int_D\left(a\rho^\gamma_\delta+\delta\rho^\beta_\delta
\right)dx\dashint_D \rho_\delta^\theta dxdt\\
&\lesssim
E\left(\|\rho_\delta\|^{\gamma+\theta}_{L^\infty(0,T,L^\gamma(D))}+
\|\rho_\delta\|^{\theta}_{L^\infty(0,T,L^\gamma(D))}\|\delta\rho_\delta\|^{\beta}_{L^\infty(0,T,L^\beta(D))}\right)\\
&\lesssim
E\left(\|\rho_\delta\|^{2(\gamma+\theta)}_{L^\infty(0,T,L^\gamma(D))}\right)+E\left(\|\rho_\delta\|^{2\theta}_{L^\infty(0,T,L^\gamma(D))}\right)
+E\left(\|\delta\rho_\delta\|^{2\beta}_{L^\infty(0,T,L^\beta(D))}\right)\\
&\le C.
\end{align*}
The constant $C$ is independent of $\delta$ in the above estimate as
well as all the estimates below. Similarly, for the term $J_2$, if
$\theta\le {\gamma\over 2}$, by using \eqref{4.37} and \eqref{4.40},
we have
\begin{align*}
J_2&\lesssim E\left|\int_0^T\int_D\psi\rho_\delta^\theta\Dv u_\delta
dxdt\right|\lesssim E\left(\int_0^T \psi\|\nabla
u_\delta\|_{L^2(D)}\|\rho_\delta^\theta\|_{L^2(D)}dt\right)\\
&\lesssim
E\left(\|\rho_\delta\|^\theta_{L^\infty(0,T;L^\gamma(D))}\|\nabla
u_\delta\|_{L^2(0,T;L^2(D))}\right)\\
&\lesssim
E\left(\|\rho_\delta\|^{2\theta}_{L^\infty(0,T;L^\gamma(D))}\right)+CE\left(\|\nabla
u_\delta\|^2_{L^2(0,T;L^2(D))}\right)\\
&\le C.
\end{align*}
For the term $J_3$, if $\theta<{\gamma\over3}$, using \eqref{4.37},
\eqref{4.39} and \eqref{4.1}, together with the embedding
$W^{1,p}(D)\subset L^\infty(D)$ for $p>3$, one has
\begin{align*}
J_3&\le E\left|\int_0^T \int_D \psi_t\rho_\delta
u^i_\delta\mathcal{B}_i\left[\rho_\delta^\theta-\dashint_D
\rho_\delta^\theta
dx\right]dxdt\right|\\
&\lesssim E\int_0^T
|\psi_t|\|\sqrt{\rho_\delta}\|_{L^2(D)}\|\sqrt{\rho_\delta}u_\delta\|_{L^2(D)}\left\|\mathcal{B}_i\left[\rho_\delta^\theta-
\dashint_D\rho_\delta^\theta dx\right]\right\|_{L^\infty(D)}dt\\
&\lesssim E\int_0^T
|\psi_t|\|\sqrt{\rho_\delta}\|_{L^2(D)}\|\sqrt{\rho_\delta}u_\delta\|_{L^2(D)}\left\|\mathcal{B}_i\left[\rho_\delta^\theta-
\dashint_D\rho_\delta^\theta dx\right]\right\|_{W^{1,{\gamma\over\theta}}(D)}dt\\
&\lesssim E\int_0^T
|\psi_t|\|\sqrt{\rho_\delta}\|_{L^2(D)}\|\sqrt{\rho_\delta}u_\delta\|_{L^2(D)}\left\|\rho_\delta^\theta\right\|_{L^{\gamma\over\theta}(D)}dt\\
&\lesssim E\left(\|\sqrt{\rho_\delta}\|_{L^2(D)}\|\sqrt{\rho_\delta}u_\delta\|_{L^2(D)}\left\|\rho_\delta\right\|^\theta_{L^{\gamma}(D)}\right)\\
&\lesssim \left(E\|\sqrt{\rho_\delta}\|^3_{L^2(D)} \right)^{1\over3}
\left(E\|\sqrt{\rho_\delta}u_\delta\|^3_{L^2(D)}
\right)^{1\over3}\left(E
\|\rho_\delta\|^{3\theta}_{L^{\gamma}(D)}\right)^{1\over3}\\
&\le C.
\end{align*}
Furthermore, for the term $J_4$, if $\theta\le {\gamma\over2}$, we
can obtain as $J_2$
\begin{align*}
J_4&\le E\left|\int_0^T\int_D
\psi\partial_{x_j}u_\delta^i\partial_{x_j}\mathcal{B}_i\left[\rho_\delta^\theta-\dashint_D
\rho_\delta^\theta dx\right]dxdt\right|\\
&\lesssim E\int_0^T \|\nabla
u_\delta\|_{L^2(D)}\left\|\mathcal{B}_i\left[\rho_\delta^\theta-\dashint_D
\rho_\delta^\theta dx\right]\right\|_{H^1(D)}dt\\
&\lesssim E\int_0^T \|\nabla
u_\delta\|_{L^2(D)}\|\rho_\delta^\theta\|_{L^2(D)}dt\le C.
\end{align*}
Similarly, for the term $J_5$, if $\theta\le {2\over3}\gamma-1$, one
has
\begin{align*}
J_5&\le E\left|\int_0^T\int_D \psi\rho_\delta u_\delta^iu_\delta^j
\partial_{x_j}\mathcal{B}_i\left[\rho_\delta^\theta-\dashint_D
\rho_\delta^\theta dx\right]dxdt\right|\\
&\lesssim E \int_0^T \|\rho_\delta\|_{L^\gamma(D)}\|\nabla
u_\delta\|^2_{L^2(D)}\left\|\mathcal{B}_i\left[\rho_\delta^\theta-\dashint_D
\rho_\delta^\theta dx\right]\right\|_{W^{1,{3\gamma\over
2\gamma-3}}(D)}dt\\
&\lesssim E\int_0^T \|\rho_\delta\|_{L^\gamma(D)}\|\nabla
u_\delta\|^2_{L^2(D)}\left\|\rho_\delta^\theta\right\|_{L^{3\gamma\over
2\gamma-3}(D)}dt\\
&\lesssim  E\left(\left\|\rho_\delta\right\|_{L^\infty(0,T;
L^\gamma(D))}\left\|\rho^\theta_\delta\right\|_{L^\infty(0,T;
L^{3\gamma\over
2\gamma-3}(D))}\int_0^T \|\nabla u_\delta\|^2_{L^2(D)}dt\right)\\
&\lesssim E\left(\|\rho_\delta\|^{\theta+1}_{L^\infty(0,T;
L^\gamma(D))}\|\nabla u_\delta\|^2_{L^2(0,T;L^2(D))}\right)\\
&\lesssim E\left(\|\rho_\delta\|^{\theta+1}_{L^\infty(0,T;
L^\gamma(D))}\right)^2+E\left(\|\nabla
u_\delta\|^2_{L^2(0,T;L^2(D))}\right)^2\\
&\le C.
\end{align*}

Next, we shall estimate the term $J_6$.  If $\gamma<6$ and
$\theta<{2\over3}\gamma-1$, using \eqref{4.37}, \eqref{4.39},
\eqref{4.1} and by  H\"{o}lder's inequality, we have
\begin{align*}
J_6&\le E\left|(1-\theta)\int_0^T\int_D \psi\rho_\delta
u^i_\delta\mathcal{B}_i\left[\rho_\delta\Dv u_\delta-\dashint_D
\rho_\delta^\theta\Dv u_\delta dx\right]dxdt \right|\\
&\lesssim E\int_0^T
\|\rho_\delta\|_{L^\gamma(D)}\|u_\delta\|_{L^6(D)}\|\mathcal{B}[\cdot]\|_{L^p(D)}dt\\
&\lesssim E\int_0^T
\|\rho_\delta\|_{L^\gamma(D)}\|u_\delta\|_{L^6(D)}\|\mathcal{B}[\cdot]\|_{W^{1,q}(D)}dt\\
&\lesssim E\int_0^T
\|\rho_\delta\|_{L^\gamma(D)}\|u_\delta\|_{L^6(D)}\left\|\rho_\delta^\theta\Dv
u_\delta\right\|_{L^q(D)}dt\\
&\lesssim E\left(\int_0^T \|\rho_\delta\|_{L^\gamma(D)}\|\nabla
u_\delta\|^2_{L^2(D)}\left\|\rho_\delta^\theta\right\|_{L^{3\gamma\over
2\gamma-3}(D)}dt\right)\\
&\lesssim
E\left(\|\rho_\delta\|_{L^\infty(0,T;L^\gamma(D))}^{\theta+1}\right)^2+E\left(\|\nabla
u_\delta\|^2_{L^2(0,T;L^2(D))}\right)^2\\
&\le C,
\end{align*}
where $ p={\small 6\gamma\over 5\gamma-6},\; q={\small 6\gamma\over
7\gamma-6}. $ If $\gamma\ge 6$ and $\theta\le 1$, we have
\begin{align*}
J_6&\lesssim E\int_0^T
\|\rho_\delta\|_{L^\gamma(D)}\|u_\delta\|_{L^6(D)}\|\mathcal{B}[\cdot]\|_{L^p(D)}dt\\
&\lesssim E\int_0^T
\|\rho_\delta\|_{L^\gamma(D)}\|u_\delta\|_{L^6(D)}\|\mathcal{B}[\cdot]\|_{W^{1,{3\over2}}(D)}dt\\
&\lesssim E\left(\int_0^T \|\rho_\delta\|_{L^\gamma(D)}\|\nabla
u_\delta\|^2_{L^2(D)}\left\|\rho_\delta^\theta\right\|_{L^6(D)}dt\right)\\
&\lesssim
E\left(\|\rho_\delta\|_{L^\infty(0,T;L^\gamma(D))}^{\theta+1}\right)^2+CE\left(\|\nabla
u_\delta\|^2_{L^2(0,T;L^2(D))}\right)^2\\
&\le C.
\end{align*}
Summing up the above results, if $\theta\le {2\over3}\gamma-1$ and
$\theta\le 1$, then we have proved that $J_6$ is bounded uniformly
in $\delta$. Similarly, if $\theta\le {2\over3}\gamma-1$, using
\eqref{4.2}, \eqref{4.37} and \eqref{4.39},  one has
\begin{align*}
J_7&\le E\left|\int_0^T\int_D\psi\rho_\delta
u_\delta^i\mathcal{B}_i\left[\Dv\left(S_m[b(\rho_\delta)]u_\delta\right)\right]
dxdt \right|\\
&\lesssim
E\int_0^T\|\rho_\delta\|_{L^\gamma(D)}\|u_\delta\|_{L^6(D)}
\|S_m[b(\rho_\delta)]u_\delta\|_{L^p(D)}dt\\
&\lesssim
E\int_0^T\|\rho_\delta\|_{L^\gamma(D)}\|u_\delta\|^2_{L^6(D)}\|\rho_\delta^\theta\|_{L^{3\gamma\over
2\gamma-3}(D)}dt\\
&\le C.
\end{align*}

Finally, \fi

Summing up the estimates for \eqref{5.6b}, we have
\begin{equation}\label{5.6c}
{\rm E}\int_0^T\psi\int_D
\left(a\rho_\delta^{\gamma+\theta}+\delta\rho_\delta^{\beta+\theta}\right)dxdt\le
C.
\end{equation}
Taking $\psi_n\in\mathcal{D}(0,T), 0\le \psi_n\le 1$ and $\int_0^T
|\partial_t \psi_n|dt\le C$. By approximation and Lebesgue
convergence theorem, we can take $\psi=1$ in \eqref{5.6c}. Then we have proved the
following result:

\begin{lemma}\label{lemma5.1}
For $\gamma>{3\over 2}$ and
$0<\theta<\min\left\{1,{\gamma\over3},{2\over3}\gamma-1\right\}$,
then there exists a constant $C$ independently of $\delta>0$, such
that
\begin{equation*}
{\rm E}\int_0^T\int_D\left(
a\rho_\delta^{\gamma+\theta}+\delta\rho_\delta^{\beta+\theta}\right)dxdt\le
C.
\end{equation*}
\end{lemma}
\subsection{Tightness Property}\label{s5.2}
Similar to Section \ref{s4.3}, we can prove the following lemma:
\begin{lemma}\label{Lemma-Tight2} Define
\begin{align*}
S=&C(0,T;\mathbb{R})\!\times\! C([0,T];L_w^{\gamma}(D))\!\times\!
L^2(0,T;H_w^1(D))\!\times\! C([0,T]; L_w^{2\gamma\over\gamma+1}(D))\\
&\times
\left(L^2(0,T;H_w^1(D))\cap L^2(0,T; L^2(D))\right)
\end{align*}
equipped with its Borel $\sigma$-algebra.  Let $\Pi_\delta$ be the
probability on $S$ which is the image of $\textrm{P}$ on $\Omega$ by
the map: $\omega\mapsto (\beta_{k,\delta}(\omega,\cdot),
\rho_\delta(\omega,\cdot),u_\delta(\omega,\cdot),\rho_\delta
u_\delta(\omega,\cdot), B_\delta(\omega,\cdot))$, that is, for any
$A\subseteq S$,
\begin{equation*}
\Pi_\delta(A)=\textrm{P}\left\{\omega\in\Omega:
(\beta_{k,\delta}(\omega,\cdot),\rho_\delta(\omega,\cdot),u_\delta(\omega,\cdot),\rho_\delta
u_\delta(\omega,\cdot),B_\delta(\omega,\cdot))\in A \right\}.
\end{equation*}
Then the family $\Pi_\delta$ is tight. \end{lemma}

\subsection{The limit passage}\label{s5.3}

In this subsection, we shall pass to the limit for $\delta\to 0$
following the idea of \cite{FNP, LLP}. According to Jakubowski-Skorohod Theorem, there exists a subsequence such that
$\Pi_{\delta}\to \Pi$ weakly, where $\Pi$ is a probability on $S$. Moreover,
 there exist a probability space
$(\tilde{\Omega}, \tilde{\mathscr{F}}, \tilde{\textrm{P}})$ and
random variables
$(\tilde{\beta}_{k,\delta},\tilde{\rho}_\delta,\tilde{u}_\delta,\tilde{\rho}_\delta
\tilde{u}_\delta,\tilde{B}_\delta)$ with distribution $\Pi_\delta$,
$(\beta_k, \rho, u, h, B)$ with values in $S$ such that
\begin{align}\label{3-tightness}
(\tilde{\beta}_{k,\delta},\tilde{\rho}_\delta,\tilde{u}_\delta,\tilde{\rho}_\delta
\tilde{u}_\delta,\tilde{B}_\delta)\to (\beta_k,\rho,u,h,B) \ \
\mbox{in} \ \ S\ \ \tilde{\textrm{P}}-\text{a.s.}.\end{align} Since
the distribution of
$(\tilde{\beta}_{k,\delta},\tilde{\rho}_\delta,\tilde{u}_\delta,\tilde{\rho}_\delta
\tilde{u}_\delta,\tilde{B}_\delta)$ is $\Pi_\delta$, we can deduce
that $(\tilde{\rho}_\delta,
\tilde{u}_\delta,\tilde{\rho}_\delta\tilde{u}_\delta,\tilde{B}_\delta)$
satisfies the same estimate as $(\rho_\delta, u_\delta,\rho_\delta
u_\delta ,B_\delta)$. To simplify notations, we drop the tilde on
the random variables. First, we can apply Lemma
\ref{lemma5.1} 
to
obtain
\begin{equation}\label{6a}
\begin{split}
&\delta\rho_\delta^\beta\to 0 \ \ \mbox{in} \ \
L^{\beta+\theta\over\beta}(\tilde{\Omega}\times(0,T)\times
D),
\\
&\rho_\delta^\gamma\to\overline{\rho^\gamma} \ \ \mbox{weakly in} \
\ L^{\gamma+\theta\over\gamma}(\tilde{\Omega}\times(0,T)\times
D)
.
\end{split}
\end{equation}
Similar to Section \ref{s4.4}, \eqref{3-tightness} implies
that
\begin{align}\label{5.7b}
\begin{split}
& \rho_\delta\to \rho \ \ \mbox{in} \ \
C([0,T];L_w^\gamma(D))\cap C([0,T];H^{-1}(D)) \ \ \tilde{\textrm{P}}-\text{a.s.},\\
& u_\delta\rightharpoonup u \ \ \mbox{in}\ \
L^2(0,T; H^1(D)))\ \ \tilde{\textrm{P}}-\text{a.s.},  \\
& \rho_\delta u_\delta \to \rho u \ \ \mbox{in} \ \
C([0,T];L_w^{2\gamma\over \gamma+1}(D))\ \
\tilde{\textrm{P}}-\text{a.s.}.
\end{split}
\end{align}
Since $\gamma>{3\over2}$, then ${2\gamma\over \gamma+1}>{6\over5}$.
By \eqref{5.7b}, one has
\begin{align}\label{5.8e0}
\rho_\delta u_\delta^i u_\delta^j\to \rho u^i u^j \ \ \mbox{in} \ \
\mathcal{D}^\prime((0,T)\times D)\ \ \tilde{\textrm{P}}-\text{a.s.}.
\end{align}
By \eqref{3-tightness}, we know that
\begin{align}\label{5.8e1}
B_\delta\to B \ \ \mbox{in $L^2(0,T; L^2(D))$ and }
B_\delta\rightharpoonup B \ \ \mbox{in $L^2(0,T; H^1(D))$} \ \
\tilde{\textrm{P}}-\text{a.s.}.
\end{align}
Moreover, from $\eqref{A-MHD}_3$ and Proposition
\ref{proposition4.1}, by Aubin-Lions Lemma, one has
\begin{align*}
B_\delta\to B \ \ \mbox{in} \ \ C([0,T]; H^{-1}(D))\ \
\tilde{\textrm{P}}-\text{a.s.}.
\end{align*}
Then we have
\begin{align}\label{5.8e2}
(\nabla\times B_\delta)\times B_\delta\to(\nabla\times B)\times B \
\ \mbox{in $\mathcal{D}^\prime([0,T]\times D)$} \ \
\tilde{\textrm{P}}-\text{a.s.}.
\end{align}
Similarly, it follows from \eqref{5.7b}, and \eqref{5.8e1} that
\begin{align}\label{5.8e3}
\nabla\times(u_\delta\times B_\delta)\to\nabla\times (u\times B) \ \
\mbox{in $\mathcal{D}^\prime([0,T]\times D)$} \ \
\tilde{\textrm{P}}-\text{a.s.}.
\end{align}
Now, it remains to show the strong convergence of $\rho_\delta$ in
$L^1(\tilde{\Omega}\times(0,T)\times D)$, i.e.,
$\overline{\rho^\gamma}=\rho^\gamma$ for proving Theorem
\ref{theorem1.1}. First, we introduce the cut off functions:
\begin{align*}
T_k(z)=kT\left({z\over k}\right), k=1,2,3,\ldots,
\end{align*}
where
\begin{align*}
T(z)=
\begin{cases}
z, \ \ z\le 1\\
2, \ \ z\ge 3\\
\end{cases}
\in C^\infty(\mathbb{R}^3), \text{ concave in } z\in\mathbb{R}.
\end{align*}
 Since
$(\rho_\delta,u_\delta)$ is a renormalized solution of
$\eqref{4.21}_1$ in $\mathcal{D}^\prime((0,T)\times\mathbb{R}^3)$
$\tilde{\textrm{P}}$-a.s., choosing $b(z)=T_k(z)$, one deduces that
in $\mathcal{D}^\prime((0,T)\times\mathbb{R}^3)$
$\tilde{\textrm{P}}$-a.s.
\begin{align}\label{5.7}
\partial_tT_k(\rho_\delta)+\Dv(T_k(\rho_\delta)u_\delta)+[T_k^\prime(\rho_\delta)\rho_\delta-T_k(\rho_\delta)]\Dv
u_\delta=0.
\end{align}
Noting that for any fixed $k$, $T_k(\rho_\delta)\in
L^p(\tilde{\Omega},L^\infty((0,T)\times D))$ and
$T_k(\rho_\delta)\to \overline{T_k(\rho)}$ weakly star in
$L^p(\tilde{\Omega},L^\infty((0,T)\times D))$ and
$\partial_tT_k(\rho_\delta)$ satisfies the equation \eqref{5.7}, by
Aubin-Lions Lemma, as $\delta\to0$, for all $1\le p<\infty$, we
can infer that $\tilde{\textrm{P}}$-a.s.
\begin{align}\label{5.7a1}
T_k(\rho_\delta)\to\overline{T_k(\rho)}\ \ \mbox{in} \ \ C_w([0,T];
L^p(D))\cap C([0,T];H^{-1}(D)).
\end{align}
Letting $\delta\to 0+$ in \eqref{5.7} and by using \eqref{5.7a1}, we have
$\tilde{\textrm{P}}$-a.s.
\begin{align}\label{5.7a}
\partial_t\overline{T_k(\rho)}+\Dv\left(\overline{T_k(\rho)}
u\right)+\overline{[T_k^\prime(\rho)\rho-T_k(\rho)]\Dv u}=0 \ ,
\end{align}
in $\mathcal{D}^\prime((0,T)\times\mathbb{R}^3)$, where
\begin{align*}
[T^\prime_k(\rho_\delta)\rho_\delta-T_k(\rho_\delta)]\Dv
u_\delta\rightharpoonup\overline{[T_k^\prime(\rho)\rho-T_k(\rho)]\Dv
u} \ \ \mbox{in} \ \ L^2(\tilde{\Omega},L^2((0,T)\times D)).
\end{align*}
Here $\overline{h(\rho)}$ is the weak limit of $h(\rho_\delta)$.

\subsection{The effective viscous flux}\label{s5.4}
Similar to the Section \ref{s4.5},  we have the following lemma:
\begin{lemma}\label{lemma5.2}
Let $(\rho_\delta,u_\delta, B_\delta)$ be the sequence of
approximate solutions constructed in Proposition
\ref{proposition4.1}. Then for any $\psi\in\mathcal{D}(0,T),\;
\phi\in\mathcal{D}(D)$,
\begin{align*}
&\lim_{\delta\to0+}\tilde{{\rm E}}\int_0^T \psi\int_D
\phi\left[a\rho_\delta^\gamma-(\lambda+2\mu)\Dv
u_\delta\right]T_k(\rho_\delta)dxdt\\
&=\tilde{{\rm E}}\int_0^T \psi\int_D
\phi\left[a\overline{\rho^\gamma}-(\lambda+2\mu)\Dv
u\right]\overline{T_k(\rho)}dxdt.
\end{align*}
\end{lemma}
\noindent Here we have used \eqref{5.7a} and the property under
\eqref{4.16}.

\subsection{The amplitude of oscillations}\label{s5.5}
In this subsection, we have the following lemma in the sense of
expectation as in \cite{WW}:
\begin{lemma}\label{lemma5.3}
There exists a constant $C$ independent of $k\ge1$ such that
\begin{align*}
\varlimsup_{\delta\to 0+}
\tilde{{\rm E}}\|T_k(\rho_\delta)-T_k(\rho)\|^{\gamma+1}_{L^{\gamma+1}((0,T)\times
D)}\le C.
\end{align*}
\end{lemma}
\subsection{The renormalized solutions}\label{5.6}
In this subsection, we shall use Lemma \ref{lemma5.3} and the
cut-off function technique
 to show the following crucial lemma as in \cite{WW}:

\begin{lemma}\label{lemma5.4}
Suppose $(\rho, u)$ be zero outside $D$, then $(\rho, u)$ solves the
continuity equation in the sense of renormalized solutions, that is,
\begin{align*}
\partial_t b(\rho)+\Dv(b(\rho)u)+\left[b^\prime(\rho)\rho-b(\rho)\right]\Dv u=0
\ \ \mbox{in} \ \ \mathcal{D}^\prime((0,T)\times\mathbb{R}^3),
\end{align*}
$\tilde{\textrm{P}}$-a.s. for any $b$ satisfying the conditions on
$b$.
\end{lemma}
\subsection{Strong convergence of the density}\label{s5.7}
This subsection is devoted to completing the proof of Theorem
\ref{theorem1.1}. First, we introduce a family of functions $L_k$ as
\begin{align*}
L_k(z)=
\begin{cases}
z\ln z,\ \ &\mbox{for} \ \ 0\le z<k,\\
z\ln k+z\int_k^z{T_k(s)\over s^2} ds, \ \ &\mbox{for} \ \ z\ge k,
\end{cases}
\in C^1(\mathbb{R}_+)\cap C^0[0,\infty).
\end{align*}
Then, for large enough $z$, $L_k(z)$ is a linear function. More
precisely, for $z\ge 3k$, $L_k(z)=\beta_k z-2k$. Here
$$\beta_k=\ln k+\int_k^{3k}{T_k(s)\over s^2} ds+{2\over3}.$$
 Denote
$b_k(z)=L_k(z)-\beta_kz$. Then $b_k(z)\in C^1(\mathbb{R}_+)\cap
C^0[0,\infty)$, and $b^\prime_k(z)=0$ when $z$ is large enough.
Meanwhile, $b^\prime_k(z)-b_k(z)=T_k(z)$. Taking $b(z)=b_k(z)$, the
fact that $(\rho_\delta,u_\delta)$ is renormalized solution of
$\eqref{4.21}_1$ yields that
\begin{align}\label{5.16}
\partial_t
L_k(\rho_\delta)+\Dv(L_k(\rho_\delta)u_\delta)+T_k(\rho_\delta)\Dv
u_\delta=0.
\end{align}
As in \eqref{5.16}, it follows from the continuity equation and
Lemma \ref{lemma5.4} that
\begin{align}\label{5.17}
\partial_t
L_k(\rho)+\Dv(L_k(\rho)u)+T_k(\rho)\Dv u=0\ \ \mbox{in} \ \
\mathcal{D}^\prime((0,T)\times D).
\end{align}
Since $L_k(z)$ is a linear function when $z$ is large enough, by
\eqref{5.16} and Aubin-Lions Lemma,  then it holds that
$\tilde{\textrm{P}}$-a.s.
\begin{align}\label{5.17a}
L_k(\rho_\delta)\to\overline{L_k(\rho)}\ \ \mbox{in} \ \ C_w([0,T];
L^\gamma(D))\cap C^0([0,T]; H^{-1}(D)),
\end{align}
as $\delta \to 0$. Taking the difference of \eqref{5.16} and
\eqref{5.17} and then taking $\phi\in\mathcal{D}(D)$ as a test
function, we have
\begin{align}
\label{5.18} &\tilde{{\rm E}}\int_D (L_k(\rho_\delta)-L_k(\rho))(t)\phi dx
 =\tilde{{\rm E}}\int_D
(L_k(\rho_{0,\delta})-L_k(\rho_0))(t)\phi dx\\
\notag &\quad
+\tilde{{\rm E}}\int_0^t\int_D(L_k(\rho_\delta)u_\delta-L_k(\rho)u)\cdot\nabla\phi+(T_k(\rho)\Dv
u-T_k(\rho_\delta)\Dv u_\delta)\phi dxdt.
\end{align}
Letting $\delta\to 0$ in \eqref{5.18}, noting that
$L_k\left(\rho_{0,\delta}\right)-L_k(\rho_0)\to 0$ as $\delta\to0$,
by \eqref{5.17a}, one has
\begin{align}\label{5.19}
&\tilde{{\rm E}}\int_D (\overline{L_k(\rho)}-L_k(\rho))(t)\phi dx
 =\tilde{{\rm E}}\int_0^t\int_D(\overline{L_k(\rho)}u-L_k(\rho)u)\cdot\nabla\phi dxdt\\
&\notag\quad +\lim_{\delta\to0+}\tilde{{\rm E}}\int_0^t\int_D[T_k(\rho)\Dv
u-T_k(\rho_\delta)\Dv u_\delta]\phi dxds,\quad
\forall\phi\in\mathcal{D}(D) .
\end{align}
By approximating $\phi$ in the above inequality, similar to Lemma
\ref{lemma4.3}, taking $\phi=1$ in \eqref{5.19}, $u|_{\partial D}=0$, Lemma
\ref{lemma5.2} and Lemma \ref{lemma5.3} imply
\begin{align}\label{5.20}
\begin{split}
&\tilde{{\rm E}}\int_D (\overline{L_k(\rho)}-L_k(\rho))(t)dx\\
&=\tilde{{\rm E}}\int_0^t\int_D T_k(\rho)\Dv u dxdt
-\lim_{\delta\to0}\tilde{{\rm E}}\int_0^T\int_D T_k(\rho_\delta)\Dv
u_\delta
dxdt\\
&\le \tilde{{\rm E}}\int_0^T\int_D ( T_k(\rho)-\overline{T_k(\rho)})\Dv u dxdt\\
&\lesssim \tilde{{\rm E}}\left(\|\Dv u\|_{L^2(\{\rho\ge
k\})}+\|T_k(\rho)-\overline{T_k(\rho)}\|_{L^2(\{\rho\le
k\})}\right).
\end{split}
\end{align}
On the one hand, since $T_k(\rho)\ge \overline{T_k(\rho)}$ and
$T_k(\rho)\le \rho$, by (5.20) in \cite{WW}, we have
\begin{align*}
\tilde{{\rm E}}\|T_k(\rho)-\overline{T_k(\rho)}\|_{L^1(\{\rho\le k\})}
&=\tilde{{\rm E}}\|\rho-\overline{T_k(\rho)}\|_{L^1(\{\rho\le k\})}\\
&\le \tilde{{\rm E}}\|\rho-\overline{T_k(\rho)}\|_{L^1(0,T;D)}\le
2Ck^{-(\gamma-1)}\to0,
\end{align*}
 as $k\to \infty$. By interpolation and
Lemma \ref{lemma5.3}, we have
\begin{align*}&\tilde{{\rm E}}\|T_k(\rho)-\overline{T_k(\rho)}\|_{L^2(\{\rho\le k\})}\\
&\le
\left(\tilde{{\rm E}}\|T_k(\rho)-\overline{T_k(\rho)}\|_{L^1(\{\rho\le
k\})}\right)^{\frac{\gamma-1}{2\gamma}}\left(\tilde{{\rm E}}\norm{T_k(\rho)-\overline{T_k(\rho)}}^{\gamma+1}_{L^{\gamma+1}(\{\rho\le
k\})}\right)^{\frac{1}{2\gamma}}\to 0.
\end{align*}
 This and $u\in L^2(\tilde{\Omega},L^2(0,T;H_0^1(D))$ imply that
the right-hand side of \eqref{5.20} goes to zero as $k\to\infty$.
Therefore,
\begin{align}\label{5.36}
\lim_{k\to\infty}\tilde{{\rm E}}\int_D\left(\overline{L_k(\rho)}-L_k(\rho)\right)dx\le
0, \; t\in[0,T].
\end{align}
On the other hand, since $L_k(\rho)$ is linear  when $\rho$ is large
enough,   we have
\begin{align}\label{5.37}
\begin{split}
& \tilde{{\rm E}}\|L_k(\rho)-\rho\ln \rho\|_{L^1(0,T;L^1(D))}
\le \tilde{{\rm E}}\int_0^T\int_{\{\rho\ge k\}}|L_k(\rho)-\rho\ln \rho|dxdt\\
&\le \tilde{{\rm E}}\int_0^T\int_{\{\rho\ge
k\}}(|L_k(\rho)|+|\rho\ln \rho|)dxdt \lesssim \tilde{{\rm
E}}\int_0^T\int_{\{\rho\ge k\}}|\rho\ln \rho|\to 0,
\end{split}
\end{align}
 as $ k\to\infty.$

Similarly, for small enough $\varepsilon$, one has
\begin{align*}
& \tilde{{\rm E}}\|L_k(\rho_\delta)-\rho_\delta\ln
\rho_\delta\|_{L^1(0,T;L^1(D))}
 \le \tilde{{\rm E}}\int_0^T\int_{\{\rho_\delta\ge k\}}|L_k(\rho_\delta)
 -\rho_\delta\ln \rho_\delta|dxdt\\
\notag &\le \tilde{{\rm E}}\int_0^T\int_{\{\rho_\delta\ge
k\}}{|L_k(\rho_\delta)|+|\rho_\delta\ln \rho_\delta|\over
\rho_\delta^\gamma}\rho_\delta^\gamma dxdt
 \le C(\varepsilon) \tilde{{\rm E}}\int_0^T\!\!\!\!\int_{\{\rho_\delta\ge
k\}}{\rho_\delta^\gamma\over \rho_\delta^{\gamma-1-\varepsilon}}dxdt\\
\notag &\le
C(\varepsilon)k^{1-\gamma+\varepsilon}\tilde{{\rm E}}\int_0^T\!\!\!\int_{\{\rho_\delta\ge
k\}}\rho_\delta^\gamma dxdt
 \le C(\varepsilon)k^{1-\gamma+\varepsilon} \to 0 \ \ \mbox{as} \ \
k\to\infty,
\end{align*}
uniformly in $\delta$. Let $k\to\infty$ in the above inequality,
then the weak lower semi-continuity of norm implies that
\begin{align}\label{5.39}
\tilde{{\rm E}}\|\overline{L_k(\rho)}-\overline{\rho\ln
\rho}\|_{L^1(0,T;L^1(D))}\to 0, \ \ \mbox{uniformly in} \ \ \delta.
\end{align}
\if false \textcolor{red}{Since $u\in L^2(\tilde{\Omega},L^2([0,T]; H^1_0
(D)))$, by Theorem 4.2 in \cite{f2}, we have $
|u|{\textrm{dist}[x,\partial D]}^{-1}\in
L^2(\tilde{\Omega},L^2([0,T]; L^2(D))).$ Let us consider a sequence
of functions $\phi_m\in \mathcal{D}(D)$ as follows:
\begin{equation*}\label{aa1}
\begin{split}
&0\leq\phi_m\leq 1;\;\phi_m(x)=1 \;\textrm{for}\; x\in D,\;
\textrm{dist}[x,\partial D]\geq\frac{1}{m},\\
&\qquad\phi_m\to 1;\; |\nabla\phi_m(x)|\leq 2m \textrm{ for all }
x\in D.
\end{split}
\end{equation*}
Taking  $\phi=\phi_m$  in \eqref{5.19},
 and letting $m\to \infty$, one has
\begin{equation}\label{m52}
\int_D(\overline{L_k(\rho)}-L_k(\rho))\,dx= \int_0^t\!\!\!\int_D
T_k(\rho)\Dv u \,dxds-\lim_{\delta\to 0+}\int_0^t\!\!\!\int_D
T_k(\rho_\delta)\Dv u_\delta dxds.
\end{equation}
Note that $\overline{L_k(\rho)}-L_k(\rho)$ is bounded by the
Definition of $L_k$. By virtue of \eqref{4.37}, we can assume
\begin{equation*}
\rho_\delta \ln(\rho_\delta)\rightarrow\overline{\rho \ln\rho}
\textrm{ weakly star in }
L^p(\tilde{\Omega},L^{\infty}([0,T];L^p(D))) \textrm{ for all }
1\leq p<\gamma.
\end{equation*}
On the other hand, denote $r(k):= \text{meas}\{(x,t)\in D\times
(0,T) |~\rho_\delta(x,t)\geq k\}$. By \eqref{4.37}, one has
$\lim_{k\rightarrow\infty}r(k)=0.$ Similar to \eqref{n3}, by the
fact $L_k(z)\leq z\ln z$, we obtain
\begin{equation*}
\begin{split}
&\tilde{{\rm E}}\|\overline{L_k(\rho)}-\overline{\rho\ln\rho}\|_{L^\infty([0,T]; L^p(D))}\\
&\leq\sup_{t\in[0,T]}\liminf_{\delta\to 0+}\tilde{{\rm E}}\|
L_k(\rho_\delta)-\rho_\delta
\ln\rho_\delta\|_{L^\infty([0,T];L^p(D))}\\
&\leq 2q(k)\sup_{\delta}\sup_{t\in[0,T]}\max\left\{1,
\tilde{{\rm E}}\int_D \tilde{A}(\rho_\delta^p
|\ln\rho_\delta|^p)\,dx\right\}\\
&\leq 2q(k)\sup_{\delta}\sup_{t\in[0,T]}\max\left\{1,
2\tilde{{\rm E}}\int_D(1+\rho_\delta^p
|\ln\rho_\delta|^p)\ln(1+\rho_\delta^p
|\ln\rho_\delta|^p)\,dx\right\}\\
&\leq 2q(k)\sup_{\delta}\sup_{t\in[0,T]}\max\left\{1,
C(p)\textrm{meas}\{D\}+C(p)\tilde{{\rm E}}\int_{D\cap\{\rho_\delta\geq
k\}}\rho_\delta^p
|\ln\rho_\delta|^{p+1}\,dx\right\}\\
&\leq 2q(k)\sup_{\delta}\sup_{t\in[0,T]}\max\left\{1,
C(p)\textrm{meas}\{D\}+C(p, \gamma)\tilde{{\rm E}}\int_D\rho_\delta^\gamma
\,dx\right\}\\
&\leq Cq(k)\rightarrow 0, \; \textrm{ as } k\rightarrow\infty,
\end{split}
\end{equation*}
where $C$ is a constant independent of $\delta$, $\tilde{A}$ is
defined in \eqref{n1}, and
\begin{equation*}
q(k):=\tilde{{\rm E}}\|\chi_{[\rho_\delta\geq k]}\|_{L_N(D)}\leq
\left(A^{-1}\left(\frac{1}{r(k)}\right)\right)^{-1}.
\end{equation*}
Then
\begin{equation}\label{m59}
\overline{L_k(\rho)}\rightarrow \overline{\rho \ln\rho} \textrm{ in
} L^1(\tilde{\Omega},L^\infty([0,T]; L^p(D))) \textrm{ as }
k\rightarrow\infty\textrm{ for all } 1\leq p<\gamma.
\end{equation}
Similarly, we have
\begin{equation}\label{m60}
L_k(\rho)\rightarrow \rho \ln\rho\textrm{ in }
L^1(\tilde{\Omega},L^\infty([0,T]; L^p(D))) \textrm{ as }
k\rightarrow\infty, \; \textrm{ for all } 1\leq p<\gamma,
\end{equation}
and, by Lemma \ref{lemma5.3},
\begin{equation}\label{m61}
T_k(\rho)\rightarrow \overline{T_k(\rho)}\textrm{ in }
L^p(\tilde{\Omega},L^p([0,T]; L^p(D))) \textrm{ as }
k\rightarrow\infty, \; \textrm{ for all } 1\leq p<\gamma+1.
\end{equation}
Finally, by Lemma \ref{lemma5.2}, we obtain as \eqref{5.19a}
\begin{equation}\label{m62}
\begin{split}
&\int_0^t\int_D T_k(\rho)\Dv u \,dxdt-\lim_{\delta\rightarrow
0+}\int_0^t\int_D T_k(\rho_\delta)\Dv
u_\delta\,dxdt \\
&\leq \int_0^t\int_D(T_k(\rho)-\overline{T_k(\rho)})\Dv u\,dxdt.
\end{split}
\end{equation}}\fi
Combining \eqref{5.36}-\eqref{5.39}\if false(Combining
\eqref{m52}-\eqref{m62})\fi, we have
\begin{align*}
\tilde{{\rm
E}}\int_D\left(\overline{\rho\ln\rho}-\rho\ln\rho\right)(x,t)dx\le
0,
\end{align*}
this together with $\rho\ln \rho\le \overline{\rho\ln\rho}$ imply
that
\begin{align*}
\overline{\rho\ln\rho}(t)=\rho\ln\rho(t) \ \ \mbox{for all $
t\in[0,T]$ a.e..}
\end{align*}
From this, by \cite[Theorem 2.11]{f2}, we can infer that
$$\rho_\delta\to \rho \ \ \text{almost everywhere in }\tilde{\Omega}\times[0,T]\times D.$$
It follows from Proposition \ref{proposition4.1} and
\cite[Proposition 2.1]{f2} that
$$\rho_\delta\to \rho \ \ \text{weakly in $L^1(\tilde{\Omega}\times(0,T)\times D)$},$$
By \cite[Theorem 2.10]{f2}, for any $\eta>0$ and all $\delta>0$,
there exists $\sigma>0$ such that $\int_F
\rho_\delta(t,x)dxdtd\tilde{\textrm{P}}<\eta$ for any measurable set
$F\subset \tilde{\Omega}\times(0,T)\times D$ with
$\text{meas}\{F\}<\sigma$. On the other hand, by Egorov's Theorem,
there exists a measurable set $F_\sigma\subset
\tilde{\Omega}\times(0,T)\times D$ such that
$\text{meas}\{F_\sigma\}<\sigma$ and $\rho_\delta\to\rho$ uniformly
in $\tilde{\Omega}\times(0,T)\times D-F_\sigma$. Note that
\begin{align}\label{s-convergence}
\tilde{{\rm E}}\int_0^T \int_D |\rho_\delta-\rho|dxdt&\le
\iiint_{F_\sigma}|\rho_\delta-\rho|dxdtd\tilde{\textrm{P}}
+\iiint_{\tilde{\Omega}\times(0,T)\times D-F_\sigma}|\rho_\delta-\rho|dxdtd\tilde{\textrm{P}}\\
&\notag\le 2\eta+T|D|\sup_{\tilde{\Omega}\times(0,T)\times
D-F_\sigma}|\rho_\delta-\rho|.
\end{align}
Then \eqref{s-convergence} tends to 0 as $\delta\to 0$ and $\eta\to
0$. This implies the strong convergence of the sequence
$\rho_\delta$ in $L^1(\tilde{\Omega}\times(0,T)\times D)$, that is,
\begin{align}\label{5.24a}
\overline{\rho^\gamma}=\rho^\gamma.
\end{align}
Similar to the proof of \eqref{P-a.s.11a} and \eqref{P-a.s.12} in
Section \ref{s3.4}, by using \eqref{assumption}, \eqref{5.7b},
\eqref{5.8e1}, \eqref{5.24a} and H\"{o}lder's inequality,  we have
\begin{equation}\label{5.24}
\begin{split}
\langle f_k(\rho_\varepsilon,\rho_\varepsilon
u_\varepsilon,x),\phi\rangle\to\langle f_k(\rho,\rho
u,x),\phi\rangle\ \ \mbox{in} \ \ L^1([0,T]) \ \
\tilde{\textrm{P}}-\text{a.s.},\\
\langle g_k(B_\varepsilon,x),\phi\rangle\to\langle
g_k(B,x),\phi\rangle\ \ \mbox{in} \ \ L^1([0,T]) \ \
\tilde{\textrm{P}}-\text{a.s.},
\end{split}
\end{equation}
and
\begin{equation}\label{5.25}
\begin{split}
& \sum_{k\ge1}\langle f_k(\rho_\delta,\rho_\delta
u_\delta,x),\phi\rangle^2\to\sum_{k\ge1}\langle
f_k(\rho,\rho u,x),\phi\rangle^2\ \ \mbox{in} \ \ L^1([0,T]) \ \ \tilde{\textrm{P}}-\text{a.s.},\\
&\sum_{k\ge1}\langle
g_k(B_\delta,x),\phi\rangle^2\to\sum_{k\ge1}\langle
g(B,x),\phi\rangle^2\ \ \mbox{in} \ \ L^1([0,T]) \ \
\tilde{\textrm{P}}-\text{a.s.}.
\end{split}
\end{equation}
As in Section \ref{section4}, by using \eqref{6a}-\eqref{5.8e3}, \eqref{5.24a}-\eqref{5.25}, we deduce that $((\tilde{\Omega}, \tilde{\mathscr{F}},
\tilde{\textrm{P}}), \beta_k, \rho, u, B)$ satisfies the following
equations:
\begin{align*}
\begin{cases}
d\rho+\Dv(\rho u)dt=0, \\
(\rho u)_t\!+\!\Dv(\rho u\otimes u)+a\nabla
\rho^\gamma\!-\!\mu\Delta u\!-\!(\lambda+\mu)\nabla\Dv
u=(\nabla\times B)\times B+\sum_{k\ge1}
f_k(\rho,\rho u,x)\dot{\beta_k^1},\\
B_t-(\nabla\times u)\times B-\nu\Delta
B=\sum_{k\ge1}g_k(B,x)\dot{\beta_k^2},
\end{cases}
\end{align*}
in $\mathcal{D}^\prime((0,T)\times D)$ $\tilde{\textrm{P}}$-a.s..
Here we formally denote $\dot{W}:=dW/dt$. The proof of Theorem
\ref{theorem1.1} is thus completed.

\if false Since $\rho\ln\rho$ is a convex function, by the Egorov
Theorem and $\rho_\delta\in
L^{\gamma+\theta}(\Omega\times(0,T)\times D)$, one obtains
$\rho_\delta\to\rho$ strongly in $L^s(\Omega\times(0,T)\times D)$,
$1\le s<\gamma+\theta$. This implies the strong convergence of the
sequence $\rho_\delta$ in $L^1(\Omega\times(0,T)\times D)$, that is
$\overline{\rho^\gamma}=\rho^\gamma$. Similarly, it follows from
\eqref{P-a.s.2}, \eqref{P-a.s.4} and \eqref{P-a.s.7} that
\begin{align}\label{5.8e3}
\nabla\times(u_\delta\times B_\delta)\to\nabla\times (u\times B) \ \
\mbox{in $\mathcal{D}^\prime([0,T]\times D)$} \ \
\tilde{P}-\text{a.s.}.
\end{align}
The proof of Theorem \ref{theorem1.1} is completed.\fi

\smallskip

\section*{Acknowledgement}

H. Wang is supported by the National Natural Science Foundation of China (No.~11901066), the
Natural Science Foundation of Chongqing (No.~cstc2019jcyj-msxmX0167) and Project No.~2019CDXYST0015 and No.~2020CDJQY-A040
supported by the Fundamental Research Funds for the Central Universities.


\end{document}